\documentclass[11pt]{article}

\usepackage{amssymb,amsmath,amscd,amsthm,xy,enumitem,mathabx,mathrsfs,color,caption,stmaryrd}
\usepackage{amsxtra,pstricks,pst-node}
\xyoption{all}

\newtheorem{thm}{Theorem}[section]
\newtheorem{prp}[thm]{Proposition}
\newtheorem{lmm}[thm]{Lemma}

\newtheorem{crl}[thm]{Corollary}

\theoremstyle{definition}
\newtheorem{dfn}[thm]{Definition}
\newtheorem{eg}[thm]{Example}

\theoremstyle{remark}
\newtheorem{rmk}[thm]{Remark}

\numberwithin{equation}{section}

\def\lra{\longrightarrow}

\def\BE#1{\begin{equation}\label{#1}}
\def\EE{\end{equation}}
\def\lr#1{\langle#1\rangle}
\def\flr#1{\left\lfloor{#1}\right\rfloor}
\def\blr#1{\big\langle#1\big\rangle}
\def\ti#1{\tilde{#1}}
\def\wt#1{\widetilde{#1}}
\def\ov#1{\overline{#1}}
\def\eref#1{(\ref{#1})}
\def\tn#1{\textnormal{#1}}
\def\sf#1{\textsf{#1}}

\def\wh#1{\widehat{#1}}
\def\wch#1{\widecheck{#1}}

\def\De{\Delta}
\def\Ga{\Gamma}
\def\La{\Lambda}
\def\Om{\Omega}
\def\Si{\Sigma}
\def\Th{\Theta}

\def\al{\alpha}
\def\be{\beta}

\def\ga{\gamma}
\def\io{\iota}

\def\la{\lambda}
\def\na{\nabla}
\def\om{\omega}
\def\si{\sigma}
\def\th{\theta}
\def\ve{\varepsilon}
\def\vph{\varphi}

\def\ze{\zeta}

\def\fB{\mathfrak B}

\def\C{\mathbb C}
\def\cC{\mathcal C}  
\def\cD{\mathcal{D}}

\def\ne{\textnormal{e}}
\def\ff{\mathfrak f}

\def\cH{\mathcal H}

\def\bI{\mathbb I}
\def\fI{\mathfrak i}
\def\cJ{\mathcal J}
\def\fJ{\mathfrak j}

\def\cL{\mathcal L}

\def\cM{\mathcal M}

\def\fM{\mathfrak M}

\def\cN{\mathcal N}
\def\cO{\mathcal O}
\def\P{\mathbb P}
\def\fP{\mathfrak P}

\def\R{\mathbb{R}}

\def\fR{\mathfrak R}

\def\cT{\mathcal T}
\def\cU{\mathcal U}

\def\Z{\mathbb{Z}}

\def\fa{\mathfrak a}

\def\be{\mathbf e}
\def\fc{\mathfrak c}

\def\fs{\mathfrak s}
\def\s{\mathbf s}

\def\bt{\mathbf t}
\def\u{\mathbf u}
\def\x{\mathbf x}

\def\Aut{\tn{Aut}}

\def\tnd{\textnormal{d}}

\def\ev{\tn{ev}}
\def\GL{\tn{GL}}

\def\Hom{\tn{Hom}}
\def\id{\textnormal{id}}

\def\Id{\tn{Id}}

\def\Ind{\textnormal{Ind}\,}

\def\pt{\tn{pt}}

\def\Re{\tn{Re}}
\def\Res{\tn{Res}}
\def\rdet{\wh{\tn{det}}}
\def\rk{\textnormal{rk}}

\def\SO{\tn{SO}}
\def\tO{\tn{O}}

\def\SL{\tn{SL}}

\def\top{\textnormal{top}}
\def\vrt{\tn{vrt}}

\def\0{\mathbf 0}
\def\1{\mathbf 1}

\def\dbar{\bar\partial}
\def\prt{\partial}
\def\eset{\emptyset}
\def\i{\infty}
\def\bp{\bar\partial}

\def\bu{\bullet}

\begin{document}

\title{Real Gromov-Witten Theory in All Genera and\\ 
Real Enumerative Geometry: Properties}
\author{Penka Georgieva\thanks{Partially supported by ERC grant STEIN-259118} $~$and 
Aleksey Zinger\thanks{Partially supported by NSF grants DMS 0846978 and 1500875}}
\date{\today}
\maketitle

\begin{abstract}
\noindent
The first part of this work constructs positive-genus  real Gromov-Witten invariants
of real-orientable symplectic manifolds of odd ``complex" dimensions;
the present part focuses on their properties that are essential for actually working
with these invariants.
We determine the compatibility of the orientations on the moduli spaces of real maps
constructed in the first part with the standard node-identifying
immersion of Gromov-Witten theory. 
We also compare these orientations with alternative ways of orienting 
the moduli spaces of real maps that are available in special cases.
In a sequel, we use the properties established in this paper to 
compare real Gromov-Witten and enumerative invariants,
to describe equivariant localization data that computes the real Gromov-Witten invariants
of odd-dimensional projective spaces, and 
to establish vanishing results for these invariants
in the spirit of Walcher's predictions.
\end{abstract}

\tableofcontents

\section{Introduction}
\label{intro_sec}

\noindent
The theory of $J$-holomorphic maps plays prominent roles in symplectic topology,
algebraic geometry, and string theory.
The foundational work of~\cite{Gr,McSa94,RT,LT,FO} has 
established the theory of (closed) Gromov-Witten invariants,
i.e.~counts of $J$-holomorphic maps from closed Riemann surfaces to symplectic manifolds.
The two main obstacles to defining real Gromov-Witten invariants,
i.e.~counts of $J$-holomorphic maps from symmetric Riemann surfaces commuting 
with the involutions on the domain and the target,
are the potential non-orientability of the moduli space of real $J$-holomorphic maps and 
the existence of real codimension-one boundary strata.
These obstacles are overcome in many genus~0 situations in~\cite{Wel4,Wel6,Cho,Sol,Ge2,Teh};
see \cite[Section~1.3]{RealGWsI} for some comparisons.
In the first part of this work, we introduce the notion of \sf{real orientation} on 
a real symplectic $2n$-manifold $(X,\om,\phi)$ and overcome both obstacles
in {\it all} genera for \sf{real-orientable} symplectic manifolds
of odd ``complex" dimension~$n$.\\

\noindent
A real orientation on a real symplectic $2n$-manifold $(X,\om,\phi)$ with $n\!\not\in\!2\Z$
induces orientations on the moduli spaces of real $J$-holomorphic maps from arbitrary genus~$g$
symmetric surfaces to~$(X,\phi)$.
Theorems~\ref{ComplexOrient_thm} and~\ref{RelSpinOrient_thm}  compare these orientations
with the natural complex orientations and with the orientations induced 
by the corresponding spin and relative spin structures 
whenever the latter three make sense.
By Theorem~\ref{CompOrient_thm}, the orientations 
 on the moduli spaces of real $J$-holomorphic maps induced by a real orientation on $(X,\om,\phi)$
are ``anti-compatible" with the node-identifying immersion~\eref{iodfn_e2} 
which is central to much of ``classical" Gromov-Witten theory.
This theorem  is instrumental for any study of 
the structure of the real Gromov-Witten invariants that depends on 
a splitting property at a conjugate pair of nodes in the spirit of \cite[2.2.6]{KM}
and in particular for interpreting real Gromov-Witten theory in terms 
of integrable systems in the spirit of~\cite{DZ}.
A similar comparison of orientations in Lagrangian Floer theory
is key to establishing the renown $A_{\i}$-relations of~\cite{FOOO}.
Theorems~\ref{CompOrient_thm}, \ref{ComplexOrient_thm}, and~\ref{RelSpinOrient_thm} 
are likewise essential for studying the properties of real GW-invariants constructed in~\cite{RealGWsI}.
For example, they play crucial roles in determining the normal bundles to
the torus-fixed loci in~\cite{RealGWsIII} and the contributions from 
the degenerate loci in~\cite{NZ}.

\subsection{Real-orientable symplectic manifolds}
\label{RealGWth_subs}

\noindent
An \textsf{involution} on a topological space~$X$ is a homeomorphism
$\phi\!:X\!\lra\!X$ such that $\phi\!\circ\!\phi\!=\!\id_X$.
By an \textsf{involution on a manifold}, we will mean a smooth involution. 
Let
$$X^\phi=\big\{x\!\in\!X\!:~\phi(x)\!=\!x\big\}$$
denote the fixed locus.
An \sf{anti-symplectic involution~$\phi$} on a symplectic manifold $(X,\om)$
is an involution $\phi\!:X\!\lra\!X$ such that $\phi^*\om\!=\!-\om$.
A \sf{real symplectic manifold} is a triple $(X,\om,\phi)$ consisting 
of a symplectic manifold~$(X,\om)$ and an anti-symplectic involution~$\phi$.\\

\noindent
Let $(X,\phi)$ be a topological space with an involution.
A \sf{conjugation} on a complex vector bundle $V\!\lra\!X$ 
\sf{lifting} an involution~$\phi$ is a vector bundle homomorphism 
$\vph\!:V\!\lra\!V$ covering~$\phi$ (or equivalently 
a vector bundle homomorphism  $\vph\!:V\!\lra\!\phi^*V$ covering~$\id_X$)
such that the restriction of~$\vph$ to each fiber is anti-complex linear
and $\vph\!\circ\!\vph\!=\!\id_V$.
A \sf{real bundle pair} $(V,\vph)\!\lra\!(X,\phi)$   
consists of a complex vector bundle $V\!\lra\!X$ and 
a conjugation~$\vph$ on $V$ lifting~$\phi$.
For example, 
$$(X\!\times\!\C,\phi\!\times\!\fc)\lra(X,\phi),$$
where $\fc\!:\C^n\!\lra\!\C^n$ is the standard conjugation on~$\C^n$,
is a real bundle pair.
If $X$ is a smooth manifold, then $(TX,\tnd\phi)$ is also a real bundle pair over~$(X,\phi)$.
For any real bundle pair $(V,\vph)\!\lra\!(X,\phi)$, 
we denote~by
$$\La_{\C}^{\top}(V,\vph)=(\La_{\C}^{\top}V,\La_{\C}^{\top}\vph)$$
the top exterior power of $V$ over $\C$ with the induced conjugation.
Direct sums, duals, and tensor products over~$\C$ of real bundle pairs over~$(X,\phi)$
are again real bundle pairs over~$(X,\phi)$.

\begin{dfn}[{\cite[Definition~5.1]{RealGWsI}}]\label{realorient_dfn4}
Let $(X,\phi)$ be a topological space with an involution and 
$(V,\vph)$ be a real bundle pair over~$(X,\phi)$.
A \sf{real orientation} on~$(V,\vph)$ consists~of 
\begin{enumerate}[label=(RO\arabic*),leftmargin=*]

\item\label{LBP_it2} a rank~1 real bundle pair $(L,\wt\phi)$ over $(X,\phi)$ such that 
\BE{realorient_e4}
w_2(V^{\vph})=w_1(L^{\wt\phi})^2 \qquad\hbox{and}\qquad
\La_{\C}^{\top}(V,\vph)\approx(L,\wt\phi)^{\otimes 2},\EE

\item\label{isom_it2} a homotopy class of isomorphisms of real bundle pairs in~\eref{realorient_e4}, and

\item\label{spin_it2} a spin structure~on the real vector bundle
$V^{\vph}\!\oplus\!2(L^*)^{\wt\phi^*}$ over~$X^{\phi}$
compatible with the orientation induced by~\ref{isom_it2}.\\ 

\end{enumerate}
\end{dfn}

\noindent
An isomorphism in~\eref{realorient_e4} restricts to an isomorphism 
\BE{realorient2_e3}\La_{\R}^{\top}V^{\vph}\approx (L^{\wt\phi})^{\otimes2}\EE
of real line bundles over~$X^{\phi}$.
Since the vector bundles $(L^{\wt\phi})^{\otimes2}$ and $2(L^*)^{\wt\phi^*}$ are canonically oriented, 
\ref{isom_it2} determines orientations on $V^{\vph}$ and $V^{\vph}\!\oplus\! 2(L^*)^{\wt\phi^*}$.
By the first assumption in~\eref{realorient_e4}, the real vector bundle
$V^{\vph}\!\oplus\!2(L^*)^{\wt\phi^*}$ over~$X^{\phi}$ admits a spin structure.\\

\noindent
Let $(X,\om,\phi)$ be a real symplectic manifold.
A \sf{real orientation on~$(X,\om,\phi)$} is a real orientation on the real bundle pair $(TX,\tnd\phi)$.
We call $(X,\om,\phi)$ \sf{real-orientable} if it admits a real orientation.

\subsection{Compatibility with node-identifying immersion}
\label{VanCompPrp_subs}

\noindent
A \sf{symmetric surface} $(\Si,\si)$ is a closed  oriented  
surface~$\Si$ (manifold of real dimension~2) with an orientation-reversing involution~$\si$.
The fixed locus of~$\si$ is a disjoint union of circles.
If in addition $(X,\phi)$ is a manifold with an involution, 
a \sf{real map} 
$$u\!:(\Si,\si)\lra(X,\phi)$$ 
is a smooth map $u\!:\Si\!\lra\!X$ such that $u\!\circ\!\si=\phi\!\circ\!u$.
We denote the space of such maps by~$\fB_g(X)^{\phi,\si}$.
The main focus of~\cite{RealGWsI} is on smooth and one-nodal connected symmetric surfaces,
but in the present paper we also need to consider disconnected and two-nodal symmetric surfaces.
Throughout this paper,  the term \sf{symmetric surface} will thus refer to smooth connected 
surfaces unless explicitly stated otherwise.\\

\noindent
For a symplectic manifold $(X,\om)$, we denote~by $\cJ_{\om}$
the space of $\om$-compatible almost complex structures on~$X$.
If $\phi$ is  an anti-symplectic involution on~$(X,\om)$, let 
$$\cJ_{\om}^{\phi}=\big\{J\!\in\!\cJ_{\om}\!:\,\phi^*J\!=\!-J\big\}.$$
For a genus~$g$ symmetric surface~$(\Si,\si)$, possibly nodal and disconnected,
we similarly denote by $\cJ_{\Si}^{\si}$
the space of complex structures~$\fJ$ on~$\Si$ compatible with the orientation such that 
$\si^*\fJ\!=\!-\fJ$.
For $J\!\in\!\cJ_{\om}^{\phi}$, $\fJ\!\in\!\cJ_{\Si}^{\si}$, and
$u\!\in\!\fB_g(X)^{\phi,\si}$, let 
$$\dbar_{J,\fJ}u=\frac{1}{2}\big(\tnd u+J\circ\tnd u\!\circ\!\fJ\big)$$
be the \textsf{$\dbar_{J,\fJ}$-operator} on~$\fB_g(X)^{\phi,\si}$.\\

\noindent 
Let $(X,\om,\phi)$ be a real-orientable symplectic $2n$-manifold with $n\!\not\in\!2\Z$,
$g,l\!\in\!\Z^{\ge0}$, $B\!\in\!H_2(X;\Z)$, and $J\!\in\!\cJ_{\om}^{\phi}$.
We denote by $\ov\fM_{g,l}(X,B;J)^{\phi}$ 
the  moduli space of equivalence classes of stable real degree~$B$  $J$-holomorphic maps
from genus~$g$ symmetric (possibly nodal) surfaces with $l$ pairs of conjugate marked points. 
By \cite[Theorem~1.4]{RealGWsI}, a real orientation on~$(X,\om,\phi)$ 
determines an orientation on this compact space, 
endows it with a virtual fundamental class, and  
thus gives rise to  genus~$g$ real GW-invariants of $(X,\om,\phi)$
that are independent of the choice of~$J\!\in\!\cJ_{\om}^{\phi}$.\\

\noindent
We denote~by  $\ov\fM_{g,l}^{\bu}(X,B;J)^{\phi}$
the moduli space of stable real degree~$B$ morphisms from possibly disconnected 
nodal symmetric surfaces of holomorphic Euler characteristic $1\!-\!g$ with $l$ pairs 
of conjugate marked points.
For each $i\!=\!1,\ldots,l$, let 
$$\ev_i\!: \ov\fM_{g,l}^{\bu}(X,B;J)^{\phi}\lra X, \qquad
\big[u,(z_1^+,z_1^-),\ldots,(z_l^+,z_l^-)\big]\lra u(z_i^+),$$
be the evaluation at the first point in the $i$-th pair of conjugate points.
If $l\!\ge\!2$, let 
$$\ov\fM_{g,l}'^{\bu}(X,B;J)^{\phi}
=\big\{[\u]\!\in\!\ov\fM_{g,l}^{\bu}(X,B;J)^{\phi}\!:\,
\ev_{l-1}([\u])\!=\!\ev_l([\u])\big\}.$$
The short exact sequence
$$0\lra T\ov\fM_{g,l}'^{\bu}(X,B;J)^{\phi}\lra 
T\ov\fM_{g,l}^{\bu}(X,B;J)^{\phi}|_{\ov\fM_{g,l}'^{\bu}(X,B;J)^{\phi}} \lra
\ev^*TX\lra 0$$
induces an isomorphism
\BE{SubIsom_e}\begin{split}
&\La_{\R}^{\top}\big(T\ov\fM_{g,l}^{\bu}(X,B;J)^{\phi}|_{\ov\fM_{g,l}'^{\bu}(X,B;J)^{\phi}}\big)
\approx 
\La_{\R}^{\top}\big(T\ov\fM_{g,l}'^{\bu}(X,B;J)^{\phi}\big)
\otimes \ev_l^*\big(\La_{\R}^{\top}(TX)\big)
\end{split}\EE
of real line bundles over $\ov\fM_{g,l}'^{\bu}(X,B;J)^{\phi}$.\\

\noindent
The identification of the last two pairs of conjugate marked points induces
an immersion 
\BE{iodfn_e2}\io\!: 
\ov\fM_{g-2,l+2}'^{\bu}(X,B;J)^{\phi}\lra  \ov\fM_{g,l}^{\bu}(X,B;J)^{\phi}\,.\EE
This immersion takes the main stratum of the domain,
i.e.~the subspace consisting of real morphisms from smooth symmetric surfaces, 
to the subspace of the target  consisting
of real morphisms from symmetric surfaces with one pair of conjugate nodes.
There is a canonical isomorphism
$$\cN\io\equiv \frac{\io^*T\ov\fM_{g,l}^{\bu}(X,B;J)^{\phi}}{T\ov\fM_{g-2,l+2}'^{\bu}(X,B;J)^{\phi}}
\approx \cL_{l+1}\!\otimes_{\C}\!\cL_{l+2}$$
of the normal bundle of~$\io$ with the tensor product of the universal tangent line bundles
for the first points in the last two conjugate pairs.
It induces an isomorphism
\BE{RestrOrient_e0}\begin{split} 
\io^*\big(\La_{\R}^{\top}\big(T\ov\fM_{g,l}^{\bu}(X,B;J)^{\phi}\big)\big)
\approx \La_{\R}^{\top}\big(T\ov\fM_{g-2,l+2}'^{\bu}(X,B;J)^{\phi}\big)
\otimes \La_{\R}^2\big(\cL_{l+1}\!\otimes_{\C}\!\cL_{l+2}\big)
\end{split}\EE
of real line bundles over $\ov\fM_{g-2,l+2}'^{\bu}(X,B;J)^{\phi}$.
Along with~\eref{SubIsom_e} with $(g,l)$ replaced by~$(g\!-\!2,l\!+\!2)$, 
it determines an isomorphism
\BE{CompOrient_e}\begin{split}   
&\La_{\R}^{\top}\big(T\ov\fM_{g-2,l+2}^{\bu}(X,B;J)^{\phi}|_{\ov\fM_{g-2,l+2}'^{\bu}(X,B;J)^{\phi}}\big)
\otimes \La_{\R}^2\big(\cL_{l+1}\!\otimes_{\C}\!\cL_{l+2}\big) \\
&\hspace{1.5in} \approx\io^*\big(\La_{\R}^{\top}\big(T\ov\fM_{g,l}^{\bu}(X,B;J)^{\phi}\big)\big)
\otimes  \ev_{l+1}^*\big(\La_{\R}^{\top}(TX)\big)
\end{split}\EE
of real line bundles over $\ov\fM_{g-2,l+2}'^{\bu}(X,B;J)^{\phi}$.

\begin{thm}\label{CompOrient_thm}
Let $(X,\om,\phi)$ be a real-orientable $2n$-manifold with $n\!\not\in\!2\Z$, $g,l\!\in\!\Z^{\ge0}$, 
\hbox{$B\!\in\!H_2(X;\Z)$}, and $J\!\in\!\cJ_{\om}^{\phi}$. 
The isomorphism~\eref{CompOrient_e} is orientation-reversing with respect 
to the orientations on the moduli spaces determined by  a real orientation on~$(X,\om,\phi)$
and  the canonical  orientations on  $\cL_{l+1}\!\otimes_{\C}\!\cL_{l+2}$
and~$TX$.
\end{thm}

\noindent
The substance of this statement is that the orientations on 
$\ov\fM_{g-2,l+2}'^{\bu}(X,B,J)^{\phi}$
induced from the orientations of $\ov\fM_{g-2,l+2}^{\bu}(X,B,J)^{\phi}$ 
and $\ov\fM_{g,l}^{\bu}(X,B,J)^{\phi}$  via the isomorphisms~\eref{SubIsom_e} 
and~\eref{RestrOrient_e0} are opposite.
This may seem surprising from the point of view of the classical (closed) GW-theory,
where moduli spaces have canonical orientations and signs do not appear.
On the other hand, systematic orientations of moduli spaces in open and real GW-theories
depend on orienting conventions and additional topological data
(such as a spin structure or a real orientation) on the target manifold.
The appearance of signs is then a fairly common occurrence, and
it is notoriously difficult to determine them correctly in most cases.
The orientation conventions in~\cite{RealGWsI} are natural from mathematical considerations
and conform with the mirror symmetry expectations described in~\cite{Wal};
see \cite[Section~3]{RealGWsI} and \cite[Section~1.3]{RealGWsIII}, respectively,
for details.
While it is possible to adjust the orientations on moduli spaces of real maps
to make the isomorphism~\eref{CompOrient_e} orientation-preserving,
this would be artificial from the geometric standpoint and 
undesirable based on mirror symmetry considerations.\\

\noindent
We note that the statement of Theorem~\ref{CompOrient_thm} is invariant under 
interchanging the points within the last two conjugate pairs simultaneously
(this corresponds to reordering the nodes of a nodal map). 
This interchange reverses the orientation of the last factor
 on the left-hand side of~\eref{CompOrient_e},
because the complex rank of $\cL_{l+1}\!\otimes_{\C}\!\cL_{l+2}$ is~1, 
and the orientation of the last factor on the right-hand side of~\eref{CompOrient_e},
because the complex rank of $TX$ is~odd.\\

\noindent
An analogue of Theorem~\ref{CompOrient_thm} with $n\!\in\!2\Z$ is 
described in Remark~\ref{twist_rmk2}.

\subsection{Comparison with complex orientation}
\label{ComplexOrient_subs0}

\noindent
Let $g_0\!\in\!\Z^{\ge0}$.
We define a \sf{$g_0$-doublet}  to be
a two-component smooth symmetric surface $(\Si,\si)$ of the~form
\BE{SymSurfDbl_e}\Si\equiv \Si_1\!\sqcup\!\Si_2
\equiv \{1\}\!\times\!\Si_0 \sqcup \{2\}\!\times\!\ov\Si_0, \qquad
\si(i,z)=\big(3\!-\!i,z\big)~~\forall~(i,z)\!\in\!\Si,\EE
where $\Si_0$ is a connected smooth oriented genus~$g_0$ surface 
and 
$\ov\Si_0$ denotes $\Si_0$ with the opposite orientation.
The holomorphic Euler characteristic of a $g_0$-doublet is $1\!-\!g$
with $g\!=\!2g_0\!-\!1$.\\

\noindent
Suppose $(X,\om,\phi)$ is a real-orientable $2n$-manifold, $l\!\in\!\Z^{\ge0}$, 
\hbox{$B\!\in\!H_2(X;\Z)$}, and $J\!\in\!\cJ_{\om}^{\phi}$.
With  $(\Si,\si)$  as in~\eref{SymSurfDbl_e}, let
$$\fM_{2g_0-1,l}^{\bu}(X,B;J)^{\phi,\si}\subset \ov\fM_{2g_0-1,l}^{\bu}(X,B;J)^{\phi}$$
denote the open subspace of real $J$-holomorphic maps from~$(\Si,\si)$.
For each $\fs\!\subset\!\{1,\ldots,l\}$, let
$$\fM_{2g_0-1,l}^{\bu}(X,B;J)_{\fs}^{\phi,\si}\subset\fM_{2g_0-1,l}^{\bu}(X,B;J)^{\phi,\si}$$
be the open subspace consisting of marked maps so that  the second point
in the $i$-th conjugate pair lies on~$\Si_1$ if and only~if $i\!\in\!\fs$.
In particular,
\BE{ComplexOrient_e0}\fM_{2g_0-1,l}^{\bu}(X,B;J)_{\fs}^{\phi,\si} \subset
\bigsqcup_{\begin{subarray}{c}B_0\in H_2(X;\Z)\\ B_0-\phi_*B_0=B\end{subarray}}
\hspace{-.25in}\big(
\fM_{g_0,l}(X,B_0;J)\!\times\!\fM_{g_0,l}(X,-\phi_*B_0;J)\big)\,,\EE
where $\fM_{g_0,l}(X,B_0;J)$ is the usual moduli space of degree~$B_0$
$J$-holomorphic maps from smooth genus $g_0$~curves with $l$~marked points.
The projection
\BE{ComplexOrient_e1}
\fM_{2g_0-1,l}^{\bu}(X,B;J)_{\fs}^{\phi,\si}\lra 
\bigsqcup_{\begin{subarray}{c}B_0\in H_2(X;\Z)\\ B_0-\phi_*B_0=B\end{subarray}}
\hspace{-.27in}\fM_{g_0,l}(X,B_0;J)\EE
to the first factor in~\eref{ComplexOrient_e0} is an isomorphism 
(in the sense of Kuranishi structures, 
i.e.~\eref{ComplexOrient_e1} identifies the deformation-obstruction theories
of the two moduli spaces).
The moduli space on the right-hand side of~\eref{ComplexOrient_e1} carries a natural orientation
obtained by homotoping the linearization of the $\dbar$-operator to a $\C$-linear Fredholm operator;
see \cite[Section~3.2]{MS}.
We will call the orientation on the left-hand side of~\eref{ComplexOrient_e1}
induced by this orientation 
\sf{the complex orientation} of $\fM_{2g_0-1,l}^{\bu}(X,B;J)_{\fs}^{\phi,\si}$.

\begin{thm}\label{ComplexOrient_thm}
Suppose $(X,\om,\phi)$ is a real-orientable $2n$-manifold with $n\!\not\in\!2\Z$, 
$g_0,l\!\in\!\Z^{\ge0}$,  $(\Si,\si)$ is a $g_0$-doublet ,
\hbox{$B\!\in\!H_2(X;\Z)$}, and $J\!\in\!\cJ_{\om}^{\phi}$. 
The orientation on $\fM_{2g_0-1,l}^{\bu}(X,B;J)_{\fs}^{\phi,\si}$ induced by 
a real orientation on~$(X,\om,\phi)$ and its complex orientation differ 
by $(-1)^{g_0+1+|\fs|}$.
\end{thm}

\noindent
Since the orientation on $\fM_{g,l}^{\bu}(X,B;J)^{\phi}$ induced by 
a real orientation on~$(X,\om,\phi)$ is compatible
with orienting the fibers of the forgetful morphisms
\BE{ffmarked_e} \ov\fM_{g,l+1}^{\bu}(X,B;J)^{\phi} \lra \ov\fM_{g,l}^{\bu}(X,B;J)^{\phi}\EE
by the first marked point in the last conjugate pair, 
the statement of this theorem is compatible with the forgetful
morphisms.
Under the assumptions of this theorem, the ``complex" dimension 
of the right-hand side of~\eref{ComplexOrient_e1} in 
the $l\!=\!0$ case,~i.e.
$$\dim_{\C}^{\vrt} \fM_{g_0,0}(X,B_0;J)=\blr{c_1(TX),B_0}+(n\!-\!3)(1\!-\!g_0),$$
is even by the second condition in~\eref{realorient_e4}. 
Thus, the ``conjugation" diffeomorphism
$$\bigsqcup_{\begin{subarray}{c}B_0\in H_2(X;\Z)\\ B_0-\phi_*B_0=B\end{subarray}}
\hspace{-.27in}\fM_{g_0,0}(X,B_0;J)
\lra 
\bigsqcup_{\begin{subarray}{c}B_0\in H_2(X;\Z)\\ B_0-\phi_*B_0=B\end{subarray}}
\hspace{-.27in}\fM_{g_0,0}(X,B_0;J), 
\quad
[u,\fJ]\lra \big[\phi\!\circ\!u\!,-\fJ\big],$$
is orientation-preserving.
This implies that the validity of Theorem~\ref{ComplexOrient_thm} is independent
of the ordering of the topological components of~$\Si$.\\

\noindent
An illustration of Theorems~\ref{CompOrient_thm} and~\ref{ComplexOrient_thm} 
in the genus~0 case is \cite[Lemma~5.2]{RealEnum}.
It describes the normal bundle to a stratum of genus~0 maps consisting of 
a central component with a pair of conjugate bubbles, i.e.~a 0-doublet , attached.
This boundary stratum is oriented by choosing one of the nodes and
taking the complex orientation associated with the corresponding bubble.
The claim of  \cite[Lemma~5.2]{RealEnum} is that the 
normal bundle is then oriented by the complex orientation of the smoothings of this node.
According to Theorem~\ref{ComplexOrient_thm}, the ``canonical" orientation 
of this boundary stratum is obtained by taking the opposite of the complex orientation
on the distinguished bubble.
According to Theorem~\ref{CompOrient_thm},  the orientation of the normal bundle is then
 opposite to the complex orientation of the smoothings of the distinguished node.
Thus, \cite[Lemma~5.2]{RealEnum} is a consequence of 
Theorems~\ref{CompOrient_thm} and~\ref{ComplexOrient_thm}.

\subsection{Comparison with spin and relative spin orientations}
\label{RelSpinOrient_subs0}

\noindent
Let $X$ be a topological space, $Y\!\subset\!X$ be a subspace, and $F\!\lra\!Y$
be a real oriented vector bundle.
A \sf{relative spin structure on~$F$} as in~\cite{FOOO} consists of a real oriented vector bundle
$E\!\lra\!X$ and a spin structure on~$F\!\oplus\!E|_Y$.
If $(X,\phi)$ is a topological space with an involution and $(L,\wt\phi)$ is a real 
bundle pair over~$(X,\phi)$, the~map
\BE{splitLisom_e} 2(L^*)^{\wt\phi^*}\lra L^*|_{X^{\phi}}, \qquad (v,w)\lra v+\fI w,\EE
is an isomorphism of real oriented vector bundles over~$X^{\phi}$.
Thus, a real orientation on a real bundle pair $(V,\vph)$ as in Definition~\ref{realorient_dfn4}
determines a relative spin structure on the real oriented vector bundle 
$V^{\vph}\!\lra\!X^{\phi}$ with $E\!=\!L^*$ in the above notation;
we will call this structure \sf{the associated relative spin structure on~$V^{\vph}$}.
If in addition $L^{\wt\phi}\!\lra\!X^{\phi}$ is orientable, 
$2(L^*)^{\wt\phi^*}$ has a canonical homotopy
class of trivializations as in the proof of \cite[Corollary~5.6]{RealGWsI}.
Such a real orientation on~$(V,\vph)$ thus determines a spin structure on~$V^{\vph}$;
we will call the latter \sf{the associated spin structure on~$V^{\vph}$}.\\

\noindent
Let $\tau$ be the standard involution on~$\P^1$;  
we take it to be given by $z\!\lra\!1/\bar{z}$ on~$\C$.
For $l\!\ge\!2$, we denote by $\cM_{0,l}^{\tau}$ the uncompactified moduli space of 
equivalence classes of $(\P^1,\tau)$ with $l$~pairs of conjugate marked points. 
The Deligne-Mumford compactification~$\ov\cM_{0,2}^{\tau}$ of $\cM_{0,2}^{\tau}$
includes 3 additional stable real two-component nodal curves.
A diffeomorphism of~$\ov\cM_{0,2}^{\tau}$  with a closed interval is given~by
\BE{cMtaudiff_e}\ov\cM_{0,2}^{\tau}\lra \ov\R^+\!\equiv\![0,\i] , \quad
\big[(z_1^+,z_1^-),(z_2^+,z_2^-)\big]\lra 
\frac{z_2^+\!-\!z_1^+}{z_2^-\!-\!z_1^+}:\frac{z_2^+\!-\!z_1^-}{z_2^-\!-\!z_1^-}
=\frac{|z_1^+\!-\!z_2^+|^2}{|1\!-\!z_1^+/z_2^-|^2}\,.\EE
It takes the two-component curve with $z_1^+$ and $z_2^+$ on the same component to~$0$
and the two-component curve with $z_1^+$ and $z_2^-$ on the same component to~$\i$.
For $l\!\ge\!2$, the fibers of the forgetful morphism
$$\ov\cM_{0,l+1}^{\tau}\lra\ov\cM_{0,l}^{\tau}$$
are oriented by the canonical complex orientation of the tangent space at 
the first marked point in the last conjugate pair.
It follows that the moduli space $\ov\cM_{0,l}^{\tau}$ is orientable.\\

\noindent
Let $(X,\om,\phi)$ be a real symplectic manifold.
By \cite[Theorem~1.3]{RealGWsI}, a real orientation on $(X,\om,\phi)$ and an orientation 
on~$\ov\cM_{0,2}^{\tau}$ determine an orientation on each moduli space 
$\fM_{0,l}(X,B;J)^{\phi,\tau}$
of real $J$-holomorphic maps from~$(\P^1,\tau)$ to~$(X,\phi)$.
The standard approach \cite{Sol,Cho,FOOO9}
 to orienting $\fM_{0,l}(X,B;J)^{\phi,\tau}$ involves orienting 
the associated moduli space of disk maps from a relative spin structure on \hbox{$TX^{\phi}\!\lra\!X^{\phi}$};
in some cases, the resulting orientation on the disk space descends to an orientation 
on $\fM_{0,l}(X,B;J)^{\phi,\tau}$.
Theorem~\ref{RelSpinOrient_thm} below compares the orientations on $\fM_{0,l}(X,B;J)^{\phi,\tau}$
resulting from the two approaches to orienting~it.
Both approaches involve some sign conventions, which we specify~next.\\

\noindent
The construction of the orientation on the real line bundle~\eref{CidentDM_e}
in the proof of \cite[Proposition~5.9]{RealGWsI} involves 
a somewhat arbitrary sign choice for the Serre duality isomorphism
\cite[(5.21)]{RealGWsI}.
The (real) dimensions of its domain and target are $3(g\!-\!1)\!+\!2l$.
Thus, this  choice has no effect on the homotopy class of this isomorphism
or the resulting orientation of the real line bundle~\eref{CidentDM_e} if $g\!\not\in\!2\Z$.
If $g\!\in\!2\Z$, changing this choice changes the resulting orientation of~\eref{CidentDM_e}
and  the orientation on the moduli spaces $\fM_{g,l}(X,B;J)^{\phi,\si}$
of real~maps.
In light of Proposition~\ref{CompOrient_prp}, the above sign choice is determined 
by a choice of orientation
of the real line bundle~\eref{CidentDM_e}  over~$\ov\cM_{0,2}^{\tau}$.
In this case, the operator~$\dbar_{\C}$ is surjective and its kernel consists
of constant $\R$-valued functions.
Thus, an orientation on~\eref{CidentDM_e} over  $\ov\cM_{0,2}^{\tau}$ 
is determined by an orientation on~$\ov\cM_{0,2}^{\tau}$.
As in \cite[Section~3]{RealEnum}, we orient~$\ov\cM_{0,2}^{\tau}$ by 
the diffeomorphism~\eref{cMtaudiff_e}.\\

\noindent
Let $G_{\tau}$ denote the group of holomorphic automorphisms of $(\P^1,\tau)$.
The exact sequence 
$$0\lra T_{\id}S^1\lra T_{\id}G_{\tau}\lra T_0\C\lra0$$
and the standard orientations of $S^1$ and $\C$ determine an orientation on~$G_{\tau}$.
Let $\fP_0(X,B;J)$ denote the space of (parametrized)  degree~$B$ $J$-holomorphic 
real maps from $(\P^1,\tau)$ to~$(X,\phi)$; thus,
\BE{fMcP_e}  \fM_{0,0}(X,B;J)^{\phi,\tau}= \fP_0(X,B;J)\big/G_{\tau}\,. \EE
An orientation on the left-hand side of~\eref{fMcP_e} 
determines an orientation on~$\fP_0(X,B;J)$ via the canonical isomorphism
\BE{fMcP_e2} 
\La_{\R}^{\top}\big(T_u\fP_0(X,B;J)\big)
\approx \La_{\R}^{\top}\big(T_{[u]}\fM_{0,0}(X,B;J)^{\phi,\tau}\big)\otimes
\La_{\R}^{\top}\big(T_{\id}G_{\tau}\big).\EE
An orientation on the marked moduli spaces $\fM_{0,l}(X,B;J)^{\tau,\phi}$
is then determined by  orienting the fibers of the forgetful morphisms~\eref{ffmarked_e}
by the first marked point in the last conjugate pair.
Since $G_{\tau}$ has two topological components, an orientation on $\fP_0(X,B;J)$
may not descend to the quotient~\eref{fMcP_e2}.
By \cite[Theorem~6.6]{Ge2} with $(E,\wt\tau)\!=\!(L,\wt\phi)^*$,
a real orientation on~$(X,\om,\phi)$ induces an orientation on $\fP_0(X,B;J)$
that descends to this quotient and extends to the stable map compactification.\\

\noindent
The (virtual) tangent space of $\fP_0(X,B;J)$ is the index (as a K-theory class)
of the linearization of the $\dbar_J$-operator at~$u$.
An orientation on this index, or equivalently on $\det D_{(TX,\tnd\phi)}|_u$,
is determined by a relative spin structure on $TX^{\phi}\!\lra\!X^{\phi}$;
see the proof of \cite[Theorem~8.1.1]{FOOO} or \cite[Theorem~6.36]{Melissa}.
If this orientation descends to the quotient~\eref{fMcP_e},
the   induced orientation on the latter depends on
the ordering of the two lines on the right-hand side of~\eref{fMcP_e2}
if 
$$\dim_{\R}^{\vrt} \fM_{0,0}(X,B;J)^{\phi,\tau}=\blr{c_1(TX),B}+n\!-\!3,$$
is odd. 
If $(X,\om,\phi)$ is real-orientable, this is the case if and only if~$n\!\in\!2\Z$.\\

\noindent
The marked moduli space $\fM_{0,l}(X,B;J)^{\phi,\tau}$ can also be oriented by first 
orienting the marked parametrized space $\fP_l(X,B;J)$  from 
the orientation of $\fP_0(X,B;J)$
via the forgetful morphism as in~\eref{ffmarked_e} and then taking 
the quotient as in~\eref{fMcP_e}.
If $l\!\ge\!2$, we can then take $(X,B)\!=\!(\pt,0)$ and obtain an orientation~on
$$\cM_{0,2}^{\tau}=\fM_{0,2}(\pt,0)^{\id,\tau}\,.$$
With the orienting convention~\eref{fMcP_e2}, this orientation agrees 
with the orientation on $\ov\cM_{0,2}^{\tau}$ determined by
the diffeomorphism~\eref{cMtaudiff_e}.

\begin{thm}\label{RelSpinOrient_thm}
Suppose $(X,\om,\phi)$ is a real-orientable manifold,  $l\!\in\!\Z^{\ge0}$, 
\hbox{$B\!\in\!H_2(X;\Z)$}, and $J\!\in\!\cJ_{\om}^{\phi}$. 
The orientations on $\fM_{0,l}(X,B;J)^{\phi,\tau}$ induced by 
a real orientation on~$(X,\om,\phi)$ as in Definition~\ref{realorient_dfn4} and 
by the associated relative spin  structure on $TX^{\phi}\!\lra\!X^{\phi}$ 
differ  by $(-1)^{\ve(B)}$, where
$$\ve(B)\equiv\flr{\frac{\lr{c_1(X),B}+2}{4}} 
\,.$$
If in addition  $L^{\wt\phi}\!\lra\!X^{\phi}$ is orientable, then
the orientations on $\fM_{0,l}(X,B;J)^{\phi,\tau}$ induced by 
the real orientation on~$(X,\om,\phi)$ and 
by the associated spin structure on~$TX^{\phi}$ are the~same.
\end{thm}

\noindent
A key step in the proof of this theorem in Section~\ref{RelSpinOrient_subs}
is Proposition~\ref{RelSpinOrient_prp};
it obtains an explicit comparison of orientations of determinants of Fredholm operators.
This comparison is in the spirit of the undetermined sign of \cite[Proposition~8.4]{Sol}.
As indicated in Section~\ref{OrientApply_subs} and illustrated in~\cite{RealGWsIII}, 
Proposition~\ref{RelSpinOrient_prp} makes it possible to determine
the equivariant weights of vector bundles along torus fixed loci
in settings such as in \cite[Section~5]{KatzLiu}, \cite[Section~4]{PSW},
and \cite[Section~6.4]{Teh}.
We in fact give three proofs of Proposition~\ref{RelSpinOrient_prp},
a direct computation and as a consequence of the equivariant computations
in~\cite{Teh}.

\begin{rmk}\label{Ge2_rmk}
The approach to orienting  the moduli spaces of real maps from $(\P^1,\tau)$ to $(X,\phi)$
by ``stabilizing''  the real bundle pair $(TX,\tnd\phi)$
with two copies of a real bundle pair $(E,\wt\tau)$ over~$(X,\phi)$
is introduced in~\cite{Ge2}. 
For these moduli spaces, the orienting procedure of \cite[Theorem~1.3]{RealGWsI} 
specializes to the orienting procedure of~\cite{Ge2}.
While the stabilizing real bundle pair~$(E,\wt\tau)$ in~\cite{Ge2} can be of any rank,
the purpose of~$(E,\wt\tau)$ is also fulfilled by $\La_{\C}^{\top}(E,\wt\tau)$
and so it is sufficient to restrict to the rank~1 real bundle pairs.
On the other hand, the proof of Theorem~\ref{RelSpinOrient_thm}
 readily extends to  real bundle pairs~$(L,\wt\phi)$ of any~rank.
In sharp contrast to the relative spin orienting procedure of \cite[Theorem~8.1.1]{FOOO},
the orientation from the approach of~\cite{Ge2} with a rank~1  real bundle $(E,\wt\tau)$
depends only on $w_1(E^{\wt\tau})$ and the spin structure on 
$TX^{\phi}\!\oplus\!2E^{\wt\tau}$, not on $(E,\wt\tau)$ itself;
see Remark~\ref{Ge2_rmk2}.
\end{rmk}

\subsection{Outline of the paper and acknowledgments}
\label{introend_subs}

\noindent
Section~\ref{prelim_sec} sets up the notation necessary for the remainder
of this paper and summarizes the orientation construction of~\cite{RealGWsI}. 
Theorems~\ref{ComplexOrient_thm} and~\ref{RelSpinOrient_thm} are proved 
in Sections~\ref{ComplexOrient_subs} and~\ref{RelSpinOrient_subs}, respectively.
Section~\ref{OrientApply_subs} obtains a number of
computationally useful statements concerning orientations of 
the determinants of real Cauchy-Riemann operators on real bundle pairs.
Theorem~\ref{CompOrient_thm} is established in Section~\ref{CompOrient_sec}.\\

\noindent
We would like to thank E.~Brugall\'e, R.~Cr\'etois, E.~Ionel, S.~Lisi,
M.~Liu, J.~Solomon, J.~Starr, M.~Tehrani, G.~Tian, and J.~Welschinger for related discussions.
The second author is very grateful to the IAS School of Mathematics for its hospitality 
during the initial stages of our project on real GW-theory.

\section{Notation and review}
\label{prelim_sec}

\noindent
We set up the necessary notation involving moduli spaces of stable maps and curves
in Section~\ref{ModSp_sub}.
We then recall standard facts concerning determinant lines of Fredholm operators in
Section~\ref{DetLB_subs}. 
Section~\ref{CompOrient_sec} reviews some of the key statements from~\cite{RealGWsI}.

\subsection{Moduli spaces of symmetric surfaces and real maps}
\label{ModSp_sub}

\noindent
Let $(\Si,\si)$ be a genus~$g$ symmetric surface.
We denote by $\cD_\si$ the group of orientation-preserving diffeomorphisms of~$\Si$  
 commuting with the involution~$\si$.
If $(X,\phi)$ is a smooth manifold with an involution, $l\!\in\!\Z^{\ge0}$,
and $B\in H_2(X;\Z)$, let  
$$\fB_{g,l}(X,B)^{\phi,\si}\subset \fB_g(X)^{\phi,\si}\times \Si^{2l}$$
denote the space of real maps  $u\!:(\Si,\si)\!\lra\!(X,\phi)$ with $u_*[ \Si]_{\Z}=B$
and $l$ pairs of conjugate non-real marked distinct points.
We define
$$\cH_{g,l}(X,B)^{\phi,\si}=
\big(\fB_{g,l}(X,B)^{\phi,\si}\!\times\!\cJ_{\Si}^{\si}\big)/\cD_\si.$$
If $J\!\in\!\cJ_{\om}^{\phi}$,
the moduli space of marked real $J$-holomorphic maps in the class $B\in H_2(X;\Z)$ 
is the subspace
$$\fM_{g,l}(X,B;J)^{\phi,\si}=
\big\{[u,(z_1^+,z_1^-),\ldots,(z_l^+,z_l^-),\fJ]\!\in\!\cH_{g,l}(X, B)^{\phi,\si}\!:~
\dbar_{J,\fJ}u\!=\!0\big\},$$
where $\dbar_{J,\fJ}$ is the usual Cauchy-Riemann operator with respect 
to the complex structures $J$ on $X$ and $\fJ$ on $\Si$. 
If $g\!+\!l\!\ge\!2$, 
$$\cM_{g,l}^\si\equiv \fM_{g,l}(\pt,0)^{\id,\si}\equiv \cH_{g,l}(\pt, 0)^{\id,\si}$$
is the moduli space of marked symmetric domains. 
There is a natural forgetful morphism
\BE{ffdfn_e0}\ff:\cH_{g,l}(X,B)^{\phi,\si}\lra  \cM_{g,l}^\si\,;\EE
it drops the map component~$u$ from each element of the domain.\\

\noindent
We denote~by 
$$\ov\fM_{g,l}(X,B;J)^{\phi,\si}\supset\fM_{g,l}(X,B;J)^{\phi,\si}$$
Gromov's convergence compactification of $\fM_{g,l}(X,B;J)^{\phi,\si}$ obtained
by including stable real maps from nodal symmetric surfaces.
The (virtually) codimension-one boundary strata~of
$$\ov\fM_{g,l}(X,B;J)^{\phi,\si}-\fM_{g,l}(X,B;J)^{\phi,\si}
\subset \ov\fM_{g,l}(X,B;J)^{\phi,\si}$$
consist of real $J$-holomorphic maps from one-nodal symmetric surfaces to~$(X,\phi)$.
Each stratum is either a (virtual) hypersurface in $\ov\fM_{g,l}(X,B;J)^{\phi,\si}$
or a (virtual) boundary of the spaces $\ov\fM_{g,l}(X,B;J)^{\phi,\si}$
for precisely two  topological types of orientation-reversing involutions~$\si$
on~$\Si$.
Let 
$$\fM_{g,l}(X,B;J)^{\phi}=\bigsqcup_{\si}\fM_{g,l}(X,B;J)^{\phi,\si}
\quad\hbox{and}\quad
\ov\fM_{g,l}(X,B;J)^{\phi}=\bigcup_{\si}\ov\fM_{g,l}(X,B;J)^{\phi,\si}$$
denote the (disjoint) union of the uncompactified real moduli spaces 
and the union of the compactified real moduli spaces, respectively, 
taken over all topological types of orientation-reversing involutions~$\si$ on~$\Si$.
If $g\!+\!l\!\ge\!2$,  we denote~by
\begin{gather*}
\ov\cM_{g,l}^{\si}\equiv \ov\fM_{g,l}(\pt,0)^{\id,\si}\supset\cM_{g,l}^{\si}\,,
\qquad
\R\ov\cM_{g,l} \equiv \ov\fM_{g,l}(\pt,0)^{\id} =\bigcup_{\si}\ov\cM_{g,l}^{\si}
\end{gather*}
the real Deligne-Mumford moduli spaces.
The forgetful morphism~\eref{ffdfn_e0} extends to a morphism
\BE{ffdfn_e}\ff\!: \ov\fM_{g,l}(X,B;J)^{\phi} \lra \R\ov\cM_{g,l}\EE
between the compactifications.

\subsection{Determinant line bundles}
\label{DetLB_subs}

\noindent
Let $(V,\vph)$ be a real bundle pair over a symmetric surface~$(\Si,\si)$.
A \textsf{real Cauchy-Riemann} (or \sf{CR-}) \sf{operator} on~$(V,\vph)$  is a linear map of the~form
\BE{CRdfn_e}\begin{split}
D=\bp\!+\!A\!: \Ga(\Si;V)^{\vph}
\equiv&\big\{\xi\!\in\!\Ga(\Si;V)\!:\,\xi\!\circ\!\si\!=\!\vph\!\circ\!\xi\big\}\\
&\hspace{.1in}\lra
\Ga_{\fJ}^{0,1}(\Si;V)^{\vph}\equiv
\big\{\ze\!\in\!\Ga(\Si;(T^*\Si,\fJ)^{0,1}\!\otimes_{\C}\!V)\!:\,
\ze\!\circ\!\tnd\si=\vph\!\circ\!\ze\big\},
\end{split}\EE
where $\bp$ is the holomorphic $\bp$-operator for some $\fJ\!\in\!\cJ_{\Si}^{\si}$
and a holomorphic structure in~$V$ and  
$$A\in\Ga\big(\Si;\Hom_{\R}(V,(T^*\Si,\fJ)^{0,1}\!\otimes_{\C}\!V) \big)^{\vph}$$ 
is a zeroth-order deformation term. 
A real CR-operator on a real bundle pair is Fredholm in the appropriate completions.\\

\noindent
If $X,Y$ are Banach spaces and $D\!:X\!\lra\!Y$ is a Fredholm operator, let
$$\det D\equiv\La_{\R}^{\top}(\ker D) \otimes \big(\La^{\top}_{\R}(\text{cok}\,D)\big)^*$$
denote the \textsf{determinant line} of~$D$. 
A continuous family of such Fredholm operators~$D_t$ over a topological space~$\cH$  
determines a line bundle over~$\cH$, called \sf{the determinant line bundle of~$\{D_t\}$}
and denoted $\det D$; see \cite[Section~A.2]{MS} and \cite{detLB} for a construction. 
A short exact sequence of Fredholm operators
\[\begin{CD}
0
@>>>X'@>>>X@>>>X''@>>>0 \\
@. @V V D' V@VV D V@VV D'' V@.\\
0@>>> Y'@>>>Y@>>>Y''@>>>0
\end{CD}\]
determines a canonical isomorphism
\BE{sum} \det D\cong (\det D')\otimes (\det D'').\EE
For a continuous family of short exact sequences of Fredholm operators, 
the isomorphisms~\eref{sum} give rise to a canonical isomorphism
between determinant line bundles.\\ 

\noindent
Families of real CR-operators often arise by pulling back data from
a target manifold by smooth maps as follows. 
Suppose $(X,J,\phi)$ is an almost complex manifold with an anti-complex  involution
and $(V,\vph)$ is a real bundle pair over~$(X,\phi)$.
Let $\na$ be a $\vph$-compatible connection in $V$ and 
$$A\in\Ga\big(X;\Hom_{\R}(V,(T^*X,J)^{0,1}\otimes_{\C}\!V)\big)^{\vph}.$$ 
For any real map $u\!:(\Si,\si)\!\lra\!(X,\phi)$ and $\fJ\!\in\!\cJ_{\Si}^{\si}$, 
let $\na^u$ denote the induced connection in $u^*V$ and
$$ A_{\fJ;u}=A\circ \prt_{\fJ} u\in\Ga(\Si;
\Hom_{\R}(u^*V,(T^*\Si,\fJ)^{0,1}\otimes_{\C}u^*V)\big)^{u^*\vph}.$$
The homomorphisms
$$\bp_u^\na =\frac{1}{2}(\na^u+\fI\circ\na^u\circ\fJ), \,\,
D_{(V,\vph);u}\equiv \bp_u^\na\!+\!A_{\fJ;u}\!: \Ga(\Si;u^*V)^{u^*\vph}\lra
\Ga^{0,1}_{\fJ}(\Si;u^*V)^{u^*\vph}$$
are real CR-operators on $u^*(V,\vph)\!\lra\!(\Si,\si)$
that form families of real CR-operators over families of maps. 
If $g,l\!\in\!\Z^{\ge0}$ and $B\!\in\!H_2(X;\Z)$, let
$$\det D_{(V,\vph)}\lra \fB_{g,l}(X,B)^{\phi,\si}\!\times\!\cJ_{\Si}^{\si}$$ 
denote the determinant line bundle of such a family.
It descends to a fibration
$$\det D_{(V,\vph)}\lra \cH_{g,l}(X,B)^{\phi,\si},$$
which is a line bundle over the open subspace of the base consisting of marked maps 
with no non-trivial automorphisms.

\begin{eg}\label{ex_tbdl}
Let $(V,\vph)\!=\!(\C,\fc)$; this is a real bundle over $(\pt,\id)$. 
If $g\!+\!l\!\ge\!2$,
the induced family of operators $\dbar_{\C}\equiv D_{(\C,\fc)}$ on $\cM_{g,l}^\si$ 
defines a line bundle 
$$\det \dbar_{\C} \lra \cM_{g,l}^\si\,.$$
If $(X,\phi)$ is an almost complex manifold with an anti-complex involution~$\phi$ and 
$$(V,\vph)=(X\!\times\!\C,\phi\!\times\!\fc)\lra (X,\phi),$$
then there is a canonical isomorphism 
$$\det D_{(\C,\fc)}\approx\mathfrak{f}^*\big(\!\det\dbar_{\C}\big)$$
of line bundles over $\cH_{g,l}(X,B)^{\phi,\si}$.
\end{eg}

\noindent
For a  real CR-operator~$D$ on a rank~$n$ real bundle pair $(V,\vph)$
over a symmetric surface~$(\Si,\si)$, we~define
 the \sf{relative determinant} of~$D$ to be the tensor product
\BE{fDdfn_e}
\rdet\,D\equiv \big(\!\det D\big)\otimes\big(\!\det\dbar_{\Si;\C}\big)^{\otimes n}\,,\EE
where $\det\dbar_{\Si;\C}$ is 
the standard real CR-operator on~$(\Si,\si)$ with values in~$(\C,\fc)$.
This notion plays a central role in the construction of real GW-theory in~\cite{RealGWsI}.\\

\noindent
Let $(X,\om,\phi)$ be a real symplectic $2n$-manifold, $g,l\!\in\!\Z^{\ge0}$, 
\hbox{$B\!\in\!H_2(X;\Z)$}, $J\!\in\!\cJ_{\om}^{\phi}$, and
$$[\u]\equiv\big[u,(z_1^+,z_1^-),\ldots,(z_l^+,z_l^-),\fJ\big]
\in \ov\fM_{g,l}(X,B;J)^{\phi}\,.$$
Denote by $\Si_u$ the domain of~$u$.
If
$$\cC\equiv\big(\Si_u,(z_1^+,z_1^-),\ldots,(z_l^+,z_l^-),\fJ\big)$$
is a stable curve, then the forgetful morphism~\eref{ffdfn_e} induces an isomorphism
\BE{ffisom_e}
\La_{\R}^{\top}\big(T_{[\u]}\ov\fM_{g,l}(X,B;J)^{\phi,\si}\big)\approx
\big(\!\det D_{(TX,\tnd\phi);u}\big)\otimes 
\La_{\R}^{\top}\big(T_{[\cC]}\ov\cM_{g,l}^{\si}\big) \,.\EE
Orientations on the two lines on the right-hand side of~\eref{ffisom_e} thus determine
an orientation on the left-hand side of~\eref{ffisom_e}.
If $(X,\om,\phi)$ is real-orientable and $n$ is odd, 
as in the cases relevant to the present paper, 
the index of $D_{(TX,\tnd\phi);u}$ is odd if and only if $g\!\in\!2\Z$.
The induced orientation on the left-hand side of~\eref{ffisom_e} then depends
on the specified order of the factors
on the right-hand side of~\eref{ffisom_e}.

\subsection{Real orientations and relative determinants}
\label{RealOrientations_subs}

\noindent
Let $(X,\phi)$ be a topological space with an involution and 
$(V,\vph)$ be a real bundle pair over~$(X,\phi)$.
An isomorphism~$\Th$ in~\eref{realorient_e4} determines orientations on 
$V^{\vph}$ and $V^{\vph}\!\oplus\! 2(L^*)^{\wt\phi^*}$.
Given a real orientation on~$(V,\vph)$ as in Definition~\ref{realorient_dfn4},
we will call  these orientations \textsf{the orientations determined by~\ref{isom_it2}}
if $\Th$ lies in the chosen homotopy class.
An isomorphism~$\Th$ in~\eref{realorient_e4} also induces an isomorphism 
\BE{detInd_e}\begin{split}
\La_{\C}^{\top}\big(V\!\oplus\!2L^*,\vph\!\oplus\!2\wt\phi^*\big)
&\approx \La_{\C}^{\top}(V,\vph)\otimes(L^*,\wt\phi^*)^{\otimes 2}\\
&\approx (L,\wt\phi)^{\otimes 2}\otimes(L^*,\wt\phi^*)^{\otimes 2}
\approx \big(\Si\!\times\!\C,\si\!\times\!\fc\big),
\end{split}\EE
where the last isomorphism is the canonical pairing.
We will call the homotopy class of isomorphisms~\eref{detInd_e} induced by 
the isomorphisms~$\Th$ in~\ref{isom_it2} \textsf{the homotopy class determined by~\ref{isom_it2}}.

\begin{prp}[{\cite[Proposition 7.3]{RBP}}]\label{canonisom_prp}
Suppose $(\Si,\si)$ is a symmetric surface, possibly disconnected and nodal, and  
$(V,\vph)$ is a rank~$n$ real bundle pair over $(\Si,\si)$.
A real orientation on $(V,\vph)$ as in Definition~\ref{realorient_dfn4}
determines a homotopy class of isomorphisms
\BE{realorient2_e2} \Psi\!: \big(V\!\oplus\!2L^*,\vph\!\oplus\!2\wt\phi^*\big)
\approx\big(\Si\!\times\!\C^{n+2},\si\!\times\!\fc\big)\EE
of real bundle pairs over $(\Si,\si)$.
An isomorphism~$\Psi$ belongs to this homotopy class if and only~if
the restriction of $\Psi$ to the real locus induces the chosen spin structure~\ref{spin_it2} 
and  the isomorphism 
\BE{realorient2_e2b}
\La_{\C}^{\top}\Psi\!: \La_{\C}^{\top}\big(V\!\oplus\!2L^*,\vph\!\oplus\!2\wt\phi^*\big)
\lra \La_{\C}^{\top}\big(\Si\!\times\!\C^{n+2},\si\!\times\!\fc\big)
=\big(\Si\!\times\!\C,\si\!\times\!\fc\big)\EE
lies in the homotopy class determined by~\ref{isom_it2}.
\end{prp}

\noindent
The only cases of this proposition relevant to~\cite{RealGWsI} 
are for~$\Si$ smooth and with one real node;
the only cases relevant to Theorem~\ref{CompOrient_thm}
are for~$\Si$ smooth and with one pair of conjugate nodes.
The proof of \cite[Proposition~5.2]{RealGWsI} establishes Proposition~\ref{canonisom_prp} 
under  the assumption that~$\Si$ is connected and smooth, 
but it applies without the first restriction.
The proof of \cite[Proposition~6.2]{RealGWsI} extends \cite[Proposition~5.2]{RealGWsI}
to one-nodal symmetric surfaces.
The analogue of this extension for symmetric surfaces with one pair of conjugate
nodes is carried out in the proof of Lemma~\ref{canonisomExt_lmm} of the present paper.
The principles behind the two extensions are used in~\cite{RBP} to establish 
the full statement of Proposition~\ref{canonisom_prp}.

\begin{crl}\label{canonisom_crl2a}
Suppose $(\Si,\si)$ is a symmetric surface, possibly disconnected and nodal,
and  $D$ is a real CR-operator 
on a rank~$n$ real bundle pair $(V,\vph)$ over~$(\Si,\si)$.
Then a real orientation on $(V,\vph)$ as in Definition~\ref{realorient_dfn4}
induces an  orientation on the relative determinant~$\rdet\,D$ of~$D$.
\end{crl}

\noindent
For $\Si$ smooth or one-nodal,
this corollary  is deduced from the corresponding cases of
Proposition~\ref{canonisom_prp} in the proofs of
\cite[Corollary~5.7]{RealGWsI} and \cite[Corollary~6.6]{RealGWsI}, respectively. 
The proof of the latter readily extends to all symmetric surfaces~$(\Si,\si)$.\\

\noindent
Corollary~\ref{canonisom_crl2a} implies that a real orientation on
a real symplectic manifold $(X,\om,\phi)$ determines an orientation on the line
\BE{detDorient_e} 
\rdet\,D_{(TX,\tnd\phi);u}\equiv
\big(\!\det D_{(TX,\tnd\phi);u}\big)\otimes 
\big(\!\det\dbar_{\C}|_{\Si_u}\big)^{\otimes n}\,.\EE
By \cite[Corollary~6.7]{RealGWsI}
and Corollary~\ref{canonisomExt2_crl2a}, this orientation varies continuously with~$[\u]$.

\begin{crl}\label{canonisom_crl}
Suppose $(\Si,\si)$ is a symmetric surface, possibly disconnected and nodal, and  
$(L,\wt\phi)\!\lra\!(\Si,\si)$ is a rank~1 real bundle pair.
If $L^{\wt\phi}\!\lra\!\Si^{\si}$ is orientable, there exists a canonical homotopy 
class of isomorphisms 
\BE{canonisom_crl_e}\big(L^{\otimes2}\!\oplus\!2L^*,\wt\phi^{\otimes2}\!\oplus\!2\wt\phi^*\big)
\approx\big(\Si\!\times\!\C^3,\si\!\times\!\fc\big)\EE
of real bundle pairs over $(\Si,\si)$.
\end{crl}

\noindent
As explained in the proof of \cite[Corollary~5.6]{RealGWsI}, there is a canonical
real orientation on the real bundle $(L,\wt\phi)^{\otimes2}$ over  $(\Si,\si)$
if $L^{\wt\phi}\!\lra\!\Si^{\si}$ is orientable.
In particular, there is a canonical homotopy class of isomorphisms
$$\big(T^*\Si^{\otimes2}\!\oplus\!2\,T\Si,(\tnd\si^*)^{\otimes2}\!\oplus\!2\tnd\si\big)
\approx\big(\Si\!\times\!\C^3,\si\!\times\!\fc\big)$$
of real bundle pairs over~$(\Si,\si)$ if~$\Si$ contains no real nodes
(of type~(H2) or~(H3) in the terminology of \cite[Section~3]{Melissa} and 
\cite[Section~3.2]{RealGWsI}).\\

\noindent
Let $g,l\in \Z^{\ge0}$ be such that $g\!+\!l\!\ge\!2$ and $(\Si,\si)$
be a smooth connected symmetric surface of genus~$g$.
Combining the Kodaira-Spencer isomorphism, Dolbeault Isomorphism, Serre Duality,  and 
Corollaries~\ref{canonisom_crl2a} and~\ref{canonisom_crl}, we find that the real line~bundle
\BE{CidentDM_e}\La_{\R}^{\top}\big(T\cM_{g,l}^{\si}\big) \otimes 
\big(\!\det\dbar_{\C}\big)
\lra \cM_{g,l}^{\si}\EE
is canonically oriented; see the proof of \cite[Proposition~5.9]{RealGWsI}.
If $n\!\not\in\!2\Z$ and the domain~$\Si_u$ of~$\u$ in~\eref{ffisom_e} is smooth,
the canonical orientation on~\eref{CidentDM_e} and an orientation on~\eref{detDorient_e}
determine an orientation on the line~\eref{ffisom_e} which varies continuously with~$\u$.
Thus, a real orientation on a real symplectic manifold $(X,\om,\phi)$ determines
orientations on the uncompactified moduli spaces $\fM_{g,l}(X,B;J)^{\phi,\si}$
of real $J$-holomorphic maps from $(\Si,\si)$ to~$(X,\phi)$.\\

\noindent
By \cite[Proposition~6.1]{RealGWsI}, the canonical orientations of 
the real line bundle~\eref{CidentDM_e} extend across 
a codimension-one boundary stratum of $\R\ov\cM_{g,l}$
if and only if the parity of the number $|\si|_0$ of connected components 
of the fixed locus~$\Si^{\si}$ of~$\Si$ remains unchanged.
By construction, the same is the case of the orientations on 
 $\fM_{g,l}(X,B;J)^{\phi,\si}$
induced by a real  orientation on $(X,\om,\phi)$ if $n\!\!\not\in\!2\Z$.
In order to orient the compactified moduli spaces $\ov\fM_{g,l}(X,B;J)^{\phi}$, 
we multiply the orientation on $\fM_{g,l}(X,B;J)^{\phi,\si}$ induced by 
a real  orientation on $(X,\om,\phi)$ by $(-1)^{g+|\si|_0+1}$.
This does not change the orientations whenever the fixed locus~$\Si^{\si}$ of~$\Si$ 
is separating.

\section{Comparison of orientations}
\label{CompOrient_sec}

\noindent
There are now standard ways of imposing orientations
on the moduli spaces  $\fM_{g,l}(X,B;J)^{\phi,\si}$
for certain types of symmetric surfaces~$(\Si,\si)$. 
Theorems~\ref{ComplexOrient_thm} and~\ref{RelSpinOrient_thm} compare such 
orientations with the orientations constructed in~\cite{RealGWsI}
and briefly described in Section~\ref{RealOrientations_subs}.

\subsection{Canonical vs.~complex}
\label{ComplexOrient_subs}

\noindent 
We continue with the notation and setup of Section~\ref{ComplexOrient_subs0}.
In the setting of Theorem~\ref{ComplexOrient_thm}, 
each of the factors in~\eref{detDorient_e} and~\eref{CidentDM_e} has a natural complex orientation.
By Lemma~\ref{ComplexOrient_lmm1} below,
the orientations of the tensor product in~\eref{detDorient_e} induced by 
a real orientation on~$(X,\om,\phi)$ and by the complex orientations on the two factors
are the same.
By Lemma~\ref{ComplexOrient_lmm2},
the canonical orientation of the tensor product in~\eref{CidentDM_e} and 
the orientation induced by the complex orientations on the two factors
differ by $(-1)^{g_0+1+|\fs|}$.\\

\noindent
The exponent of $g_0\!+\!1$ above arises for the following reason.
Let $V$ be a complex vector space of dimension~$k$.
The~map
\BE{duaVorient_e} \Hom_{\C}(V,\C)\lra  \Hom_{\R}(V,\R), \qquad
\th\lra \Re\,\th,\EE
is then an isomorphism of real vector spaces.
Its domain is a complex vector space and thus has a canonical complex orientation;
its image has an orientation induced from the complex orientation of~$V$.
The isomorphism~\eref{duaVorient_e} is orientation-preserving with respect to these orientations 
if and only~if $k$ is even.
The only step in the proof of \cite[Proposition~5.9]{RealGWsI} not compatible 
with the natural complex orientations is taking the (real) dual in \cite[(5.21)]{RealGWsI}.
The sign discrepancy of~\eref{duaVorient_e} for the twists by the marked points
is taken into account earlier in the proof.
The ``remaining" vector space in \cite[(5.21)]{RealGWsI} has complex dimension $3g_0\!-\!3$
and accounts for the exponent of $g_0\!+\!1$ in the sign of Theorem~\ref{ComplexOrient_thm}.\\ 

\noindent
A rank~$n$ real bundle pair $(V,\vph)$ over a doublet  $(\Si,\si)$ as in~\eref{SymSurfDbl_e}
corresponds to a complex vector bundle $V_0\!\lra\!\Si_0$ with 
$$V= V_1\!\sqcup\!V_2
\equiv \{1\}\!\times\!V_0 \sqcup \{2\}\!\times\!\ov{V}_0, \qquad
\vph(i,z)=\big(3\!-\!i,v\big)~~\forall~(i,v)\!\in\!V,$$
where $\ov{V}_0$ denotes $V_0$ with the opposite complex structure.
With these identifications,
$$\Ga(\Si;V)^{\vph}\subset \Ga(\Si_1;V_1)\oplus\Ga(\Si_2;V_2), \qquad
\Ga_{\fJ}^{0,1}(\Si;V)^{\vph}\subset \Ga_{\fJ}^{0,1}(\Si_1;V_1)\oplus\Ga_{-\fJ}^{0,1}(\Si_2;V_2),$$
and the projections
\BE{ComplexOrient_e5}
\Ga(\Si;V)^{\vph}\lra \Ga(\Si_0;V_0) \qquad\hbox{and}\qquad
\Ga_{\fJ}^{0,1}(\Si;V)^{\vph}\subset \Ga_{\fJ}^{0,1}(\Si_0;V_0)\EE
to the first component are isomorphisms of real vector spaces.
Via these projections, every real CR-operator~$D$ on the real bundle pair $(V,\vph)$
corresponds to an operator 
$$D_0\!:\Ga(\Si_0;V_0) \lra \Ga_{\fJ}^{0,1}(\Si_0;V_0)\,.$$
The projections~\eref{ComplexOrient_e5} induce isomorphisms between 
the kernels and cokernels of~$D$ and~$D_0$ and thus an isomorphism
\BE{ComplexOrient_e7}\det D\approx\det D_0\,.\EE
Since $D_0$ is a \sf{real linear CR-operator on~$V_0$} in the sense
of \cite[Definition~C.1.5]{MS},
$\det D_0$ has a canonical ``complex" orientation  obtained by homotoping~$D_0$ 
to a $\C$-linear Fredholm operator; see \cite[Section~3.2]{MS}.
We will call the orientation on $\det D$ induced from this orientation via
the isomorphism~\eref{ComplexOrient_e7} \sf{the complex orientation of} $\det D$.

\begin{lmm}\label{ComplexOrient_lmm1}
Let $(\Si,\si)$, $(V,\vph)$, and $D$ be as above. 
The orientations of the relative determinant $\rdet\,D$ of~$D$ 
induced by a real orientation on~$(V,\vph)$ as in Corollary~\ref{canonisom_crl2a}
and by the complex orientations on the two factors
are the same.
\end{lmm}

\begin{proof}
The homotopy class of isomorphisms as in~\eref{realorient2_e2} determined 
by a real orientation on~$(V,\vph)$ determines an orientation on the~line
\BE{ComplexOrient1_e3}\begin{split}
&\big(\!\det D_{(V\oplus 2L^*,\vph\oplus 2\wt\phi^*)}\big)\otimes 
\big(\!\det\dbar_{\Si;\C}\big)^{\otimes (n+2)}\\
&\hspace{1.2in}\approx \big(\!\det\!\big(D_{(V\oplus 2L^*,\vph\oplus 2\wt\phi^*)}\big)_0\big)\otimes 
\big(\!\det\!\big(\dbar_{\Si;\C}\big)_0\big)^{\otimes (n+2)}\,.
\end{split}\EE
Any isomorphism~$\Psi$ in~\eref{realorient2_e2} corresponds to an isomorphism
$$\Psi_0\!: V_0\!\oplus\!2L_0^*\lra \Si_0\!\times\!\C^{n+2}$$
of complex vector bundles over~$\Si_0$ by the restriction to~$\Si_1\!\subset\!\Si$.
The isomorphism in~\eref{ComplexOrient1_e3} is orientation-preserving with respect 
to the orientation on the left-hand side induced by~$\Psi$
and the orientation on the right-hand side induced by~$\Psi_0$.
Since $\Psi_0$ is a $\C$-linear isomorphism,
the operator on $\Si_0\!\times\!\C^{n+2}$ induced by $(D_{(V\oplus 2L^*,\vph\oplus 2\wt\phi^*)})_0$ 
via~$\Psi_0$  
is a real linear CR-operator.
Since any two such operators are homotopic, 
the orientation on the last factor in~\eref{ComplexOrient1_e3} induced from the complex orientation
of the third factor in~\eref{ComplexOrient1_e3} is the complex orientation. 
Thus, the orientation  on the left-hand side of~\eref{ComplexOrient1_e3}
induced by a real orientation on~$(V,\vph)$ is  
the orientation induced by the complex orientations on the two factors.\\

\noindent
By~\eref{sum}, there are horizontal canonical isomorphisms
\BE{ComplexOrient1_e7}\begin{split}
\xymatrix{ \det D_{(V\oplus 2L^*,\vph\oplus 2\wt\phi^*)}
\ar[d]_{\eref{ComplexOrient_e7}}^{\approx}
\ar[rr]^>>>>>>>>>>>>>{\eref{sum}}_>>>>>>>>>>>>>{\approx}&& 
\big(\!\det D_{(V,\vph)}\big)\otimes\big(\!\det D_{(L^*,\wt\phi^*)}\big)^{\otimes2}
\ar[d]_{\eref{ComplexOrient_e7}}^{\approx}\\
\det\!\big(D_{(V\oplus 2L^*,\vph\oplus 2\wt\phi^*)}\big)_0 
\ar[rr]^>>>>>>>>>>{\eref{sum}}_>>>>>>>>>>{\approx} &&
\big(\!\det\!\big(D_{(V,\vph)}\big)_0\big)
\otimes\big(\!\det\!\big(D_{(L^*,\wt\phi^*)}\big)_0\big)^{\otimes2}}
\end{split}\EE
making the diagram commute.
Thus, the top isomorphism in~\eref{ComplexOrient1_e7} is orientation-preserving with respect to the complex orientations on the three determinants. 
The orientation of $\rdet\,D$
induced by a real orientation on~$(V,\vph)$ as in Corollary~\ref{canonisom_crl2a}
 is obtained by combining
\begin{enumerate}[label=(\arabic*),leftmargin=*]

\item  the orientation on LHS of~\eref{ComplexOrient1_e3} induced by,

\item the top isomorphism in~\eref{ComplexOrient1_e7}, and

\item  the canonical orientations of 
 $(\det D_{(L^*,\wt\phi^*)})^{\otimes2}$ and $(\det\dbar_{\Si;\C}\big)^{\otimes2}$.

\end{enumerate}
By the last sentence of the previous paragraph and the sentence after~\eref{ComplexOrient1_e7},
this is the orientation induced by the complex orientations on the two factors. 
\end{proof}

\noindent
For $g\!\in\!\Z$ and $l\!\in\!\Z^{\ge0}$ with $g\!+\!l\!\ge\!2$, we denote~by 
$$\R\ov\cM_{g,l}^{\bu}\supset \R\cM_{g,l}^{\bu}$$ 
the Deligne-Mumford moduli space of possibly disconnected stable nodal symmetric surfaces
of Euler characteristic $2(1\!-\!g)$ with $l$ pairs of conjugate marked points
and its subspace consisting of smooth curves.
If $g_0,l\!\in\!\Z^{\ge0}$ with $2g_0\!+\!l\!\ge\!3$ and 
$(\Si,\si)$ is a $g_0$-doublet  as in~\eref{SymSurfDbl_e}, let 
$$\cM_{2g_0-1,l}^{\si}=\fM_{2g_0-1,l}^{\bu}(\pt,0)^{\id,\si}
\subset \R\cM_{2g_0-1,l}^{\bu}\,.$$
For each $\fs\!\subset\!\{1,\ldots,l\}$, let
$$\cM_{2g_0-1,l;\fs}^{\si}= \fM_{2g_0-1,l}^{\bu}(\pt,0)_{\fs}^{\id,\si}
\subset\cM_{2g_0-1,l}^{\si}$$
be the open subspace consisting of marked curves so that  the second point
in the $i$-th conjugate pair lies on~$\Si_1$ if and only~if $i\!\in\!\fs$.
In particular,
\BE{ComplexOrient_e20}
\cM_{2g_0-1,l;\fs}^{\si}\subset \cM_{g_0,l}\!\times\!\cM_{g_0,l}\,,\EE
where $\cM_{g_0,l}$ is the usual Deligne-Mumford moduli space of 
smooth genus~$g_0$ curves with $l$~marked points.
The projection
\BE{ComplexOrient_e21}
\cM_{2g_0-1,l;\fs}^{\si} \lra \cM_{g_0,l}\EE
to the first factor in~\eref{ComplexOrient_e20} is a diffeomorphism.
The moduli space on the right-hand side of~\eref{ComplexOrient_e21} carries a natural complex orientation.
We will call the orientation on the left-hand side of~\eref{ComplexOrient_e21}
induced by this orientation 
\sf{the complex orientation} of $\cM_{2g_0-1,l;\fs}^{\si}$.

\begin{lmm}\label{ComplexOrient_lmm2}
Let $g_0,l\!\in\!\Z^{\ge0}$ with $2g_0\!+\!l\!\ge\!3$ and $(\Si,\si)$ 
be a $g_0$-doublet.
The canonical orientation on the real line bundle
\BE{ComplexOrient2_e0}\La_{\R}^{\top}\big(T\cM_{2g_0-1,l;\fs}^{\si}\big) \otimes 
\big(\!\det\dbar_{\C}\big) \lra \cM_{2g_0-1,l;\fs}^{\si}\EE
constructed as in the proof of \cite[Proposition~5.9]{RealGWsI}
and the orientation induced by the complex orientations of the factors differ 
by $(-1)^{g_0+1+|\fs|}$.
\end{lmm}

\begin{proof}
Since the interchange of the points within a conjugate pair reverses the canonical orientation
of~\eref{ComplexOrient2_e0}, it is sufficient to establish the claim for $\fs\!=\!\eset$.
Let
$$[\cC_0]=\big[\Si_0,z_1^+,\ldots,z_l^+,\fJ\big]\in \cM_{g_0,l}
\quad\hbox{and}\quad
[\cC]=\big[\Si,(z_1^+,z_1^-),\ldots,(z_l^+,z_l^-),
\fJ\!\sqcup\!(-\fJ)\big]\in \cM^{\si}_{2g_0-1,l}.$$
Similarly to the proof of \cite[Proposition~5.9]{RealGWsI}, we define
\begin{alignat*}{2}
T\cC_0&=T\Si_0\big(\!-\!z_1^+\!-\!\ldots\!-\!z_l^+\!\big), &\qquad
T^*\cC_0&=T^*\Si_0\big(z_1^+\!+\!\ldots\!+\!z_l^+\big),\\
T\cC&=T\Si\big(\!-\!z_1^+\!-\!z_1^-\!-\!\ldots\!-\!z_l^+\!-\!z_l^-\big), &\qquad
T^*\cC&=T^*\Si\big(z_1^+\!+\!z_1^-\!+\!\ldots\!+\!z_l^+\!+\!z_l^-\big).
\end{alignat*}
Denote by $S\cC_0$  the skyscraper sheaf over~$\Si_0$ and by 
$S\cC^+$, $S\cC^-$, and~$S\cC$ the skyscraper sheaves 
over~$\Si$ given~by 
$$S\cC_0=T^*\Si_0|_{z_1^++\ldots+z_l^+}, \quad
S\cC^+=T^*\Si|_{z_1^++\ldots+z_l^+}, \quad S\cC^-=T^*\Si|_{z_1^-+\ldots+z_l^-}\,,
\quad S\cC=S\cC^+\!\oplus\!S\cC^-\,.$$
The projection
\BE{H0ScC_e}
\pi_1\!:  H^0(\Si;S\cC)^{\si}=
\big(H^0(\Si;S\cC^+)\!\oplus\!H^0(\Si;S\cC^-)\big)^{\si} 
\lra H^0(\Si;S\cC^+)= H^0(\Si_0;S\cC_0)\EE
is an isomorphism of real vector spaces.
In the proof of \cite[Proposition~5.9]{RealGWsI},
we orient the domain of this isomorphism and its dual,
i.e.~the space of homomorphisms into~$\R$, via the isomorphism
$$\pi_1^*\!:
 H^0(\Si;S\cC^+)^*=T_{z_1^+}\Si\!\oplus\!\ldots\!\oplus\!T_{z_l^+}\Si
\lra \big(H^0(\Si;S\cC)^{\si}\big)^*$$
from the complex orientations of $T_{z_1^+}\Si,\ldots,T_{z_l^+}\Si$.
Thus, the isomorphism 
\BE{ComplexOrient2_e5}\Hom_{\C}\big(H^0(\Si_0;S\cC_0),\C)\approx H^0(\Si;S\cC^+)^*
\stackrel{\pi_1^*}{\lra} \big(H^0(\Si;S\cC)^{\si}\big)^*\EE
is orientation-preserving with respect to the complex orientation on the left-hand side
and the orientation in the proof of \cite[Proposition~5.9]{RealGWsI}
on the right-hand side.\\

\noindent
The Kodaira-Spencer map, Dolbeault isomorphism, and Serre Duality for $[\cC]\!\in\!\cM_{2g_0-1,l}^{\si}$
as in \cite[(5.20),(5.21)]{RealGWsI} and \hbox{$[\cC_0]\!\in\!\cM_{g_0,l}$}
form a commutative diagram
$$\xymatrix{T_{[\cC]}\cM^{\si}_{2g_0-1,l}\ar[r]^{\tn{KS}}_{\approx}\ar[d]^{\approx} & 
\wch{H}^1(\Si;T\cC)^{\si} \ar[r]^{\tn{DI}}_{\approx} \ar[d]^{\approx} & 
H^1(\Si;T\cC)^{\si} \ar[r]^>>>>>>>>>>>{\tn{SD}}_>>>>>>>>>>>{\approx} \ar[d]^{\approx} &  
\big(H^0(\Si;T^*\cC\!\otimes\!T^*\Si)^{\si}\big)^*  \\
T_{[\cC_0]}\cM_{g_0,l}\ar[r]^{\tn{KS}}_{\approx}&  
\wch{H}^1(\Si_0;T\cC_0) \ar[r]^{\tn{DI}}_{\approx}&  
H^1(\Si_0;T\cC_0) \ar[r]^>>>>>{\tn{SD}}_>>>>>{\approx} & 
\Hom_{\C}\big(H^0(\Si_0;T^*\cC_0\!\otimes\!T^*\Si_0),\C\big)  \ar[u]_{\approx}^{\pi_1^*}}$$
with the vertical arrows  given by the restrictions to $\Si_1\!=\!\Si_0$.
Since the isomorphisms in the bottom row of the above diagram are $\C$-linear, 
the natural isomorphism
\BE{ComplexOrient2_e7}\begin{split}
&\La_{\R}^{\top}\big(T_{[\cC]}\cM^{\si}_{2g_0-1,l}\big)\otimes 
\La_{\R}^{\top}\big(\big(H^0(\Si;T^*\cC\!\otimes\!T^*\Si)^{\si}\big)^* \big)\\
&\hspace{1.2in}\approx \La_{\R}^{\top}\big(T_{[\cC_0]}\cM_{g_0,l}\big)\otimes 
\La_{\R}^{\top}\big( \Hom_{\C}\big(H^0(\Si_0;T^*\cC_0\!\otimes\!T^*\Si_0),\C\big)\big)
\end{split}\EE
is orientation-preserving with respect to the orientation on the left-hand side
in the proof of \cite[Proposition~5.9]{RealGWsI} and 
the orientation on the right-hand side induced by the complex orientations on the factors.\\

\noindent
Since $2g_0\!+\!l\!\ge\!3$,
\BE{ComplexOrient2_e9} H^1(\Si_0;T^*\cC_0\!\otimes\!T^*\Si_0)=0, \qquad
\dim_{\C}H^0(\Si_0;T^*\cC_0\!\otimes\!T^*\Si_0)=3g_0\!-\!3+l.\EE
The short exact sequence of sheaves \cite[(5.22)]{RealGWsI} over~$\Si$ and its analogue
over~$\Si_0$ induce a commutative diagram of exact sequences
$$\xymatrix@C-=0.5cm{H^0(\Si;T^*\Si\!\otimes\!T^*\Si)^{\si} \ar[r]\ar[d]^{\approx}& 
H^0(\Si;T^*\cC\!\otimes\!T^*\Si)^{\si} \ar[r]\ar[d]^{\approx}&  
H^0(\Si;S\cC)^{\si} \ar[r]\ar[d]^{\approx}_{\pi_1}&  
H^1(\Si;T^*\Si\!\otimes\!T^*\Si)^{\si} \ar[d]^{\approx} \\
H^0(\Si_0;T^*\Si_0\!\otimes\!T^*\Si_0) \ar[r]& H^0(\Si_0;T^*\cC_0\!\otimes\!T^*\Si_0) \ar[r]& 
H^0(\Si_0;S\cC_0) \ar[r]&  H^1(\Si_0;T^*\Si_0\!\otimes\!T^*\Si_0) \,,}$$ 
where we omit the zero vector spaces on the ends of the two rows.
Since the isomorphisms in the bottom row of the above diagram are $\C$-linear, 
the natural isomorphism
\begin{equation*}\begin{split}
&\La_{\R}^{\top}\big(H^0(\Si;T^*\cC\!\otimes\!T^*\Si)^{\si}\big)\otimes
\det\dbar_{(T^*\Si,\tnd\si^*)^{\otimes2}}
\otimes \La_{\R}^{\top}\big(H^0(\Si_0;S\cC)^{\si}\big) \\
&\hspace{.5in}\approx 
\La_{\R}^{\top}\big(H^0(\Si_0;T^*\cC_0\!\otimes\!T^*\Si_0),\C\big)
\otimes \det\!\big(\dbar_{(T^*\Si,\tnd\si^*)^{\otimes2}}\big)_0
\otimes  \La_{\R}^{\top}\big(H^0(\Si_0;S\cC_0)\big)
\end{split}\end{equation*}
is orientation-preserving with respect to the orientation on the left-hand side
in the proof of \cite[Proposition~5.9]{RealGWsI} and 
the orientation on the right-hand side induced by the complex orientations on the factors.\\

\noindent
By the choice of the orientation on $H^0(\Si_0;S\cC)^{\si}$, the isomorphism~$\pi_1^*$
in~\eref{ComplexOrient2_e5} is orientation-preserving with respect to the complex orientation 
on its domain.
Since the complex dimension of the last vector space is~$l$,
it follows that the sign of the vertical isomorphism~$\pi_1$ in
 the last commutative diagram is~$(-1)^l$. 
Thus, the sign of the natural isomorphism
\BE{ComplexOrient2_11}\begin{split}
&\La_{\R}^{\top}\big(H^0(\Si;T^*\cC\!\otimes\!T^*\Si)^{\si}\big)\otimes
\det\dbar_{(T^*\Si,\tnd\si^*)^{\otimes2}}\\
&\hspace{1.2in}\approx 
\La_{\R}^{\top}\big(H^0(\Si_0;T^*\cC_0\!\otimes\!T^*\Si_0),\C\big)
\otimes  \det\!\big(\dbar_{(T^*\Si,\tnd\si^*)^{\otimes2}}\big)_0
\end{split}\EE
with respect to the orientation on the left-hand side
in the proof of \cite[Proposition~5.9]{RealGWsI} and 
the orientation on the right-hand side induced by the complex orientations on the factors
is~$(-1)^l$.\\

\noindent
By Lemma~\ref{ComplexOrient_lmm1}, the natural isomorphism
\BE{ComplexOrient2_15}\det\dbar_{(T^*\Si,\tnd\si^*)^{\otimes2}}\otimes \det\dbar_{\Si;\C} 
\approx  \det\!\big(\dbar_{(T^*\Si,\tnd\si^*)^{\otimes2}}\big)_0
\otimes \det\!\big(\dbar_{\Si;\C}\big)_0\EE
is orientation-preserving with respect to the orientation on the left-hand side
induced by a real orientation on $(T^*\Si,\tnd\si^*)^{\otimes2}$
and 
the orientation on the right-hand side induced by the complex orientations on the factors.
The canonical orientation on the real line bundle~\eref{ComplexOrient2_e0}
is obtained by combining 
the canonical orientations of the left-hand sides 
of~\eref{ComplexOrient2_e7}, \eref{ComplexOrient2_11}, and~\eref{ComplexOrient2_15}.\\

\noindent
By the second statement in~\eref{ComplexOrient2_e9},
the sign of the canonical isomorphism
\begin{equation*}\begin{split}
&\La_{\R}^{\top}\big(\big(H^0(\Si;T^*\cC\!\otimes\!T^*\Si)^{\si}\big)^* \big)
\otimes \La_{\R}^{\top}\big(H^0(\Si;T^*\cC\!\otimes\!T^*\Si)^{\si}\big)\\
&\hspace{.5in}\approx 
\La_{\R}^{\top}\big( \Hom_{\C}\big(H^0(\Si_0;T^*\cC_0\!\otimes\!T^*\Si_0),\C\big)\big)
\otimes \La_{\R}^{\top}\big(H^0(\Si_0;T^*\cC_0\!\otimes\!T^*\Si_0)\big)
\end{split}\end{equation*}
with respect to the canonical orientation on the left-hand side and 
the orientation on the right-hand side induced by the complex orientations on the factors
is~$(-1)^{3g_0-3+l}$. 
Combining this with the sign of the isomorphism~\eref{ComplexOrient2_11}, we obtain the claim.
\end{proof}

\begin{proof}[{\bf\emph{Proof of Theorem~\ref{ComplexOrient_thm}}}]
Throughout this argument, we will refer to the orientation on the moduli space
$\fM_{2g_0-1,l}^{\bu}(X,B;J)^{\phi}$
determined by a fixed real orientation on~$(X,\om,\phi)$ as \sf{the canonical orientation}.
Since the canonical orientation is compatible with orienting the fibers of 
the forgetful morphism~\eref{ffmarked_e} by the first point in the last conjugate pair,
we can assume that $2g_0\!+\!l\!\ge\!3$.
Let $[\u]$ be an element of $\fM_{2g_0-1,l}^{\bu}(X,B;J)_{\fs}^{\phi,\si}$,
$[\u_0]$ be its image under~\eref{ComplexOrient_e1}, 
and $[\cC]\!\in\!\cM_{2g_0-1,l;\fs}^{\si}$ and $[\cC_0]\!\in\!\cM_{g_0,l}$
be their images under the forgetful morphisms to the corresponding Deligne-Mumford moduli spaces.\\

\noindent
The canonical orientation of the tangent space at~$[\u]$ is obtained from
the canonical isomorphism
\BE{ComplexOrient0_e3}\begin{split}
&\La_{\R}^{\top}\big(T_{[\u]}\fM_{2g_0-1,l}^{\bu}(X,B;J)_{\fs}^{\phi}\big)
\otimes \big(\!\det\dbar_{\Si;\C}\big)^{\otimes{(n+1)}} \\
&\qquad\approx \Big(
\big(\!\det D_{(TX,\tnd\phi)}|_u\big)\!\otimes\!\big(\!\det\dbar_{\Si;\C}\big)^{\otimes n}\Big)
\otimes 
\Big(
\La_{\R}^{\top}\big(T_{[\cC]}\cM_{2g_0-1,l;\fs}^{\si}\big) \otimes 
\big(\!\det\dbar_{\Si;\C}\big)\Big)
\end{split}\EE
determined by the forgetful morphism~\eref{ffdfn_e} and the canonical orientation
of $(\!\det\dbar_{\Si;\C})^{\otimes{(n+1)}}$ for $n\!\not\in\!2\Z$.
The orientation of the first tensor product on the right-hand side of~\eref{ComplexOrient0_e3}
is determined by the real orientation on~$(X,\om,\phi)$ as in Corollary~\ref{canonisom_crl2a}.
The orientation of the last tensor product on the right-hand side of~\eref{ComplexOrient0_e3}
is the canonical orientation of \cite[Proposition~5.9]{RealGWsI}.
The standard complex orientation of the tangent space at~$[\u_0]$ is obtained from the canonical 
isomorphism
\BE{ComplexOrient0_e5}\begin{split}
&\La_{\R}^{\top}\big(T_{[\u_0]}\fM_{g_0,l}(X,B_0;J)\big)
\otimes \big(\!\det\!\big(\dbar_{\Si;\C}\big)_0\big)^{\otimes(n+1)}\\
&\qquad\approx \Big(
\big(\!\det\!\big(D_{(TX,\tnd\phi)}|_u\big)_0\big)\!\otimes\!
\big(\!\det\!\big(\dbar_{\Si;\C}\big)_0\big)^{\otimes n}\Big)
\otimes 
\La_{\R}^{\top}\big(T_{[\cC_0]}\cM_{g_0,l}\big) \otimes 
\big(\!\det\!\big(\dbar_{\Si;\C}\big)_0\big)
\end{split}\EE
determined by the forgetful morphism to the Deligne-Mumford space and 
the standard complex orientation of $\det(\dbar_{\C}|_{\Si_u})_0$.
The orientations of all four factors on the right-hand side of~\eref{ComplexOrient0_e5}
are the standard complex orientations.\\

\noindent
The restriction to $\Si_1\!=\!\Si_0$ intertwines the isomorphisms~\eref{ComplexOrient0_e3}
and~\eref{ComplexOrient0_e5}  and respects the four factors  on the right-hand sides.
By Lemma~\ref{ComplexOrient_lmm1}, the isomorphism between the first pairs of factors
on the right-hand sides is orientation-preserving.
By Lemma~\ref{ComplexOrient_lmm2}, the sign of the isomorphism between 
the last pairs of factors is $(-1)^{g_0+1+|\fs|}$.
This establishes the claim.
\end{proof}

\subsection{Canonical vs.~spin and relative spin}
\label{RelSpinOrient_subs}

\noindent
We establish Theorem~\ref{RelSpinOrient_thm} and similar  statements 
by relating the orientations arising from Corollary~\ref{canonisom_crl2a}
to the orienting procedure for the determinants of Fredholm operators over
oriented symmetric half-surfaces described in~\cite{XCapsSetup}.\\

\noindent
An \sf{oriented symmetric half-surface} (or simply \sf{oriented sh-surface}) 
is a pair $(\Si^b,c)$ consisting of an oriented bordered smooth surface~$\Si^b$ 
and an involution $c\!:\prt\Si^b\!\lra\!\prt\Si^b$ preserving each component
and the orientation of~$\prt\Si^b$.
The restriction of~$c$  to a boundary component $(\prt\Si^b)_i$
is either the identity or the antipodal map
$$\fa\!:S^1\lra S^1, \qquad z\lra -z,$$
for a suitable identification of $(\prt\Si^b)_i$ with $S^1\!\subset\!\C$;
the latter type of boundary structure is called \sf{crosscap} 
in the string theory literature.
We denote~by 
$$\prt_0^c\Si^b,\prt_1^c\Si^b\subset \prt\Si^b\subset \Si^b$$ 
the unions of the standard boundary
components of~$(\Si^b,c)$ and of the crosscaps, respectively.
If $\prt_1^c\Si^b\!=\!\eset$, $(\Si^b,c)$ is a bordered surface in the usual sense.
An  oriented sh-surface $(\Si^b,c)$  \sf{doubles} 
to a symmetric surface~$(\Si, \si)$
so that $\si$ restricts to~$c$ on the cutting circles (the boundary of~$\Si^b$);
see \cite[(1.6)]{XCapsSetup}.
In particular, $\Si^{\si}\!=\!\prt_0^c\Si^b$.
Since this doubling construction covers all topological types of orientation-reversing 
involutions $\si$ on $\Si$, for every symmetric surface $(\Si,\si)$ there is 
an oriented sh-surface $(\Si^b,c)$ which doubles to~$(\Si,\si)$.
\\

\noindent
A \textsf{real bundle pair} $(V^b,\wt{c})$ over an oriented sh-surface~$(\Si^b,c)$
consists of a complex vector bundle $V^b\!\lra\!\Si^b$ with a conjugation~$\wt{c}$
on $V^b|_{\prt\Si^b}$ lifting~$c$.
Via the doubling construction after \cite[Remark~3.4]{XCapsSetup},
such a pair $(V^b,\wt{c})$ 
corresponds to a real bundle pair~$(V,\vph)$ over the associated symmetric surface~$(\Si,\si)$
so that $V^b\!=\!V|_{\Si^b}$ and~$\wt{c}$ is the restriction of~$\vph$ to~$V^b|_{\prt\Si^b}$.
In particular,
$$V^{\vph}=\big(V^b\big)^{\wt{c}}\subset V|_{\Si^{\si}}=V^b\big|_{\prt_0^c\Si^b}$$
is a totally real subbundle.\\

\noindent
By \cite[Lemma~2.4]{Teh}, the homotopy classes of trivializations of the real bundle pair~$(V,\vph)$
over $\prt_1^c\Si^b$ correspond to 
the homotopy classes of trivializations of its top exterior power $\La^{\top}_{\C}(V,\vph)$.
If $(L,\wt\phi)$ is a rank~1 real bundle pair over~$(\Si,\si)$, 
the real bundle pair $2(L,\wt\phi)$ has a canonical homotopy class of trivializations
over~$\prt_1^c\Si^b$;
see the proof of \cite[Theorem~1.3]{Teh}.
Thus, a homotopy class of trivializations of~$(V,\vph)$ over~$\prt_1^c\Si^b$
corresponds to a homotopy  of trivializations of~$(V,\vph)\!\oplus\!2(L,\wt\phi)$.
Furthermore, a homotopy class of isomorphisms of real bundle pairs as in~\eref{realorient_e4} 
determines a homotopy class of trivializations
of the restriction of~$(V,\vph)$ to~$\prt_1^c\Si^b$.
It also induces an orientation on 
the real vector bundle $V^{\vph}\!\lra\!\prt_0^c\Si^b$.\\

\noindent
If the real vector bundle $V^{\vph}\!\lra\!\prt_0^c\Si^b$ is oriented,  
a relative spin structure on~$V^{\vph}$
consists of an oriented vector bundle $L\!\lra\!\Si$
and a homotopy class of trivializations of the oriented vector~bundle 
\BE{relspintriv_e} V^{\vph}\oplus L|_{\prt_0^c\Si^b}\lra \prt_0^c\Si^b=\Si^{\si}\,.\EE
Since every oriented vector bundle over~$\Si^b$ is trivializable, 
the vector bundle $L|_{\Si^b}\!\lra\!\Si^b$ admits a trivialization~$\Psi_L^b$.
Along with a trivialization of~\eref{relspintriv_e}, 
the restriction of~$\Psi_L^b$ to $L|_{\prt_0^c\Si^b}$ induces a trivialization of~$V^{\vph}$.
If  $\prt_1^c\Si^b\!=\!\eset$ and the rank~$n$ of~$V$ is at least~3, 
the homotopy classes of the trivializations of $V^{\vph}$ induced
by two trivializations of $L|_{\Si^b}$ differ on an even number of
components of~$\prt_0^c\Si^b\!=\!\prt\Si^b$.\\

\noindent
A \textsf{real CR-operator} on a real bundle pair~$(V^b,\ti{c})$ 
over an  oriented sh-surface~$(\Si^b,c)$ is a linear map of the~form
\begin{equation*}\begin{split}
D^b=\bp^b\!+\!A\!: \Ga(\Si^b;V^b)^{\wt{c}}
\equiv&\big\{\xi\!\in\!\Ga(\Si^b;V^b)\!:\,\xi\!\circ\!c\!=\!\wt{c}\!\circ\!\xi|_{\prt\Si^b}\big\}\\
&\hspace{.5in}\lra
\Ga_{\fJ^b}^{0,1}(\Si^b;V^b)\equiv\Ga\big(\Si^b;(T^*\Si^b,\fJ^b)^{0,1}\!\otimes_{\C}\!V^b\big),
\end{split}\end{equation*}
where $\bp^b$ is the holomorphic $\bp$-operator for some complex structure~$\fJ^b$ on~$\Si^b$
and holomorphic structure in~$V^b$ and  
$$A\in\Ga\big(\Si^b;\Hom_{\R}(V^b,(T^*\Si^b,\fJ^b)^{0,1}\!\otimes_{\C}\!V^b) \big)$$ 
is a zeroth-order deformation term. 
By \cite[Corollary~3.3]{XCapsSetup},  $\fJ^b$ doubles to some $\fJ\!\in\!\cJ_{\Si}^{\si}$
if and only if $c$ is real-analytic with respect to~$\fJ^b$.
In such a case, $D^b$ is Fredholm in appropriate completions
and corresponds to a real CR-operator~$D$ on the associated real bundle 
pair~$(V,\vph)$ over~$(\Si,\si)$; see \cite[Proposition~3.6]{XCapsSetup}.
In particular, there is a canonical isomorphism
\BE{bdnisom_e}
\rdet\,D\equiv \big(\!\det D\big)\otimes \big(\!\det\dbar_{\Si;\C}\big)^{\otimes n}
\approx \big(\!\det D^b\big)\otimes \big(\!\det\dbar_{\Si;\C}^b\big)^{\otimes n}
\equiv \rdet\,D^b\,, \EE
where $n\!=\!\rk_{\C}V$, $\dbar_{\Si;\C}$ is the standard real CR-operator
on the trivial real bundle pair $(\Si\!\times\!\C,\si\!\times\!\fc)$ over~$(\Si,\si)$
as in Example~\ref{ex_tbdl}, and
$\dbar_{\Si;\C}^b\!\equiv\!\dbar_{\Si^b;\C}$ is the standard real CR-operator
on the trivial relative bundle pair $(\Si^b\!\times\!\C,c\!\times\!\fc)$ over~$(\Si^b,c)$.\\

\noindent
An orientation on the right-hand side of~\eref{bdnisom_e} thus determines an orientation 
on the left-hand side of~\eref{bdnisom_e}.
By the proofs of \cite[Lemma~6.37]{Melissa} and \cite[Theorem~1.1]{XCapsSetup}, 
an orientation on the former is determined by a collection consisting~of
\begin{enumerate}[label=(OC\arabic*),leftmargin=*]

\item\label{SpinTriv_it} a homotopy class of trivializations of~$V^{\vph}$ 
over~$\prt_0^c\Si^b$;

\item\label{MohTriv_it} 
a homotopy class of trivializations of the real bundle pair $(V,\vph)$
over~$\prt_1^c\Si^b$.

\end{enumerate}
If $n\!\ge\!3$, changing the homotopy class in~\ref{SpinTriv_it} within its orientation 
class over precisely one topological component of~$\prt_0^c\Si^b$ changes 
the induced orientation  on the right-hand side of~\eref{bdnisom_e}.
Changing the homotopy class in~\ref{MohTriv_it} 
class over precisely one topological component of~$\prt_1^c\Si^b$ also changes 
the induced orientation  on the right-hand side of~\eref{bdnisom_e}.\\

\noindent
Let $(L,\wt\phi)$ be a rank~1 real bundle pair over~$(\Si,\si)$ and 
$D_L$ be a real CR-operator on~$(L,\wt\phi)$.
By the sentence above containing~\ref{SpinTriv_it} and~\ref{MohTriv_it}
applied with $V$ replaced by $V\!\oplus\!2L$,
an orientation~on
\BE{staborientbd_e}\big(\!\det (D^b\!\oplus\!D_{2L}^b)\big)\otimes 
\big(\!\det\dbar_{\Si;\C}^b\big)^{\otimes (n+2)}
\approx \big(\!\det D^b\big)\!\otimes\! \big(\!\det\dbar_{\Si;\C}^b\big)^{\otimes n}
\otimes
\big(\!\det D_L^b\big)^{\otimes2}\!\otimes\! \big(\!\det\dbar_{\Si;\C}^b\big)^{\otimes 2}\EE
is determined by a trivialization~$\psi_{V\oplus2L}$ of the real vector bundle 
$V^{\vph}\!\oplus\!2L^{\wt\phi}$ over~$\prt_0^c\Si^b$ 
and  a trivialization $\psi_{V\oplus2L}'$ of the rank~1 real bundle pair
$(V\!\oplus\!2L,\vph\!\oplus\!2\wt\phi)$
 over~$\prt_1^c\Si^b$.
Since the last two factors in~\eref{staborientbd_e} are canonically oriented,
$\psi_{V\oplus2L}$ and $\psi_{V\oplus2L}'$ thus determine 
an orientation on the right-hand side of~\eref{bdnisom_e}.
We will call it the \sf{stabilization orientation induced~by} $\psi_{V\oplus2L}$
and $\psi_{V\oplus2L}'$, omitting $\psi_{V\oplus2L}'$ if $\prt_1^c\Si^b\!=\!\eset$
and $\psi_{V\oplus2L}$ if $\prt_0^c\Si^b\!=\!\eset$.\\

\noindent
Via~\eref{splitLisom_e} with $L^*$ replaced by~$L$,  
$\psi_{V\oplus2L}$ also induces a trivialization of~\eref{relspintriv_e}.
If $\prt_1^c\Si^b\!\!=\!\eset$,   
$\psi_{V\oplus2L}$ thus determines a relative spin structure on~$V^{\vph}$,
and another orientation on the right-hand side of~\eref{bdnisom_e}.
We will call the latter
the \sf{associated relative spin} (or simply \sf{ARS}) \sf{orientation}.
If $L^{\wt\phi}\!\lra\!\Si^{\si}$ is orientable (but $\prt_1^c\Si^b$ is not necessarily empty), 
then
\begin{enumerate}[label=$\bu$,leftmargin=*]

\item $\psi_{V\oplus2L}$ and the canonical homotopy class of trivializations of $2L^{\wt\phi}$ 
determine a homotopy class of trivializations of $V^{\vph}$ over $\prt_0^c\Si^b$, and

\item $\psi_{V\oplus2L}'$ and the canonical homotopy class of trivializations of
 $2(L,\wt\phi)$ determine a homotopy class of trivializations of $(V,\vph)$
over $\prt_1^c\Si^b$.

\end{enumerate}
Thus, $\psi_{V\oplus2L}$ and $\psi_{V\oplus2L}'$ determine
another orientation on the right-hand side of~\eref{bdnisom_e} in this case;
we will call it the \sf{associated spin} (or simply \sf{AS}) \sf{orientation}.
Lemmas~\ref{RelSpinOrient_lmm1}-\ref{RelSpinOrient_lmm5} and 
Corollary~\ref{RelSpinOrient_crl} below compare these three orientations
on the right-hand side of~\eref{bdnisom_e}.\\

\noindent
In the case of the involutions 
$$\tau\!:\P^1\lra\P^1, \quad z\lra 1/\bar{z}, \qquad\hbox{and}\qquad
\eta\!:\P^1\lra\P^1, \quad z\lra -1/\bar{z},$$
we can take $\Si^b$ to be the unit disk around the origin in $\C\!\subset\!\P^1$.
This will be our default choice in these settings.

\begin{lmm}\label{RelSpinOrient_lmm1}
With notation as above, suppose $(\Si,\si)\!=\!(\P^1,\tau)$.
If $L^{\wt\phi}\!\lra\!S^1$ is orientable, 
the stabilization and  AS orientations on the right-hand side of~\eref{bdnisom_e} 
induced by a trivialization~$\psi_{V\oplus2L}$ of $V^{\vph}\!\oplus\!2L^{\wt\phi}$  are the same.
\end{lmm}

\begin{proof}
Fix a trivialization $\psi_L\!:L^{\wt\phi}\!\lra\!S^1\!\times\!\R$; 
the canonical homotopy class of trivializations of  $2L^{\wt\phi}$ is 
the class containing~$2\psi_L$.
A trivialization~$\psi_V$ of~$V^{\vph}$ lies in the associated
homotopy class of trivializations of~$V^{\vph}$ if and only if 
$\psi_{V\oplus2L}$ and $\psi_V\!\oplus\!2\psi_L$ lie in 
the same homotopy class  of trivializations of $V^{\vph}\!\oplus\!2L^{\wt\phi}$.
In this case, the natural isomorphism~\eref{staborientbd_e} is orientation-preserving
with respect to the orientation on the left-hand side induced by~$\psi_{V\oplus2L}$
and the orientations~on
\BE{RelSpinOrient1_e3}\big(\!\det D^b\big)\!\otimes\!
\big(\!\det\dbar_{\Si;\C}^b\big)^{\otimes n}
\qquad\hbox{and}\qquad
\big((\!\det D_L^b)\!\otimes\!(\!\det\dbar_{\Si;\C}^b)\big)^{\otimes 2}\EE
induced by~$\psi_V$ and $\psi_L$, respectively.
Since the last orientation is the same as the orientation induced by the canonical orientations
of $(\det D_L^b)^{\otimes2}$ and $(\det\dbar_{\Si;\C}^b)^{\otimes 2}$, 
the stabilization  orientation on the first tensor product in~\eref{RelSpinOrient1_e3} induced 
by~$\psi_{V\oplus2L}$ and
the AS orientation (i.e.~the orientation induced by~$\psi_V$) are the~same. 
\end{proof}

\begin{lmm}\label{RelSpinOrient_lmm2}
With notation as above, suppose $(\Si,\si)\!=\!(\P^1,\tau)$.
If $L^{\wt\phi}\!\lra\!S^1$ is orientable,
the stabilization and  ARS orientations on the right-hand side of~\eref{bdnisom_e} 
induced by a trivialization~$\psi_{V\oplus2L}$ of $V^{\vph}\!\oplus\!2L^{\wt\phi}$  
are the same if and only if $\deg L\!\in\!4\Z$.
\end{lmm}

\begin{proof}
Let $d\!=\!\deg L$. 
By \cite[Proposition~4.1]{BHH}, we can assume that $(L,\wt\phi)$ is the holomorphic line
$\cO_{\P^1}(d)$ with the standard lift of~$\tau$.
Since $L^{\wt\phi}\!\lra\!S^1$ is orientable, $d\!\in\!2\Z$.
By \cite[Theorem~C.3.6]{MS}, there exists a trivialization~$\Psi_L^b$
of $L|_{\Si^b}$ so~that
\BE{RelSpinOrient2_e3}\Psi_L^b\big(L^{\wt\phi}\big)=\big\{\big(\ne^{\fI\th},a\ne^{\fI d\th/2}\big)\!:
\,\ne^{\fI\th}\!\in\!S^1,\,a\!\in\!\R\big\}\subset S^1\!\times\!\C\,.\EE
Let $\psi_L$ be the trivialization of $L^{\wt\phi}$ given~by
$$\psi_L\big(\{\Psi_L^b\}^{-1}(\ne^{\fI\th},a\ne^{\fI d\th/2})\big)
=\big(\ne^{\fI\th},a\big)\in S^1\!\times\!\R\,.$$
The trivialization $\Psi_L^{\prt}$ of $L|_{\prt\Si^b}$ induced by $2\psi_L$ 
via~\eref{splitLisom_e} with $L^*$ replaced by~$L$ is then described~by
\BE{RelSpinOrient2_e7}\Psi_L^{\prt}\!:L|_{\prt\Si^b}\lra S^1\!\times\!\C, \quad
\Psi_L^{\prt}\big(\{\Psi_L^b\}^{-1}(\ne^{\fI\th},c)\big)
=\big(\ne^{\fI\th},c\ne^{-\fI d\th/2}\big)  
~~~\forall\,\big(\ne^{\fI\th},c\big)\in S^1\!\times\!\C.\EE
Thus, the homotopy classes of $\Psi_L^{\prt}$ and $\Psi_L^b|_{\prt\Si^b}$
differ by $d/2$ times a generator of $\pi_1(\SO(2))\!\approx\!\Z$.\\

\noindent
Let $\psi_V$ and $\psi_V'$ be trivializations of~$V^{\vph}$ such that 
$\psi_V\!\oplus\!2\psi_L$ and $\psi_V'\!\oplus\!\Psi_L^b|_{\prt\Si^b}$
lie in the same homotopy class of trivializations of $V^{\vph}\!\oplus\!2L^{\wt\phi}$
as~$\psi_{V\oplus2L}$.
By Lemma~\ref{RelSpinOrient_lmm1}, the stabilization orientation 
on the right-hand side of~\eref{bdnisom_e} 
induced by~$\psi_{V\oplus2L}$ 
is the orientation induced by~$\psi_V$ as in the proof of \cite[Theorem~8.1.1]{FOOO}. 
By definition, the ARS orientation on the right-hand side of~\eref{bdnisom_e} 
induced by~$\psi_{V\oplus2L}$ 
is the orientation induced by~$\psi_V'$.
By~\eref{RelSpinOrient2_e7}, $\psi_V$ and $\psi_V'$ are homotopic
(and thus the two induced orientations are the same) if and only if $d/2\!\in\!2\Z$.
\end{proof}

\noindent
Suppose $(\Si,\si)\!=\!(\P^1,\tau)$ and $\deg L\!=\!1$.
Similarly to the proof of Lemma~\ref{RelSpinOrient_lmm2}, 
\cite[Proposition~4.1]{BHH} and \cite[Theorem~C.3.6]{MS} imply that 
there exists a trivialization~$\Psi_L^b$ of $L|_{\Si^b}$ so~that
\eref{RelSpinOrient2_e3} holds with $d\!=\!1$.
Let $\psi_0$ be the trivialization of $2L^{\wt\phi}$ given~by
\BE{psi0dfn_e}\psi_0\big(\{\Psi_L^b\}^{-1}(\ne^{\fI\th},a_1\ne^{\fI\th/2}),
\{\Psi_L^b\}^{-1}(\ne^{\fI\th},a_2\ne^{\fI\th/2})\big)
=\big(\ne^{\fI\th},(a_1\!+\!\fI a_2)\ne^{\fI\th/2}\big)\in S^1\!\times\!\C\EE
for all $a_1,a_2\!\in\!\R$.

\begin{prp}\label{RelSpinOrient_prp}
The orientation on $\det D_{2L}^b\!=\!(\det D_{L}^b)^{\otimes2}$
induced by the trivialization~$\psi_0$ as in the proof of \cite[Theorem~8.1.1]{FOOO}
agrees with the canonical square orientation.
\end{prp}

\noindent
We give three proofs.
In the first one, we write out the real holomorphic sections  
and the relevant trivializations explicitly.
In the second proof, we  use the comparisons of different orientations
on the moduli spaces  of real lines obtained in~\cite{Teh}.
The last argument deduces the claim directly from the fixed-edge equivariant
contribution determined in~\cite{Teh}. 
In all three arguments,
we take $D_L$ to be  the standard $\bar\partial$-operator in $\cO_{\P^1}(1)$.

\begin{proof}[{\bf\emph{Proof~1}}]
Let $\P^1_{\bu}\!=\!\P^1\!-\!\{1\}$.
The holomorphic map
$$h\!: B\!\equiv\!\big\{t\!\in\!\C\!:\,|t|\!<\!1\big\}\lra\P^1, \qquad
t\lra \ne^{\fI t}\,,$$
is injective and intertwines the standard conjugation on~$B$ with $\tau$ on~$\P^1$.
We can assume~that 
\begin{gather*}
L =\big(h(B)\!\times\!\C\sqcup \P_{\bu}^1\!\times\!\C\big)\big/\!\!\sim, \quad
\big(h(t),tc\big)\sim\big(h(t),c\big)~~\forall~(t,c)\!\in\!(B\!-\!0)\!\times\!\C,\\
\wt\phi\big([t,c]\big)=\big[\bar{t},\bar{c}\big]~~\forall~(t,c)\!\in\!B\!\times\!\C, \quad
\wt\phi\big([z,c]\big)=\big[\tau(z),\bar{c}\big]~~\forall~(z,c)\!\in\!\P_{\bu}^1\!\times\!\C.
\end{gather*}
The space of real holomorphic sections of~$L$ is then generated~by the sections~$s_1$ and~$s_2$
described~by
$$s_1\big([z]\big)=1, \quad s_2\big([z]\big)=\fI\frac{1\!+\!z}{1\!-\!z}
\qquad\forall~z\!\in\!\P_{\bu}^1.$$ 
The canonical orientation for $\det D_{2L}^b$ is then determined by the basis 
$$s_{11}\equiv (s_1,0), \quad s_{12}\equiv (s_2,0), 
\quad s_{21}\equiv (0,s_1),\quad s_{22}\equiv (0,s_2),$$
for the kernel of the surjective operator~$D_{2L}^b$.\\

\noindent 
We define a trivialization $\Psi_L^b$ of $L$ over the unit disk~$\Si^b$ 
around $z\!=\!0$ in $\C\!\subset\!\P^1$ by
\begin{alignat*}{2}
\Psi_L^b\big([h(t),c]\big)&=\big(\ne^{\fI t},2\fI\frac{\ne^{\fI t}\!-\!1}{t}c\big) &\qquad
&\forall~(t,c)\!\in\!B\!\times\!\C, \\
\Psi_L^b\big([z,c]\big)&=\big(z,2\fI(z\!-\!1)c\big) &\qquad
&\forall~(z,c)\!\in\!\big(\P^1_{\bu}\!-\!\{\i\}\big)\!\times\!\C.
\end{alignat*}
This trivialization satisfies~\eref{RelSpinOrient2_e3} with $d\!=\!1$.
The trivialization $\psi_0$ of $2L^{\wt\phi}$ over~$S^1$ extends to the trivialization
\begin{gather*}
\Psi_0\!:2L|_{\P^1-\{0,\i\}}\lra \big(\P^1\!-\!\{0,\i\}\big)\times\C^2, \\
\Psi_0\big([z,c_1],[z,c_2]\big)=\big(z,\fI(z\!-\!z^{-1})c_1\!-\!z^{-1}(1\!-\!z)^2c_2,
z^{-1}(1\!-\!z)^2c_1\!+\!\fI(z\!-\!z^{-1})c_2\big)\,.
\end{gather*}
This trivialization intertwines $2\wt\phi$ 
with the standard lift of $\tau|_{\P_{\bu}^1-\{0,\i\}}$ to a conjugation on 
the trivial bundle
\hbox{$(\P_{\bu}^1\!-\!\{0,\i\})\!\times\!\C^2$}.\\

\noindent
We note that
\begin{alignat*}{2}
\big\{\Psi_0s_{11}\big\}(z)&=\big(\fI z^{-1}(z^2\!-\!1),z^{-1}(1\!-\!z)^2\big),
&\quad
\big\{\Psi_0s_{12}\big\}(z)&=\big(z^{-1}(1\!+\!z)^2,\fI z^{-1}(1\!-\!z^2)\big),\\
\big\{\Psi_0s_{21}\big\}(z)&=\big(-\!z^{-1}(1\!-\!z)^2,\fI z^{-1}(z^2\!-\!1)\big),
&\quad
\big\{\Psi_0s_{22}\big\}(z)&=\big(-\!\fI z^{-1}(1\!-\!z^2),z^{-1}(1\!+\!z)^2\big).
\end{alignat*}
The orientation on $\det D_{2L}^b$ induced by the trivialization~$\psi_0$ is obtained from
the isomorphism
$$\ker D_{2L}^b\lra \R\!\oplus\!\R\oplus\big\{\Res_{z=0}(\Psi_0\xi)\!:\,
\xi\!\in\!\ker D_{2L}^b\big\}, \quad
\xi\lra \big(\{\Psi_0\xi\}(1),\Res_{z=0}(\Psi_0\xi)\big).$$
The last space above is a complex subspace of~$\C^2$.
Under this isomorphism, the basis $s_{11},s_{12},s_{21},s_{22}$ is sent~to
$$(0,0;-\fI,1), \quad (4,0;1,\fI), \quad (0,0;-1,-\fI), \quad
(0,4;-\fI,1).$$
Thus, an oriented basis for the target of the above isomorphism is given~by
$$(4,0;0,0), \quad (0,4;0,0), \quad (0,0;-\fI,1), \quad (0,0;1,\fI).$$
The change of basis matrix from the first basis to this one is given~by
$$\left( \begin{array}{cccc} 0& 1& 0& 0\\ 0& 0& 0& 1\\ 1& 0& 0& 1\\
0& 1& -1& 0\end{array}\right).$$
The determinant of this matrix is $+1$.
\end{proof}

\begin{proof}[{\bf\emph{Proof~2}}]
Define 
\begin{gather*}
\tau_3\!:\P^3\lra\P^3, \qquad 
[Z_1,Z_2,Z_3,Z_4]\lra\big[\ov{Z}_2,\ov{Z}_1,\ov{Z}_4,\ov{Z}_3\big],\\
 \fM_1(\P^1) =\fM_1(\P^1,1)^{\tau,\tau},\qquad
\fM_1(\P^3)=\fM_1(\P^3,1)^{\tau_3,\tau}.
\end{gather*} 
The inclusion $\io\!:\P^1\!\lra\!\P^3$ as the first two coordinates induces
an embedding of $\fM_1(\P^1)$ into~$\fM_1(\P^3)$.
Let
$$\cN_{\io(0)}\P=\frac{T_{\io(0)}\P^3}{T_{\io(0)}\P^1} \qquad\hbox{and}\qquad
\cN_{[\io,0]}\fM \equiv \frac{T_{[\io,0]}\fM_1(\P^3)}{T_{[\io,0]}\fM_1(\P^1)}$$ 
denote the normal bundle of $\P^1$ in~$\P^3$ at $[1,0,0,0]$ and
the normal bundle of $\fM_1(\P^1)$ in~$\fM_1(\P^3)$ at~$\io$ with the positive marked point at $z\!=\!0$,
respectively.
The former is a complex vector space and thus is canonically oriented.
The differential of the evaluation~map~$\ev_1$ induces an isomorphism
\BE{pf2_e7}\tnd_{[\io,0]}\ev_1\!:\,\cN_{[\io,0]}\fM\lra \cN_{\io(0)}\P\,.\EE
By \cite[Lemma~5.3]{Teh}, this isomorphism is orientation-{\it reversing} 
with respect to the algebraic orientations on $\fM_1(\P^1)$ in~$\fM_1(\P^3)$
defined in \cite[Section~5.2]{Teh}.\\

\noindent
Since the normal bundle of $(\P^1,\tau)$ in $(\P^3,\tau_3)$ is isomorphic to $2(L,\wt\phi)$,
the composition
$$ \ker D_{2L}^b\lra T_{[\io,0]}\fM_1(\P^3)\lra \cN_{[\io,0]}\fM$$
is an isomorphism.
Combining it with~\eref{pf2_e7}, we obtain an isomorphism
\BE{pf2_e9} \ker D_{2L}^b\lra \cN_{[\io,0]}\fM \lra \cN_{\io(0)}\P\,.\EE
Since the canonical orientation on $\det D_{2L}^b$ is obtained from the isomorphism
$$ \ker D_{2L}^b\lra 2L_0, \qquad \xi\lra \xi(0),$$
and the complex orientation on~$L_0$, the isomorphism~\eref{pf2_e9}
is orientation-{\it preserving} with respect to   the canonical orientation on 
the left-hand side.\\

\noindent
The real vector bundle
\BE{pf2_e11} 4L^{\wt\phi}\lra S^1\!=\!\R\P^1\subset\P^1\EE
carries a canonical spin structure; see \cite[Section~5.5]{Teh}.
Along with Euler's sequence for~$\P^3$, it determines an orientation on~$\fM_1(\P^3)$;
we will call it the \textsf{spin orientation}.
It agrees with the orientation induced by the trivialization~$2\psi_0$ over~$S^1$.
Along with Euler's sequence for~$\P^1$ and the relative spin orienting procedure
of \cite[Theorem~8.1.1]{FOOO}, the canonical spin structure on~\eref{pf2_e11}
determines an orientation on~$\fM_1(\P^1)$;
we will call it the \textsf{relative spin orientation}.
Along with the spin orientation on $\fM_1(\P^3)$, 
it induces an orientation on~$\cN_{[\io,0]}\fM$;
we will call it the \textsf{spin orientation}.
Since~$\psi_0$ extends over the disk $\Si^b\!\subset\!\P^1$,
the first isomorphism in~\eref{pf2_e9} is orientation-{\it preserving} with respect
to the orientation on the left-hand side induced by~$\psi_0$
and the spin orientation on~$\cN_{[\io,0]}\fM$.\\

\noindent
As summarized in the paragraph above \cite[Remark~6.9]{Teh}, 
the algebraic orientations on~$\fM_1(\P^1)$ and~$\fM_1(\P^3)$  
are the same as the relative spin orientation and 
the opposite of the spin orientation, respectively.
Therefore,  the spin orientation on~$\cN_{[\io,0]}\fM$ is the opposite of
the algebraic orientation.
Since the second isomorphism in~\eref{pf2_e9} is orientation-reversing
with respect to the latter,
it follows that the composite isomorphism in~\eref{pf2_e9} is orientation-preserving
with respect
to the orientation on the left-hand side induced by~$\psi_0$.
Since this is also the case with respect to the canonical orientation on the left-hand side,
these two orientations on~$\ker D_{2L}^b$ agree.  
\end{proof}

\begin{proof}[{\bf\emph{Proof~3}}]
Under a change of coordinate on $2(L,\wt\phi)$ which is homotopic to the identity, 
the trivialization~$\psi_0$ is equivalent to the trivialization \cite[(6.13)]{Teh}.
By  \cite[Section~6.4]{Teh}, there are natural $S^1$-actions on $(\P^1,\tau)$ and 
$2(L,\wt\phi)$ so~that the evaluation isomorphism
\BE{pf3_e3}\ker D_{2L}^b \stackrel{\ev_{2L;0}}{\lra}  2L|_0, \qquad
\xi\lra\xi(0),\EE
is $S^1$-equivariant. 
By the $d_0\!=\!1$, $\bar{i}\!\in\!2\Z$ case of \cite[(6.21)]{Teh},
the $S^1$-equivariant Euler class of $\ker D_{2L}^b$
with respect to the orientation induced by~$\psi_0$ is given~by
$$\be\big(\ker D_{2L}^b\big)
=-\big(\la_i\!-\!\la_j\big)\big(-\la_i\!-\!\la_j\big)
=\big(\la_i\!-\!\la_j\big)\big(\la_i\!+\!\la_j\big)
=\be(2L|_0)\,.$$
This establishes the claim.
\end{proof}

\begin{crl}\label{RelSpinOrient_crl0}
With notation as above, suppose $(\Si,\si)\!=\!(\P^1,\tau)$.
If $L^{\wt\phi}\!\lra\!S^1$ is not orientable,
the stabilization and  ARS orientations on the right-hand side of~\eref{bdnisom_e} 
induced by a trivialization~$\psi_{V\oplus2L}$ of $V^{\vph}\!\oplus\!2L^{\wt\phi}$  
 are the same if and only if $\deg L\!-\!1\!\in\!4\Z$.
\end{crl}

\begin{proof}
Let $d\!=\!\deg L$. 
Since $L^{\wt\phi}\!\lra\!S^1$ is not orientable, $d\!\not\in\!2\Z$.
Similarly to the proof of Lemma~\ref{RelSpinOrient_lmm2}, 
\cite[Proposition~4.1]{BHH} and \cite[Theorem~C.3.6]{MS} imply that 
there exists a trivialization~$\Psi_L^b$ of $L|_{\Si^b}$ so~that
\eref{RelSpinOrient2_e3} holds.
Let $\psi_{2L}$ be the trivialization of $2L^{\wt\phi}$ given~by
$$\psi_{2L}\big(\{\Psi_L^b\}^{-1}(\ne^{\fI\th},a_1\ne^{\fI d\th/2}),
\{\Psi_L^b\}^{-1}(\ne^{\fI\th},a_2\ne^{\fI d\th/2})\big)
=\psi_0\big((\ne^{\fI\th},a_1\ne^{\fI\th/2}),(\ne^{\fI\th},a_2\ne^{\fI\th/2})\big)
\in S^1\!\times\!\C\,.$$
The trivialization $\Psi_L^{\prt}$ of $L|_{\prt\Si^b}$ induced by $\psi_{2L}$ 
via~\eref{splitLisom_e} with $L^*$ replaced by~$L$ is then described~by
\BE{RelSpinOrient4_e7}\Psi_L^{\prt}\!:L|_{\prt\Si^b}\lra S^1\!\times\!\C, \quad
\Psi_L^{\prt}\big(\{\Psi_L^b\}^{-1}(\ne^{\fI\th},c)\big)
=\big(\ne^{\fI\th},c\ne^{-\fI (d-1)\th/2}\big)  
~~~\forall\,\big(\ne^{\fI\th},c\big)\in S^1\!\times\!\C.\EE
Thus, the homotopy classes of $\Psi_L^{\prt}$ and $\Psi_L^b|_{\prt\Si^b}$
differ by $(d\!-\!1)/2$ times a generator of $\pi_1(\SO(2))\!\approx\!\Z$.\\

\noindent
Let $\psi_V$ and $\psi_V'$ be trivializations of~$V^{\vph}$ such that 
$\psi_V\!\oplus\!\psi_{2L}$ and $\psi_V'\!\oplus\!\Psi_L^b|_{\prt\Si^b}$
lie in the same homotopy class of trivializations of $V^{\vph}\!\oplus\!2L^{\wt\phi}$
as~$\psi_{V\oplus2L}$.
By Proposition~\ref{RelSpinOrient_prp}, the stabilization orientation 
on the right-hand side of~\eref{bdnisom_e} 
induced by~$\psi_{V\oplus2L}$ via the isomorphism~\eref{staborientbd_e} 
is the orientation induced by~$\psi_V$ 
as in the proof of \cite[Theorem~8.1.1]{FOOO}. 
By definition, the ARS orientation on the right-hand side of~\eref{bdnisom_e} 
induced by~$\psi_{V\oplus2L}$ 
is the orientation induced by~$\psi_V'$.
By~\eref{RelSpinOrient4_e7}, $\psi_V$ and $\psi_V'$ are homotopic
(and thus the two induced orientations are the same) if and only if $(d\!-\!1)/2\!\in\!2\Z$.
\end{proof}

\begin{lmm}\label{RelSpinOrient_lmm5}
With notation as above, suppose $(\Si,\si)\!=\!(\P^1,\eta)$.
The stabilization and  AS orientations on the right-hand side of~\eref{bdnisom_e} 
induced by a trivialization~$\psi_{V\oplus2L}'$ of $(V\!\oplus\!2L,\vph\!\oplus\!2\wt\phi)$  
are the~same.
\end{lmm}

\begin{proof}
Fix a trivialization $\psi_L'$ of $(L,\wt\phi)$ over $(S^1,\fa)$; 
the canonical homotopy class of trivializations of  $2(L,\wt\phi)$ is 
the class containing~$2\psi_L'$.
A trivialization~$\psi_V'$ of~$(V,\vph)$ over~$(S^1,\fa)$ lies in the associated
homotopy class of trivializations of~$(V,\vph)$ over~$(S^1,\fa)$ if and only if 
$\psi_{V\oplus2L}'$ and $\psi_V'\!\oplus\!2\psi_L'$ lie in 
the same homotopy class  of trivializations of $(V\!\oplus\!2L,\vph\!\oplus\!2\wt\phi)$
over~$(S^1,\fa)$.
In this case, the natural isomorphism~\eref{staborientbd_e} is orientation-preserving
with respect to the orientation on the left-hand side induced by~$\psi_{V\oplus2L}'$
and the orientations on~\eref{RelSpinOrient1_e3}
induced by~$\psi_V'$ and $\psi_L'$, respectively.
Since the last orientation is the same as the orientation induced by the canonical orientations
of $(\det D_L^b)^{\otimes2}$ and $(\det\dbar_{\C}^b)^{\otimes 2}$, 
the stabilization  orientation on the first tensor product in~\eref{RelSpinOrient1_e3} induced 
by~$\psi_{V\oplus2L}'$ and
the AS orientation (i.e.~the orientation induced by~$\psi_V'$) are the~same. 
\end{proof}

\begin{crl}\label{RelSpinOrient_crl}
Let $(\Si^b,c)$, $(\Si,\si)$, $(V,\vph)$, $(L,\wt\phi)$, $D$, and~$D^b$  
be as above Lemma~\ref{RelSpinOrient_lmm1}.
\begin{enumerate}[label=(\arabic*),leftmargin=*]

\item\label{RelSpinorient_it}
If $\prt_1^c\Si^b\!=\!\eset$ and $(\prt\Si^b)_1,\ldots,(\prt\Si^b)_m$ are
the components of $\prt_0^c\Si^b\!=\!\prt\Si^b$, then
the stabilization and ARS orientations on the right-hand side of~\eref{bdnisom_e} 
induced by a trivialization of $V^{\vph}\!\oplus\!2L^{\wt\phi}$
are the same if and only~if 
$$\deg L-\big|\big\{i\!=\!1,\ldots,m\!:\,w_1(L^{\wt\phi})|_{(\prt\Si^b)_i}\!\neq\!0\big\}\big|
\in4\Z.$$

\item\label{RSOorient_it}
If $L^{\wt\phi}\!\lra\!\prt_0^c\Si^b$ is orientable, then
the stabilization and AS orientations   on the right-hand side of~\eref{bdnisom_e} 
induced by a trivialization of $V^{\vph}\!\oplus\!2L^{\wt\phi}$ and 
a trivialization of $(V\!\oplus\!2L,\vph\!\oplus\!2\wt\phi)|_{\prt_1^c\Si^b}$
are the same.

\end{enumerate}
\end{crl}

\begin{proof} For each $i\!=\!1,\ldots,m$, let
$$\ve_i(L)=\begin{cases}0,&\hbox{if}~w_1(L^{\wt\phi})|_{(\prt\Si^b)_i}\!=\!0,\\
1,&\hbox{if}~w_1(L^{\wt\phi})|_{(\prt\Si^b)_i}\!\neq\!0.\end{cases}$$
As in the proofs of \cite[Lemma~6.37]{Melissa} and \cite[Theorem~1.1]{XCapsSetup}, 
we pinch off a circle near each boundary 
component $(\prt\Si^b)_i$ to form a closed surface~$\Si'$ with $m$~disks $B_1,\ldots,B_m$ attached.
We deform the bundles~$V$ and~$L$ to bundles~$V_0$ and~$L_0$ over the resulting nodal surface~$\Si_0$
so that $\deg L_0|_{\Si'}\!=\!0$.
Thus, a trivialization of $L_0|_{\prt\Si^b}$ that extends over each disk extends over~$\Si_0$.
The two determinants on the right-hand side of~\eref{bdnisom_e} are canonically isomorphic
to the determinants of the induced real linear CR-operators~$D_0$ and $\dbar_0$
on~$V_0$ and $\Si_0\!\times\!\C$, respectively.
An orientation on $(\det D_0)\!\otimes\!(\det\dbar_0)^{\otimes n}$ is determined
by orientations of the analogous tensor products over~$\Si_0$ and the $m$~disks.
The former have canonical complex orientations.
If $\prt_1^c\Si^b\!=\!\eset$,
the stabilization and ARS orientations of the tensor products of the determinant lines over~$B_i$ 
induced by a trivialization of $V^{\vph}\!\oplus\!2L^{\wt\phi}$ are the same if and only~if 
\BE{RelSpinOrientCrl_e}\deg L_0|_{B_i}-\ve_i(L)\in4\Z\,;\EE
see Lemma~\ref{RelSpinOrient_lmm2} and Corollary~\ref{RelSpinOrient_crl0}.
Summing up~\eref{RelSpinOrientCrl_e} over $i\!=\!1,\ldots,m$, we obtain the first claim.
The second claim follows similarly from Lemmas~\ref{RelSpinOrient_lmm1} and~\ref{RelSpinOrient_lmm5}.
\end{proof}

\begin{proof}[{\bf\emph{Proof of Theorem~\ref{RelSpinOrient_thm}}}]
Since the fibers of the forgetful morphism~\eref{ffmarked_e} are canonically oriented,
it is sufficient to establish the claims for $l\!=\!2$.
In this case, the moduli space is oriented via the canonical isomorphism~
\eref{ffisom_e} with $(g,l)\!=\!(0,2)$ and~$\si\!=\!\tau$.
By the paragraph above Theorem~\ref{RelSpinOrient_thm},
the orientation of the last factor in~\eref{ffisom_e} is the same in all three approaches
to orienting the moduli space.
The orientations of the first factor on the right-hand side of~\eref{ffisom_e}
are compared by Corollary~\ref{RelSpinOrient_crl} with~$L$ replaced by~$L^*$. 
Taking into account that $c_1(TX)\!=\!2c_1(L)$, we obtain Theorem~\ref{RelSpinOrient_thm}.
\end{proof}

\begin{rmk}
It is not necessary to require that the rank~$n$ of the real bundle pair~$(V,\vph)$
being stabilized be at least~3, since lower-rank real bundle pairs can first be stabilized with
the trivial rank~2 real bundle pair.
The proof of Theorem~\ref{RelSpinOrient_thm} requires only the $(\Si,\si)\!=\!(\P^1,\tau)$
case of Corollary~\ref{RelSpinOrient_crl}, but it is natural to formulate it for arbitrary symmetric
surfaces~$(\Si,\si)$.
\end{rmk}

\begin{rmk}\label{Ge2_rmk2}
Two real line bundles $L_1^{\R},L_2^{\R}\!\lra\!Y$ are isomorphic if and only if 
$w_1(L_1^{\R})\!=\!w_1(L_2^{\R})$, provided $Y$ is paracompact.
In such a case, there is a canonical homotopy class of isomorphisms between 
$2L_1^{\R}$ and~$2L_2^{\R}$.
If $V^{\R}\!\lra\!Y$ is an oriented vector bundle, a spin structure on $V^{\R}\!\oplus\!2L_1^{\R}$
thus corresponds to a spin structure on $V^{\R}\!\oplus\!2L_2^{\R}$.
The proofs of Proposition~\ref{RelSpinOrient_prp} and
Corollaries~\ref{RelSpinOrient_crl0} and~\ref{RelSpinOrient_crl}\ref{RelSpinorient_it} 
imply that the stabilization orientation on the right-hand side of~\eref{bdnisom_e} 
induced by a spin structure on $V^{\vph}\!\oplus\!2L^{\wt\phi}$
depends only on $w_1(L^{\wt\phi})$ and this spin structure,
and not on~$(L,\wt\phi)$ itself.
\end{rmk}

\subsection{Some applications}
\label{OrientApply_subs}

\noindent
We now make a number of explicit statements concerning orientations of 
the determinants of real CR-operators on real bundle pairs
over~$(\P^1,\tau)$ and~$(\P^1,\eta)$.
The proofs of these statements, which are useful for  computational purposes
and are applied in~\cite{RealGWsIII}, are in the spirit of Section~\ref{RelSpinOrient_subs}.\\

\noindent
Let $\ga_1^{\R}\!\lra\!\R\P^1$
denote the tautological line bundle. 
For $f\!:\R\P^1\!\lra\!\GL_k\R$, define
$$\Psi_f\!: \R\P^1\!\times\!\R^k \lra \R\P^1\!\times\!\R^k \qquad\hbox{by}\quad
\Psi_f(z,v)=\big(z,f(z)v\big).$$
Denote by $I_k^-\!\in\!\tO(k)$ the diagonal matrix with the first diagonal entry equal to  $-1$
and the remaining diagonal entries equal to~1. 

\begin{lmm}\label{VBaut_lmm}
Let $k,m\!\in\!\Z^{\ge0}$.
If  $k\!\ge\!2$, every automorphism~$\Psi$ of the real vector bundle
$$V_{k,m}\equiv \big(\R\P^1\!\times\!\R^k\big)\oplus m\ga_1^{\R}\lra\R\P^1$$
is homotopy equivalent to an automorphism of the form $\Psi_f\!\oplus\!\Id_{m\ga_1^{\R}}$
for some $f\!:\R\P^1\!\lra\!\tO(k)$;
any two such maps~$f$ differ by an even multiple of a generator of $\pi_1(\SO(k))$.
If $m\!\ge\!1$, the automorphism~$\Psi$ negating a $\ga_1^{\R}$~component  
is not homotopic to $\Psi_f\!\oplus\!\Id_{m\ga_1^{\R}}$ for any constant map~$f$.
If $m\!\ge\!2$, the interchange~$\Psi$ of two of the $\ga_1^{\R}$~components
is not homotopic to $\Psi_f\!\oplus\!\Id_{m\ga_1^{\R}}$ for any constant map~$f$.
\end{lmm}

\begin{proof}
Let $I_k^+\!=\!I_k$, $x_0\!\in\!\R\P^1$ be any point, and
$$\Aut_{x_0}^{\pm}(V_{k,m})\equiv\big\{\Psi\!\in\!\Aut(V_{k,m})\!:\,
\Psi_{x_0}\!=\!I_k^{\pm}\!\oplus\!I_{m\ga_1^{\R}|_{x_0}}\big\}.$$
Since $\tO(k\!+\!m)$ has two connected components, one containing $I_{k+m}^+$ and the other $I_{k+m}^-$,
it is sufficient to establish the first two claims of this lemma for an automorphism
$\Psi\!\in\!\Aut_{x_0}^{\pm}(V_{k,m})$.\\

\noindent
Since every line bundle over the interval $\bI\!\equiv\![0,1]$ is trivial,
\BE{VBaut_e5}\Aut_{x_0}^{\pm}(V_{k,m}) \approx
\big\{f\!\in\!C(\bI;\tO(k\!+\!m))\!:\,f(0),f(1)\!=\!I_{k+m}^{\pm}\big\}.\EE
The first claim thus follows from the~map
$$\pi_1\big(\tO(k),I_k^{\pm}\big)\lra \pi_1\big(\tO(k\!+\!m),I_{k+m}^{\pm}\big)$$
induced by the natural inclusion $\tO(k)\!\lra\!\tO(k\!+\!m)$ being surjective for $k\!\ge\!2$.
The second claim follows from the kernel of this map being the even multiples of a generator 
of $\pi_1(\SO(k))$.\\

\noindent
By rotating in the fibers of $2\ga_1^{\R}$, 
the interchange of the two components of $2\ga_1^{\R}$ can be homotoped~to 
the automorphism negating the first component and leaving the second component unchanged.
Thus, the last claim of the lemma follows from the third.
It is sufficient to establish the latter for $k\!\ge\!1$.\\

\noindent
We first consider the $(k,m)\!=\!(1,1)$ case of the third claim.
Since every line bundle over~$\bI$ is trivial,
$$V_{1,1}=\big(\bI\!\times\!\C\big)/\!\sim, \qquad (1,c)\sim\big(0,\bar{c}\big)~~\forall~c\!\in\!\C.$$
With respect to this identification, the relevant automorphism~$\Psi$ is given~by
$$\Psi\!: V_{1,1}\lra V_{1,1}, \qquad \Psi\big([t,c]\big)=\big[t,\bar{c}\big].$$
For each $s\!\in\!\R$, define an automorphism~$\Psi_s$ of~$V_{1,1}$ by
$$\Psi_s\!: V_{1,1}\lra V_{1,1}, \qquad \Psi_s\big([t,c]\big)=\big[t,\ne^{\fI \pi(1-2t)s}\bar{c}\big].$$
The family $(\Psi_s)_{s\in[0,1]}$ is a homotopy from the automorphism~$\Psi$ of~$V_{1,1}$
to the element of $\Aut_{x_0}^-(V_{1,1})$ corresponding to the~map
$$f\!: (\bI,0,1)\lra\big(\tO(2),I_2^-,I_2^-), \qquad t\lra \ne^{-2\pi\fI t}I_2^-\,,$$
under the identification~\eref{VBaut_e5}. 
Since $f$ is a generator of $\pi_1(\tO(2),I_2^-)\!\approx\!\Z$, its image under 
the homomorphism 
$$\pi_1\big(\tO(2),I_2^-\big)\lra\pi_1\big(\tO(k\!+\!m),I_{k+m}^-\big)$$
induced by the natural inclusion $\tO(2)\!\lra\!\tO(k\!+\!m)$ is non-trivial.
This implies the last claim.
\end{proof}

\noindent
Let $a\!\in\!\Z^{\ge0}$, 
$(L,\wt\phi)$ be a rank~1 real bundle pair over $(\P^1,\tau)$
of degree~$1\!+\!2a$, and $D_L$ be a real CR-operator
on~$(L,\wt\phi)$.
Fix a nonzero vector $e\!\in\!T_0\P^1$.
The homomorphism
$$\ev_{L;0}: \ker D_L\lra (1\!+\!a)L|_0, \qquad
\ev_{L;0}(\xi)=\big(\xi(0),\na_e\xi,\ldots,\na_e^{\otimes a}\xi\big),$$
is then an isomorphism.
It thus induces an orientation on $\det D_L$ from the complex orientation of~$L|_0$;
we will call the former the \sf{complex orientation of~$\det D_L$}.\\

\noindent
Let $(L_0,\wt\phi_0)$ be a rank~1 real bundle pair over $(\P^1,\tau)$
of degree~1.
If $(L_1,\wt\phi_1)$ and $(L_2,\wt\phi_2)$ are rank~1 real bundle pairs over $(\P^1,\tau)$
of odd degrees, the composition of the isomorphism~$\psi_0$ in~\eref{psi0dfn_e}
with the isomorphism
$$L_1^{\wt\phi_1}\!\oplus\!L_2^{\wt\phi_2} \approx L_0^{\wt\phi_0}\!\oplus\!L_0^{\wt\phi_0}$$
induced by isomorphisms on each component determines an orientation~on
$$\det D_{L_1\oplus L_2} \approx 
\big(\!\det D_{L_1}\big)\otimes \big(\!\det D_{L_2}\big)$$
via the isomorphism~\eref{bdnisom_e} with $\Si^b$ being the unit disk around $0\!\in\!\C$.
By the third statement of Lemma~\ref{VBaut_lmm}, 
changing the homotopy class of a component isomorphism would change
the orientation and the spin of the induced trivialization and 
thus would have no effect on the induced orientation.
This is also implied by the next statement.

\begin{crl}\label{psi0orient_crl}
Suppose $a_1,a_2\!\in\!\Z^{\ge0}$, 
$(L_1,\wt\phi_1)$ and $(L_2,\wt\phi_2)$ are rank~1 real bundle pairs over $(\P^1,\tau)$
of degrees~$1\!+\!2a_1$ and~$1\!+\!2a_2$, respectively, and  
$D_{L_1}$ and $D_{L_2}$  are real CR-operators
on $(L_1,\wt\phi_1)$ and $(L_2,\wt\phi_2)$.
The orientations on $\det(D_{L_1}\!\oplus\!D_{L_2})$
induced by the isomorphism~$\psi_0$ in~\eref{psi0dfn_e} 
and by the complex orientations on  $\det(D_{L_1})$ and $\det(D_{L_2})$
are the~same.
\end{crl}

\begin{proof}
The construction of the orientation on the determinant line induced by a trivialization
of the real part of the bundle in the proofs of \cite[Theorem~8.1.1]{FOOO}
and \cite[Lemma~6.37]{Melissa}
commutes with the evaluations at the interior points;
these can be used to reduce the degree of the bundle.
Thus, it is sufficient to consider the case $a_1,a_2\!=\!0$.
The latter is Proposition~\ref{RelSpinOrient_prp}.
\end{proof}

\noindent
Suppose 
\BE{P1seq_e} 0\lra (V,\vph)\lra (V_{\bu},\vph_{\bu})\oplus (V_c,\vph_c)\lra (\cL,\wt\phi)\lra0\EE
is an exact sequence of real bundle pairs over $(\P^1,\tau)$ such that 
$V_{\bu}^{\vph_{\bu}}\!\lra\!S^1$ is orientable of rank $k\!\ge\!2$ and
$$(V_c,\vph_c)=\bigoplus_{i=1}^m\big(V_{c;i},\phi_{c;i}\big)
\qquad\hbox{and}\qquad
 (\cL,\wt\phi)=\bigoplus_{i=1}^m\big(L_i,\wt\phi_i)$$
are direct sums of rank~1 real vector bundle pairs of odd positive degrees.
By Lemma~\ref{VBaut_lmm}, the short exact sequence~\eref{P1seq_e} and
a trivialization of $V_{\bu}^{\vph_{\bu}}$ determine
a homotopy class of trivializations of~$V^{\vph}$ up~to
\begin{enumerate}[label=(\arabic*),leftmargin=*]

\item simultaneous flips of the orientation and the spin,

\item composition with an even multiple of a generator of $\pi_1(\SO(k))$.

\end{enumerate}
Via the isomorphism~\eref{bdnisom_e} with $\Si^b$ being the unit disk around $0\!\in\!\C$,
a trivialization of $V_{\bu}^{\vph_{\bu}}$ thus determines an orientation 
of the determinant of a real CR-operator~$D_V$ on the real bundle pair~$(V,\vph)$.
It also  determines an orientation 
of the determinant of a real CR-operator~$D_{V_{\bu}}$ 
on the real bundle pair~$(V_{\bu},\vph_{\bu})$.
A short exact sequence 
\BE{P1seq_e2} 0\lra D_V\lra D_{V_{\bu}}\oplus D_{V_c}\lra D_{\cL}\lra0\EE
of real CR-operators on the real bundle pairs in~\eref{P1seq_e}
gives rise to an isomorphism
\BE{P1seq_e3}  \det\big(D_V\big)\otimes\det\big(D_{\cL}\big)
\approx \det\big(D_{V_{\bu}}\big)\otimes \det\big(D_{V_c}\big)\,.\EE

\begin{crl}\label{seqtau_crl}
The isomorphism~\eref{P1seq_e3} is orientation-preserving with respect~to
\begin{enumerate}[label=$\bu$,leftmargin=*]

\item the orientations~on $\det(D_V)$ and $\det(D_{V_{\bu}})$ 
induced by a trivialization of $V_{\bu}^{\vph_{\bu}}$ and

\item the complex orientations on $\det(D_{\cL})$ and $\det(D_{V_c})$.

\end{enumerate}
\end{crl}

\begin{proof}
Since the claim is invariant under augmenting $(V_c,\vph_c)$ and $(\cL,\wt\phi)$
by the same rank~1 real bundle pair of odd positive degree,
we can assume that $m\!=\!2m'$ for some $m'\!\in\!\Z^{\ge0}$.
By Corollary~\ref{psi0orient_crl}, the complex orientations on 
$\det(D_{\cL})$ and $\det(D_{V_c})$ are then induced by the trivializations $m'\psi_0$
of~$\cL^{\wt\phi}$  and~$V_c^{\vph_c}$.
The short exact sequence~\eref{P1seq_e} determines a homotopy class of isomorphisms
of real bundle pairs
\BE{P1seq_e7} (V,\vph)\oplus (\cL,\wt\phi)\approx 
(V_{\bu},\vph_{\bu})\oplus (V_c,\vph_c)\EE
over $(\P^1,\tau)$.
By the above, the orientations on
\BE{P1seq_e8}\begin{split}
\det\big(D_V\!\oplus\!D_{\cL}\big)&=
\det\big(D_V\big)\!\otimes\!\det\big(D_{\cL}\big)
\qquad\hbox{and}\\
 \det\big(D_{V_{\bu}}\!\oplus\!D_{V_c}\big)
&=\det\big(D_{V_{\bu}}\big)\!\otimes\!\det\big(D_{V_c}\big)
\end{split}\EE
specified in the statement of this corollary are induced by homotopy classes of
trivializations of the real bundles
$$ V^{\vph}\!\oplus\!\cL^{\wt\phi},V_{\bu}^{\vph_{\bu}}\!\oplus\!V_c^{\vph_c}\lra S^1$$ 
that are identified under the isomorphism~\eref{P1seq_e7} restricted to the real parts 
of the bundles.
The isomorphism~\eref{P1seq_e3} is orientation-preserving with respect to these orientations.
\end{proof}

\noindent
We will next obtain an analogue of Corollary~\ref{seqtau_crl} for  real bundle pairs 
over~$(\P^1,\eta)$. 
Define a $\C$-antilinear automorphism of~$\C^2$ by 
$$\fc_{\eta}\!:\C^2\lra\C^2, \qquad \fc_{\eta}\big(v_1,v_2\big)=
\big(\bar{v}_2,-\bar{v}_1\big);$$ 
it has order~4.
Let 
$$\ga=\cO_{\P^1}(-1)\equiv
\big\{(\ell,v)\!\in\!\P^1\!\times\!\C^2\!:\,v\!\in\!\ell\!\subset\!\C^2\big\}$$ 
denote the tautological line bundle.
For $a\!\in\!\Z^+$, the involution $\eta$ lifts to 
a conjugation on $2\ga^{\otimes a}$ as 
$$\wt\eta_{1,1}^{(-a)}\big(\ell,v^{\otimes a},w^{\otimes a}\big)
=\big(\eta(\ell),(\fc_{\eta}(w))^{\otimes a},
(-\fc_{\eta}(v))^{\otimes a}\big).$$
We denote the induced conjugations on 
$$2\cO_{\P^1}(a)=\big(2\ga^{\otimes a}\big)^*
\qquad\hbox{and}\qquad
\cO_{\P^1}(2a) \equiv \La_{\C}^2\big(2\cO_{\P^1}(a)\big)$$
by $\wt\eta_{1,1}^{(a)}$ and $\wt\eta_1^{(2a)}$, respectively.
We note that $\wt\eta_{1,1}^{(2a)}\!\approx\!2\wt\eta_1^{(2a)}$.\\

\noindent
Let $a\!\in\!\Z^{\ge0}$ and $D_a$ be the real CR-operator
on $(2\cO_{\P^1}(1\!+\!2a),\wt\eta_{1,1}^{(1+2a)})$ induced by the standard $\dbar$-operator
on~$2\cO_{\P^1}(1\!+\!2a)$.
Fix a holomorphic connection~$\na$ on $\cO_{\P^1}(1\!+\!2a)$ and
a nonzero vector $e\!\in\!T_0\P^1$.
The homomorphism
\begin{gather*}
\ev_{a;0}: \ker D_a\lra 
\big((1\!+\!a)\cO_{\P^1}(1\!+\!2a)|_0\big)\!\oplus\!\big((1\!+\!a)\cO_{\P^1}(1\!+\!2a)|_0\big),\\
\ev_{a;0}(\xi_1,\xi_2)=\big((\xi_1(0),\na_e\xi_1,\ldots,\na_e^{\otimes a}\xi_1),
(\xi_2(0),\na_e\xi_2,\ldots,\na_e^{\otimes a}\xi_2)\big),
\end{gather*}
is then an isomorphism.
It thus induces an orientation on $\det D_a$ from the complex orientation of~$\cO_{\P^1}(1\!+\!2a)|_0$;
we will call the former the \sf{complex orientation of~$\det D_a$}.\\

\noindent
As before,
denote by $S^1\!\subset\!\P^1$  and $\Si^b\!\subset\!\P^1$  the unit circle 
and the unit disk around $0\!\in\!\P^1$, respectively. 
Let $\psi_0'$ be the trivialization of $(2\cO_{\P^1}(1),\wt\eta_{1,1}^{(1)})$
over~$S^1$ given~by
\BE{psi0prdfn_e}\psi_0'\big(\al_1,\al_2\big)
=\left(\begin{array}{c}\fI\al_1(1,z)\!-\!\fI z^{-1}\al_2(1,z)\\
\al_1(1,z)\!+\!z^{-1}\al_2(1,z)\end{array}\right)\in\C^2
\qquad\forall~(\al_1,\al_2)\!\in\!2\cO_{\P^1}(1)|_z,~z\!\in\!S^1.\EE
This is a component of the composite trivialization appearing in
the proof of \cite[Proposition~6.2]{Teh}.  
The next statement is the analogue of Proposition~\ref{RelSpinOrient_prp} 
in this setting.

\begin{crl}\label{EtaOrient_crl}
The orientation on $\det D_0^b$ induced by the trivialization~$\psi_0'$ 
as in the proof of \cite[Lemma~2.5]{Teh}
agrees with the complex orientation.
\end{crl}

\noindent
We give three proofs of this statement;
they correspond to the three proofs of  Proposition~\ref{RelSpinOrient_prp}.

\begin{proof}[{\bf\emph{Proof~1}}]
We denote by $p_1$ and $p_2$ the two standard holomorphic sections of $\cO_{\P^1}(1)$:
$$p_1\big(\ell,(v_1,v_2)\big)=v_1, \quad 
p_2\big(\ell,(v_1,v_2)\big)=v_2 \qquad\forall~\big(\ell,(v_1,v_2)\big)\in\ga.$$
The complex orientation for $\det D_0^b$ is  determined by the basis 
$$s_{11}\equiv (p_1,p_2), \quad s_{12}\equiv (\fI p_1,-\fI p_2), 
\quad s_{21}\equiv (-p_2,p_1),\quad s_{22}\equiv (\fI p_2,\fI p_1),$$
for the kernel of the surjective operator~$D_0^b$.\\

\noindent 
The trivialization $\psi_0'$ extends as a trivialization~$\Psi_0'$
of $(2\cO_{\P^1}(1),\wt\eta_{1,1}^{(1)})$ over $\P^1\!-\!\{0,\i\}$
by the same formula.
We note that
\begin{alignat*}{2}
\big\{\Psi_0's_{11}\big\}(z)&=(0,2),
&\quad
\big\{\Psi_0's_{12}\big\}(z)&=(-2,0),\\
\big\{\Psi_0's_{21}\big\}(z)&=\big(-\!\fI z^{-1}(1\!+\!z^2),z^{-1}(1\!-\!z^2)\big),
&\quad
\big\{\Psi_0's_{22}\big\}(z)&=\big(z^{-1}(1\!-\!z^2),\fI z^{-1}(1\!+\!z^2)\big).
\end{alignat*}
The orientation on $\det D_0^b$ induced by the trivialization~$\psi_0'$ is obtained from
the isomorphism
$$\ker D_0^b\lra \R\!\oplus\!\R\oplus\big\{\Res_{z=0}(\Psi_0'\xi)\!:\,
\xi\!\in\!\ker D_0^b\big\}, \quad
\xi\lra \big(\Re(\{\Psi_0'\xi\}(1)),\Res_{z=0}(\Psi_0'\xi)\big).$$
The last space above is a complex subspace of~$\C^2$.
Under this isomorphism, the basis $s_{11},s_{12},s_{21},s_{22}$ is sent~to
$$(0,2;0,0), \quad (-2,0;0,0), \quad (0,0;-\fI,1), \quad
(0,0;1,\fI).$$
Thus, an oriented basis for the target of the above isomorphism is given~by
$$(2,0;0,0), \quad (0,2;0,0), \quad (0,0;-\fI,1), \quad (0,0;1,\fI).$$
The change of basis matrix from the first basis to this one is given~by
$$\left( \begin{array}{cccc} 0& -1& 0& 0\\ 1& 0& 0& 0\\ 0& 0& 1& 0\\
0& 0& 0& 1\end{array}\right).$$
The determinant of this matrix is $+1$.
\end{proof}

\begin{proof}[{\bf\emph{Proof~2}}]
Define 
\begin{gather*}
\eta_3\!:\P^3\lra\P^3, \qquad 
[Z_1,Z_2,Z_3,Z_4]\lra\big[\ov{Z}_2,-\ov{Z}_1,\ov{Z}_4,-\ov{Z}_3\big],\\
 \fM_1(\P^1) =\fM_1(\P^1,1)^{\eta,\eta},\qquad
\fM_1(\P^3)=\fM_1(\P^3,1)^{\eta_3,\eta}.
\end{gather*} 
We now proceed through the first two paragraphs of the second proof of Proposition~\ref{RelSpinOrient_prp}
replacing $\tau$, $\tau_3$, and $D_{2L}^b$ by~$\eta$, $\eta_3$, and~$D_0^b$, respectively.
By \cite[Lemma~5.3]{Teh}, the isomorphism~\eref{pf2_e7} is still orientation-{\it reversing} 
with respect to the algebraic orientations on $\fM_1(\P^1)$ in~$\fM_1(\P^3)$
defined in \cite[Section~5.2]{Teh}.
The isomorphism~\eref{pf2_e9} is now orientation-{\it preserving} with respect to  
the complex orientation on the left-hand side.\\

\noindent
Along with Euler's sequence for~$\P^1$ and the  orienting procedure
of \cite[Lemma~2.5]{Teh}, the trivialization~$\psi_0'$
determines an orientation on~$\fM_1(\P^1)$;
we will call it the \textsf{$\psi_0'$-orientation}.
Since the top exterior power of the real bundle pair 
\BE{etapf2_e11}2\big((2\cO_{\P^1}(1),\wt\eta_{1,1}^{(1)}\big)\big|_{S^1}\lra 
\big(S^1,\eta|_{S^1}\big)\subset(\P^1,\eta)\EE
is canonically a square, it admits a canonical homotopy class of trivializations;
see Lemma~2.4 and Section~5.5 in~\cite{Teh}.
Along with Euler's sequence for~$\P^3$, it determines an orientation on~$\fM_1(\P^3)$;
we will call it the \textsf{square root orientation}.
Along with the $\psi_0'$-orientation on $\fM_1(\P^1)$, 
it induces an orientation on~$\cN_{[\io,0]}\fM$;
we will call it the \textsf{$\psi_0'$-orientation}.
Since the square root orientation on~$\fM_1(\P^3)$ agrees with the orientation
induced by the trivialization~$2\psi_0'$ of~\eref{etapf2_e11},
the first isomorphism in~\eref{pf2_e9} is orientation-{\it preserving} with respect
to the orientation on the left-hand side induced by~$\psi_0'$
and the $\psi_0'$-orientation on~$\cN_{[\io,0]}\fM$.\\

\noindent
As summarized in the paragraph above \cite[Remark~6.9]{Teh}, 
the algebraic orientations on~$\fM_1(\P^1)$ and~$\fM_1(\P^3)$  
are the same as the $\psi_0'$-orientation and 
the opposite of the square root orientation, respectively.
Therefore,  the $\psi_0'$-orientation on~$\cN_{[\io,0]}\fM$ is the opposite of
the algebraic orientation.
Since the second isomorphism in~\eref{pf2_e9} is orientation-reversing
with respect to the latter,
it follows that the composite isomorphism in~\eref{pf2_e9} is orientation-preserving
with respect
to the orientation on the left-hand side induced by~$\psi_0'$.
Since this is also the case with respect to the complex orientation on the left-hand side,
these two orientations on~$\ker D_0^b$ agree.  
\end{proof}

\begin{proof}[{\bf\emph{Proof~3}}]
The reasoning in the third proof of Proposition~\ref{RelSpinOrient_prp}
with $\tau$ and $D_{2L}^b$ replaced by~$\eta$ and~$D_0^b$, respectively,
applies without any changes,
except \cite[(6.13)]{Teh} is no longer relevant.
\end{proof}

\noindent
By \cite[Lemma~2.4]{Teh}, the homotopy classes of trivializations of
\hbox{$(2\cO_{\P^1}(1\!+\!2a),\wt\eta_{1,1}^{(1\!+\!2a)})$}
over~$S^1$ correspond to the homotopy classes of trivializations~of
\BE{etared_e}\La_{\C}^{\top}\big(2\cO_{\P^1}(1\!+\!2a),\wt\eta_{1,1}^{(1\!+\!2a)}\big)
\approx
\La_{\C}^{\top}\big(2\cO_{\P^1}(1),\wt\eta_{1,1}^{(1)}\big)
\otimes
\big(\cO_{\P^1}(2a),\wt\eta_1^{(2a)}\big)^{\otimes2}\EE
over~$S^1$.
Since the last factor in~\eref{etared_e} is a square, it has a canonical homotopy class
of trivializations over~$S^1$.
Thus, the trivialization~$\psi_0'$ of the first factor on
the right-hand side of~\eref{etared_e} determines   a  homotopy class
of trivializations of $(2\cO_{\P^1}(1\!+\!2a),\wt\eta_{1,1}^{(1\!+\!2a)})$ over~$S^1$
and thus an orientation on~$\det D^b_a$.
The next statement is the analogue of Corollary~\ref{psi0orient_crl};
it is deduced from Corollary~\ref{EtaOrient_crl} in the same way 
as  Corollary~\ref{psi0orient_crl} is obtained from Proposition~\ref{RelSpinOrient_prp}.

\begin{crl}\label{EtaOrient_crl2}
Suppose $a\!\in\!\Z^{\ge0}$.
The orientation on $\det D_a^b$ induced by the trivialization~$\psi_0'$ 
as in the proof of \cite[Lemma~2.5]{Teh} 
agrees with the complex orientation. 
\end{crl}

\noindent
Suppose 
\BE{etaP1seq_e} 0\lra (V,\vph)\lra (V_{\bu},\vph_{\bu})\oplus (V_c,\vph_c)\lra (\cL,\wt\phi)\lra0\EE
is an exact sequence of real bundle pairs over $(\P^1,\eta)$ such~that
$$(V_c,\vph_c)=\bigoplus_{i=1}^m(2\cO_{\P^1}(1\!+\!2a_i),\wt\eta_{1,1}^{(1\!+\!2a_i)})
\qquad\hbox{and}\qquad
 (\cL,\wt\phi)=\bigoplus_{i=1}^m(2\cO_{\P^1}(1\!+\!2a_i'),\wt\eta_{1,1}^{(1\!+\!2a_i')})$$
for some $a_i,a_i'\!\in\!\Z^{\ge0}$.
Since the homotopy classes of trivializations of $(V_c,\vph_c)|_{S^1}$
correspond to  the homotopy classes of trivializations of $ (\cL,\wt\phi)|_{S^1}$,
a homotopy class of trivializations of $(V_{\bu},\vph_{\bu})|_{S^1}$
determines a homotopy class of trivializations of $(V,\vph)|_{S^1}$
via the exact sequence~\eref{etaP1seq_e}.
Via the isomorphism~\eref{bdnisom_e} with $\Si^b$ being the unit disk around $0\!\in\!\C$,
a trivialization of $(V_{\bu},\vph_{\bu})|_{S^1}$ thus determines an orientation 
of the determinant of a real CR-operator~$D_V$ on the real bundle pair~$(V,\vph)$.
It also  determines an orientation 
of the determinant of a real CR-operator~$D_{V_{\bu}}$ 
on the real bundle pair~$(V_{\bu},\vph_{\bu})$.
A short exact sequence~\eref{P1seq_e2} 
of real CR-operators on the real bundle pairs in~\eref{etaP1seq_e}
gives rise to an isomorphism as in~\eref{P1seq_e3}.

\begin{crl}\label{EtaOrient_crl3}
The isomorphism~\eref{P1seq_e3} is orientation-preserving with respect~to
\begin{enumerate}[label=$\bu$,leftmargin=*]

\item the orientations~on $\det(D_V)$ and $\det(D_{V_{\bu}})$ 
induced by a trivialization of $(V_{\bu},\vph_{\bu})$ over~$S^1$  and

\item the complex orientations on $\det(D_{\cL})$ and $\det(D_{V_c})$.

\end{enumerate}
\end{crl}

\begin{proof}
By Corollary~\ref{EtaOrient_crl2}, the complex orientations on 
$\det(D_{\cL})$ and $\det(D_{V_c})$ are  induced by the trivializations $m\psi_0'$
of $(\cL,\wt\phi)$ and $(V_c,\vph_c)$ over~$S^1$.
The short exact sequence~\eref{etaP1seq_e} determines a homotopy class of isomorphisms~\eref{P1seq_e7}
over $(\P^1,\eta)$.
Thus, the orientations on~\eref{P1seq_e8}
specified in the statement of this corollary are induced by homotopy classes of
trivializations of the real bundle~pairs
$$ (V,\vph)\!\oplus\!(\cL,\wt\phi),(V_{\bu},\vph_{\bu})\!\oplus\!(V_c,\vph_c)
\lra \big(S^1,\eta|_{S^1}\big)$$ 
that are identified under the isomorphism~\eref{P1seq_e7} restricted to~$S^1$.
The isomorphism~\eref{P1seq_e3} is orientation-preserving with respect to these orientations.
\end{proof}

\section{The compatibility of the canonical orientations}
\label{OrientCompat_sec}

\noindent
In this section, we establish Theorem~\ref{CompOrient_thm}.
In order to do so, we study how each step 
in the construction of the orientation on $\fM_{g,l}(X,B;J)^{\phi}$ in
\cite[Section~5]{RealGWsI}
extends across the strata consisting of maps from symmetric surfaces with 
a pair of conjugate nodes.
The argument is similar to \cite[Section~6]{RealGWsI}, which studies the extendability
of the orientation on $\fM_{g,l}(X,B;J)^{\phi}$ induced by a real orientation on~$(X,\om,\phi)$
across the codimension-one strata.
We also compare the resulting extensions with the corresponding objects over 
the normalizations.

\subsection{Two-nodal symmetric surfaces}
\label{TwoNodal_subs}

\noindent
We begin by establishing Proposition~\ref{canonisom_prp} 
for symmetric surfaces with one pair of conjugate nodes.
If $(\Si,\si)$ is a symmetric surface, possibly nodal and disconnected, and 
$G$ is a Lie group with a natural conjugation, such as $\C^*$, 
$\SL_n\C$, or $\GL_n\C$,
denote by $\cC(\Si,\si;G)$ the topological group of continuous maps $f\!:\Si\!\lra\!G$
such that \hbox{$f(\si(z))\!=\!\ov{f(z)}$} for all $z\!\in\!\Si$.
The restrictions of such functions to the fixed locus $\Si^{\si}\!\subset\!\Si$ take values in
the real locus of~$G$, i.e.~$\R^*$, $\SL_n\R$, and $\GL_n\R$,
in the three examples.

\begin{lmm}\label{homotoid_lmm}
Suppose $(\Si,\si)$ is a symmetric surface, possibly nodal and disconnected,
$x\!\in\!\Si\!-\!\Si^{\si}$, and $G$ is a connected Lie group with a natural conjugation.
For every $f\!\in\!\cC(\Si,\si;G)$ and an open neighborhood $U\!\subset\!\Si$ of~$x$, 
there exists a path $f_t\!\in\!\cC(\Si,\si;G)$ such that $f_0\!=\!f$,
$f_1(x)\!=\!\Id$, and $f_t\!=\!f$ on $\Si\!-\!U\!\cup\!\si(U)$.
\end{lmm}

\begin{proof}
By shrinking~$U$, we can assume that $U\!\cap\!\si(U)\!=\!\eset$.
Let $\rho\!:\Si\!\lra\![0,1]$ be a smooth $\si$-invariant function such that 
$\rho(x)\!=\!1$ and $\rho\!=\!0$ on $\Si\!-\!U\!\cup\!\si(U)$.
Choose a path $g_t\!\in\!G$ such that $g_0\!=\!\Id$ and $g_1\!=\!f(x)$.
The~path $f_t\!\in\!\cC(\Si,\si;G)$ given~by
$$f_t(z)=\begin{cases} g_{\rho(z)t}^{-1}f(z),&\hbox{if}~z\!\in\!U;\\
\ov{g_{\rho(z)t}}^{-1}f(z),&\hbox{if}~z\!\in\!\si(U);\\
f(z),&\hbox{if}~z\!\not\in\!U\!\cup\!\si(U);
\end{cases}$$
has the desired properties.
\end{proof}

\noindent
We will denote the nodes of a connected symmetric surface~$(\Si,\si)$ with one pair of conjugate nodes
by $x_{12}^{\pm}$.
A \sf{normalization} of such $(\Si,x_{12}^{\pm},\si)$ is a smooth, possibly disconnected,
symmetric surface~$(\wt\Si,\wt\si)$ with two distinguished  pairs of conjugate points,
$(x_1^+,x_1^-)$ and $(x_2^+,x_2^-)$;
the normalization map takes $x_i^+$ to~$x_{12}^+$ and  $x_i^-$ to~$x_{12}^-$.

\begin{lmm}\label{homotopextend_lmm}
Suppose $(\Si,\si)$ is a connected symmetric surface with one pair of conjugate nodes,
$n\!\in\!\Z^+$, and  $f\!\in\!\cC(\Si,\si;\SL_n\C)$.
If 
$$f|_{\Si^{\si}}\!: \Si^{\si} \lra \SL_n\R$$
is homotopic to a constant map,
then $f$ is homotopic to the constant map~$\Id$ through maps 
\hbox{$f_t\!\in\!\cC(\Si,\si;\SL_n\C)$}.
\end{lmm}

\begin{proof}
By Lemma~\ref{homotoid_lmm}, we can assume that $f(x_{12}^+)\!=\!\Id$.
Let $\wt{f}\!\in\!\cC(\wt\Si,\wt\si;\SL_n\C)$ be the function corresponding 
to $f\!\in\!\cC(\Si,\si;\SL_n\C)$. 
In particular, $\wt{f}(x_1^{\pm}),\wt{f}(x_2^{\pm})\!=\!\Id$.\\

\noindent
We proceed as in the proof of \cite[Lemma~5.4]{RealGWsI}, which contains 
a picture illustrating a similar argument.
Choose a symmetric half-surface $\wt\Si^b\!\!\subset\!\wt\Si$ and 
a neighborhood $U\!\subset\!\wt\Si^b$ of~$\prt\wt\Si^b$ so that
either $x_1^+,x_2^+\!\in\!\wt\Si^b\!-\!U$ or $x_1^+,x_2^-\!\in\!\wt\Si^b\!-\!U$.
Let $x_2\!=\!x_2^+$ in the first case, $x_2\!=\!x_2^-$ in the second case,
and $x_1\!=\!x_1^+$ in both cases.
Take the cutting paths~$C_i$ so that $x_1,x_2\!\not\in\!C_i$ and 
 the extensions of the homotopies of~$\wt{f}$ from~$C_i$ to~$\wt\Si^b$ 
so that they do not change~$\wt{f}$ at~$x_1$ or~$x_2$.
The surface~$D$ obtained by cutting $\wt\Si^b$ along these paths is either a disk~$D^2$
or two disjoint copies of~$D^2$.
Choose disjoint embedded paths~$\ga_1$ and~$\ga_2$ in~$D$ as in the last paragraph of
the proof of \cite[Lemma~5.4]{RealGWsI} from~$\prt D$
to~$x_1$ and~$x_2$, respectively.
Since  $\wt{f}(x_i)\!=\!\Id$ in this case, we can homotope~$\wt{f}$ to~$\Id$ 
over~$\ga_i$ while keeping it fixed at the endpoints.
Similarly to the second paragraph in the proof of this lemma, 
this homotopy extends over~$D$ without changing~$\wt{f}$ over~$\prt D$ or~$\ga_{3-i}$
and thus descends to~$\wt\Si^b$.
We then cut~$D$ along~$\ga_1$ and~$\ga_2$ into another disk or a pair of disks
and proceed as in
the second half of the last paragraph in the proof of \cite[Lemma~5.4]{RealGWsI}.
The doubled homotopy in the proof of this lemma
then satisfies $\wt{f}_t(x_1^{\pm})\!=\!\wt{f}_t(x_2^{\pm})$
and so descends to~$\Si$.
\end{proof}

\begin{crl}\label{RBPhomotExt_crl}
Let $(\Si,\si)$ be a connected symmetric surface with one pair of conjugate nodes and
$$\Phi,\Psi\!:(V,\vph)\lra \big(\Si\!\times\!\C^n,\si\!\times\!\fc\big)$$
be isomorphisms of real bundle pairs over $(\Si,\si)$.
If the isomorphisms 
\begin{equation*}\begin{split}
\Phi|_{V^{\vph}},\Psi|_{V^{\vph}}\!: V^{\vph}&\lra \Si\!\times\!\R^n, \\
\La_{\C}^{\top}\Phi,\La_{\C}^n\Psi\!:\La_{\C}^{\top}(V,\vph)
&\lra \La_{\C}^{\top}\big(\Si\!\times\!\C^n,\si\!\times\!\fc\big)
=\big(\Si\!\times\!\C,\si\!\times\!\fc\big)
\end{split}\end{equation*}
are homotopic, then so are the isomorphisms~$\Phi$ and~$\Psi$. 
\end{crl}

\begin{proof}
The first paragraph of the proof of \cite[Corollary~5.5]{RealGWsI} applies without 
any changes.
The second paragraph applies with \cite[Lemma~5.4]{RealGWsI}
replaced by Lemma~\ref{homotopextend_lmm} above.
\end{proof}

\begin{lmm}\label{canonisomExt_lmm}
Proposition~\ref{canonisom_prp} holds for connected symmetric
surfaces with one pair of conjugate nodes.
\end{lmm}

\begin{proof}
Let $\wt{V},\wt{L}\!\lra\!\wt\Si$ be complex vector bundles and
$$\psi_1\!:\wt{V}\big|_{x_1^{\pm}}\lra \wt{V}\big|_{x_2^{\pm}}
\qquad\hbox{and}\qquad 
\psi_2\!:\wt{L}\big|_{x_1^{\pm}}\lra \wt{L}\big|_{x_2^{\pm}}$$
be isomorphisms of complex vector spaces such~that 
$$V=\wt{V}\big/\!\!\sim, ~~ v\!\sim\!\psi_1(v)~\forall\,v\!\in\!\wt{V}\big|_{x_1^{\pm}},
\qquad\hbox{and}\qquad 
L=\wt{L}\big/\!\!\sim, ~~ v\!\sim\!\psi_2(v)~\forall\,v\!\in\!\wt{L}\big|_{x_1^{\pm}}.$$
Denote by $\wt\vph_1$ and~$\wt\vph_2$ the lift of $\vph$ to~$\wt{V}$ and 
the lift of $\wt\phi$ to~$\wt{L}$, respectively.
Define
$$\big(\wt{W},\wt\vph_{12}\big)=\big(\wt{V}\!\oplus\!2\wt{L}^*,\wt\vph_1\!\oplus\!2\wt\vph_2^*\big),
\qquad 
\psi_{12}=\psi_1\oplus2(\psi_2^{-1})^*\!: \wt{W}\big|_{x_1^{\pm}}\lra \wt{W}\big|_{x_2^{\pm}}\,.$$
Thus,  $(\wt{V},\wt\vph_1)$ and $(\wt{L},\wt\vph_2)$ are real bundle pairs over $(\wt\Si,\wt\si)$
that descend to the real bundle pairs $(V,\vph)$ and $(L,\wt\phi)$ over~$(\Si,\si)$.
Furthermore,
\BE{canIsomExt_e3}\psi_{12}\circ\wt\vph_{12}=\wt\vph_{12}\circ\psi_{12}\,.\EE
For any $f\!\in\!\cC(\wt\Si,\wt\si;\GL_{n+2}\C)$, let
$$\wt\Psi_f\!:\big(\wt\Si\!\times\!\C^{n+2},\wt\si\!\times\!\fc\big)\lra
\big(\wt\Si\!\times\!\C^{n+2},\wt\si\!\times\!\fc\big), \qquad
\wt\Psi_f(z,v)=\big(z,f(z)v\big).$$
Let $\wt\si'(x_i^{\pm})\!=\!x_{3-i}^{\pm}$ for $i\!=\!1,2$.\\

\noindent
The choices~\ref{isom_it2} and~\ref{spin_it2} in Definition~\ref{realorient_dfn4}
for~$(\Si,\si)$  lift to~$(\wt\Si,\wt\si)$. 
By \cite[Proposition~5.2]{RealGWsI}, there thus exists an isomorphism 
$$\wt\Phi\!:(\wt{W},\wt\vph_{12})\lra \big(\wt\Si\!\times\!\C^{n+2},\wt\si\!\times\!\fc\big)$$
of real bundle pairs over $(\wt\Si,\wt\si)$ that lies in the homotopy class
determined by the lifted real orientation.
It satisfies the spin structure requirement of Proposition~\ref{canonisom_prp}.
By the proof of \cite[Proposition~5.2]{RealGWsI}, $\wt\Phi$ can be chosen so that 
it induces the isomorphism in~\eref{realorient2_e2b} over~$(\wt\Si,\wt\si)$
determined by the lift of a given isomorphism in~\eref{realorient_e4} over~$(\Si,\si)$.
This implies~that 
\BE{canIsomExt_e2a}
\big\{\wt\si'\!\times\!\id\big\}\!\circ\!\big\{\La_{\C}^{\top}\wt\Phi\big\}=
\big\{\La_{\C}^{\top}\wt\Phi\big\}\!\circ\!\big\{\La_{\C}^{\top}\psi_{12}\big\}\!:
\La_{\C}^{\top}\wt{W}|_{x_1^{\pm}}\lra 
\{x_2^{\pm}\}\!\times\!\La_{\C}^{\top}\C^{n+2}\!=\!\{x_2^{\pm}\}\!\times\!\C.\EE
In the next paragraph, we homotope~$\wt\Phi$ near~$x_1^{\pm}$ so that it descends 
to an isomorphism~$\Psi$ over~$\Si$; the latter satisfies 
the two properties in the last sentence of Proposition~\ref{canonisom_prp}.
By Corollary~\ref{RBPhomotExt_crl}, any two such isomorphisms~$\Psi$ are homotopic.\\

\noindent
Define $\psi^{\pm}\!\in\!\GL_{n+2}\C$ by
\BE{canIsomExt_e7}
\id\!\times\!\psi^{\pm}=\big\{\wt\si'\!\times\!\Id\big\}
\!\circ\!\wt\Phi\!\circ\!\psi_{12}\!\circ\!\wt\Phi^{-1}\!:
\{x_1^{\pm}\}\!\times\!\C^{n+2}\lra \{x_1^{\pm}\}\!\times\!\C^{n+2}\,.\EE
By~\eref{canIsomExt_e2a}, $\det_{\C}\!\psi^{\pm}\!=\!1$, i.e.~$\psi\!\in\!\SL_{n+2}\C$.
By~\eref{canIsomExt_e3}, $\ov{\psi^+}\!=\!\psi^-$.
Since $\SL_{n+2}\C$ is connected, there exist $f\!\in\!\cC(\wt\Si,\wt\si;\SL_{n+2}\C)$ 
and a neighborhood $U$ of~$x_1^+$ in~$\wt\Si$ such~that 
\BE{canIsomExt_e4b} 
f(z)=
\begin{cases}\psi^{\pm},&\hbox{if}~z\!=\!x_1^{\pm};\\
\Id,&\hbox{if}~z\!\not\in\!U\!\cup\!\wt\si(U);
\end{cases}
\qquad x_2^{\pm}\!\not\in\!U,\quad U\!\cap\!\wt\si(U)=\eset.\EE
By~\eref{canIsomExt_e7} and~\eref{canIsomExt_e4b},
$$ \big\{\wt\si'\!\times\!\Id\big\}\!\circ\!\wt\Psi_f\!\circ\!\wt\Phi=
\wt\Psi_f\!\circ\!\wt\Phi\!\circ\!\psi_{12}\!:
\wt{W}|_{x_1^{\pm}}\lra \{x_2^{\pm}\}\!\times\!\C^{n+2}\,.$$
Thus, $\wt\Psi_f\!\circ\!\wt\Phi$ descends 
to an isomorphism~$\Psi$ in~\eref{realorient2_e2} of real bundle pairs over~$(\Si,\si)$
that induces the isomorphism in~\eref{realorient2_e2b} determined by 
a given isomorphism in~\eref{realorient_e4}.
\end{proof}

\noindent
Suppose $(\Si,x_{12}^{\pm},\si)$ and $(\wt\Si,x_1^{\pm},x_2^{\pm},\wt\si)$ are as above.
A rank~$n$ real bundle pair~$(V,\vph)$ over~$(\Si,\si)$ lifts to 
a rank~$n$ real bundle pair~$(\wt{V},\wt\vph)$ over~$(\wt\Si,\wt\si)$.
A real orientation on~$(V,\vph)$ lifts to a real orientation on~$(\wt{V},\wt\vph)$.
A real CR-operator~$D$ on~$(V,\vph)$ lifts 
to a real CR-operator~$\wt{D}$ on~$(\wt{V},\wt\vph)$.
There is a short exact sequence of~Fredholm operators
\BE{DvswtD_e}\begin{split}
\xymatrix{0 \ar[r]& \Ga(\Si;V)^{\vph} \ar[r]\ar[d]^D&  
\Ga(\wt\Si;\wt{V})^{\wt\vph} \ar[r]^<<<<<<{\ev_{x_{12}^+}}\ar[d]^{\wt{D}}& 
V_{x_{12}^+} \ar[r]\ar[d]& 0\\
0 \ar[r]& \Ga^{0,1}_{\fJ}(\Si;V)^{\vph} \ar[r]& 
\Ga^{0,1}_{\fJ}(\wt\Si;\wt{V})^{\wt\vph} \ar[r]& 0\ar[r]& 0 }
\end{split}\EE
with the last homomorphism in the top row given~by
$$\ev_{x_{12}^+}(\xi)=\xi(x_1^+)\!-\!\xi(x_2^+) \in 
V_{x_{12}^+}\!=\!\wt{V}_{x_1^+}\!=\!\wt{V}_{x_2^+}\,.$$
Thus, there is a canonical isomorphism 
\BE{DvswtD_e2}
\det \wt{D} \approx \det D \otimes \La_{\R}^{2n}V_{x_{12}^+}\EE
of real lines.\\

\noindent
If $\Psi$ is an isomorphism as in~\eref{realorient2_e2} and $\wt\Psi$ is its lift to~$(\wt\Si,\wt\si)$,
then the diagram 
$$\xymatrix{0 \ar[r]& \Ga\big(\Si;V\!\oplus\!2L^*\big)^{\vph\oplus2\wt\phi^*} 
\ar[r]\ar[d]^{\Psi}&  
\Ga\big(\wt\Si;\wt{V}\!\oplus\!2\wt{L}^*\big)^{\wt\vph\oplus2\wt\phi^*} 
\ar[r]^<<<<<{\ev_{x_{12}^+}}\ar[d]^{\wt\Psi}& 
V_{x_{12}^+}\!\oplus\! 2L^*_{x_{12}^+} \ar[r]\ar[d]^{\Psi}& 0\\
0 \ar[r]& \cC(\Si,\si;\C^{n+2}) \ar[r]&  
\cC(\wt\Si,\wt\si;\C^{n+2}\big) \ar[r]^<<<<<<{\ev_{x_{12}^+}}& 
\C^{n+2} \ar[r]& 0}$$
commutes.
If~$\Psi$ is an isomorphism as in~\eref{realorient2_e2} in the homotopy class
determined by a real orientation on~$(V,\vph)$,
then the lift~$\wt\Psi$ of~$\Psi$ to~$(\wt{V},\wt\vph)$
lies in the homotopy class of isomorphisms determined by 
the induced real orientation on~$(\wt{V},\wt\vph)$.
These two observations yield the following comparison of 
the orientations on the relative determinants provided by 
 Corollary~\ref{canonisom_crl2a}.

\begin{crl}\label{RealOrient_crl}
Let $(\Si,\si)$, $(\wt\Si,\wt\si)$, $(V,\vph)$, and $(\wt{V},\wt\vph)$ be as above.
The isomorphism
\BE{RealOrientCrl_e}
\rdet\,\wt{D}
\approx \big(\rdet\,D\big) \otimes 
\La_{\R}^{2n}V_{x_{12}^+}\!\otimes\!\La_{\R}^{2n}\C^n\EE
induced by the isomorphisms~\eref{DvswtD_e2} for $(V,\vph)$ and 
$(\Si\!\times\!\C^n,\si\!\times\!\fc)$ is orientation-preserving 
with respect to the orientation on $\rdet\,D$ determined by 
a real orientation on~$(V,\vph)$,
the orientation on~$\rdet\,\wt{D}$  determined by
 the lifted real orientation on~$(\wt{V},\wt\vph)$, 
and the complex orientations of  $V_{x_{12}^+}$ and~$\C^n$.
\end{crl}

\subsection{Smoothings of two-nodal symmetric surfaces}
\label{SymmSurfDegen_subs}

\noindent
For a disk $\De\!\subset\!\C$ centered at the origin, let 
\begin{gather*}
\De^*=\De\!-\!\{0\}, \qquad 
\De_{\R}^2=\big\{(t,\bar{t})\!:\,t\!\in\!\De\big\}\,, \qquad
\De_{\R}^{*2}=\De^{*2}\!\cap\!\De_{\R}^2\,, \\
\tau_{\De}\!:\De^2\lra\De^2, \quad \tau_{\De}\big(t^+,t^-\big)= \big(\ov{t^-},\ov{t^+}\big).
\end{gather*}
Thus, $\De_{\R}^2$ is the fixed locus of the anti-complex involution $\tau_{\De}$ on~$\De^2$.\\

\noindent
Let $\cC\!\equiv\!(\Si,z_1,\ldots,z_l)$ be a  marked Riemann surface with two nodes and
$\pi\!:\cU\!\lra\!\De^2$ be a holomorphic map from a complex manifold 
 with sections $s_1,\ldots,s_l\!:\De^2\!\lra\!\cU$.
We will call the tuple $(\pi,s_1,\ldots,s_l)$ a \textsf{smoothing of~$\cC$} if 
\begin{enumerate}[label=$\bullet$,leftmargin=*]
\item $\Si_{\bt}\!\equiv\!\pi^{-1}(\bt)$ is a smooth compact Riemann surface for all 
$\bt\!\in\!\De^{*2}$;
\item $s_i(\bt)\!\neq\!s_j(\bt)$ for all $\bt\!\in\!\De^2$ and $i\!\neq\!j$;
\item $(\Si_0,s_1(0),\ldots,s_l(0))\!=\!\cC$.
\end{enumerate}
Suppose $\cC\!\equiv\!(\Si,(z_1^+,z_1^-),\ldots,(z_l^+,z_l^-))$ is 
a marked symmetric Riemann surface with involution~$\si$ and a pair of conjugate nodes,
$(\pi,s_1,\ldots,s_l)$ is as above,  and
$\wt\tau_{\De}\!:\cU\!\lra\!\cU$ is an anti-holomorphic involution lifting the involution~$\tau_{\De}$.
We will call the tuple $(\pi,\wt\tau_{\De},s_1,\ldots,s_l)$ a \textsf{smoothing of~$\cC$} if 
$(\pi,s_1,\wt\tau_{\De}\!\circ\!s_1,\ldots,s_l,\wt\tau_{\De}\!\circ\!s_l)$ is a smoothing of~$\cC$
and $\wt\tau_{\De}|_{\Si_0}\!=\!\si$.
In such a case, let $\si_{\bt}\!=\!\wt\tau_{\De}|_{\Si_{\bt}}$ for each $\bt\!\in\!\De^2_{\R}$.\\

\noindent
With $(\pi,\wt\tau_{\De},s_1,\ldots,s_l)$ as above,
denote by $x_{12}^{\pm}\!\in\!\Si$  and $\wt\Si\!\lra\!\Si$
the nodes and the normalization of~$\Si$, respectively,
and set $\Si^*\!=\!\Si\!-\!\{x_{12}^{\pm}\}$.
Let
\begin{equation*}\begin{split}
\cU_0^+&\equiv
\big\{(t^+,t^-,z_1^+,z_2^+)\!\in\!\De^2\!\times\!\C^2\!:\,
|z_1^+|,|z_2^+|\!<\!1,~z_1^+z_2^+\!=\!t^+\big\},\\
\cU_0^-&\equiv
\big\{(t^+,t^-,z_1^-,z_2^-)\!\in\!\De^2\!\times\!\C^2\!:\,
|z_1^-|,|z_2^-|\!<\!1,~z_1^-z_2^-\!=\!t^-\big\}.
\end{split}\end{equation*}
As fibrations over~$\De$,
\BE{cUdfn_e}\cU\approx\big(\cU_0^+\!\sqcup\!\cU_0^-\!\sqcup\!\cU'\big)\big/\sim, \qquad 
(\bt,z_1^{\pm},z_2^{\pm})\sim\begin{cases}
(\bt,z_1^{\pm}),&\hbox{if}~|z_1^{\pm}|\!>\!|z_2^{\pm}|;\\
(\bt,z_2^{\pm}),&\hbox{if}~|z_1^{\pm}|\!<\!|z_2^{\pm}|;\end{cases}\EE
for some family~$\cU'$ of deformations of~$\Si^*$ over~$\De^2$,
a choice of coordinates $z_i^{\pm}$ on~$\wt\Si$ centered at~$x_i^{\pm}$,
and their extensions to~$\cU$.
The local coordinates $z_i^{\pm}$ and the family~$\cU'$ in~\eref{cUdfn_e} 
can be chosen so that~$\cU'$ is preserved by~$\wt\tau_{\De}$ and 
the identification in~\eref{cUdfn_e} intertwines~$\wt\tau_{\De}$ with 
the involution
\BE{cU0inv_e}\cU_0^{\pm}\lra\cU_0^{\mp}, \qquad 
\big(t^+,t^-,z_1^{\pm},z_2^{\pm}\big)\lra\big(\ov{t^-},\ov{t^+},\ov{z_1^{\pm}},\ov{z_2^{\pm}}\big).\EE
In particular, $\cU$ retracts onto~$\Si_0$ respecting the involution~$\wt\tau_{\De}$.\\

\noindent
Suppose $\pi\!:\cU\!\lra\!\De^2$ and $\wt\tau_{\De}$ are as above, $(V,\vph)\!\lra\!(\cU,\wt\tau_{\De})$
is a real bundle pair, and $\na$ and~$A$ are a connection and a 0-th order deformation
term on~$(V,\vph)$ as in Section~\ref{DetLB_subs}.
The restriction of $\na$ and~$A$ to $(V,\vph)|_{(\Si_{\bt},\si_{\bt})}$ with 
$\bt\!\in\!\De_{\R}^2$ determines a real CR-operator~$D_{\bt}$.
By \cite[Appendix~D.4]{Huang} and \cite[Section 3.2]{EES}, the determinant lines of these
operators form a line bundle
\BE{detExt_e}\det D_{(V,\vph)}\lra\De_{\R}^2\,.\EE
We denote by $\det\dbar_{\C}\!\lra\!\De_{\R}^2$ the determinant line bundle 
associated with the standard holomorphic structure on $(\cU\!\times\!\C,\wt\tau_{\De}\!\times\!\fc)$.
The proof of the next statement is essentially identical to the proof
of \cite[Corollary~6.7]{RealGWsI},
with Lemma~\ref{canonisomExt_lmm} replacing the use of \cite[Proposition~6.2]{RealGWsI}.

\begin{crl}\label{canonisomExt2_crl2a}
Let $(\pi,\wt\tau_{\De})$, $(V,\vph)$, and $(\na,A)$ be as above.
Then a real orientation on $(V,\vph)$ as in Definition~\ref{realorient_dfn4}
induces an orientation on the line bundle
\BE{canonisomExt2_e}
\rdet\,D_{(V,\vph)}\equiv
\big(\!\det D_{(V,\vph)}\big) \otimes (\det\dbar_{\C})^{\otimes n}\lra \De_{\R}^2 ,\EE
where $n\!=\!\rk_{\C}V$.
The restriction of this orientation to the fiber over each $\bt\!\in\!\De_{\R}^{*2}$
is the orientation on $\rdet\,D_{\bt}$
induced by the restriction of the real orientation to 
$(V,\vph)|_{(\Si_{\bt},\si_{\bt})}$ as in Corollary~\ref{canonisom_crl2a}.
\end{crl}

\noindent
Let $(\Si,\si)$ be a smooth symmetric surface and $(L,\wt\phi)$ be a rank~1 real bundle 
pair over~$(\Si,\si)$.
For a pair $\x\!\equiv\!(x^+,x^-)$ of conjugate points of~$(\Si,\si)$, define
\begin{gather*}
L(\x)=L\big(x^+\!+\!x^-\big), \quad L^{\otimes2}(\x)=L\!\otimes_{\C}\!L(\x),\quad
\big\{L(\x)^{\otimes2}\big\}_{\x}^0=
\big(L(\x)^{\otimes 2}|_{x^+}\!\oplus\!L(\x)^{\otimes 2}|_{x^-}\big)^{\wt\phi^{\otimes2}},\\
\big\{L(\x)^{\otimes2}\big\}_{\x}^1=\big\{L(\x)^{\otimes2}\big\}_{\x}^0\oplus
\big(L^{\otimes2}(\x)|_{x^+}\!\oplus\!L^{\otimes 2}(\x)|_{x^-}\big)^{\wt\phi^{\otimes2}}.
\end{gather*}
The projection
$$\big\{L(\x)^{\otimes2}\big\}_{\x}^1\lra 
\big\{L(\x)^{\otimes2}\big\}_{x^+}^1 \equiv L(\x)^{\otimes 2}|_{x^+}\!\oplus\!L^{\otimes2}(\x)|_{x^+}$$
is an isomorphism of real vector spaces and thus induces an orientation on its domain
from the complex orientation of its target.
This induced orientation  is invariant under the interchange of $x^+$
and~$x^-$; we will call it the \sf{canonical orientation} of $\{L(\x)^{\otimes2}\}_{\x}^1$.\\

\noindent
For a real CR-operator~$D_{L(\x)^{\otimes2}}$ on $(L(\x),\wt\phi)^{\otimes2}$, 
there is  
a short exact sequence 
$$\xymatrix{0 \ar[r]& 
\Ga\big(\Si;L^{\otimes2}(\x)\big)^{\wt\phi} \ar[r]\ar[d]^{D_{L^{\otimes2}(\x)}} &
\Ga\big(\Si;L(\x)^{\otimes2}\big)^{\wt\phi} \ar[r]\ar[d]^{D_{L(\x)^{\otimes2}}}& 
\big\{L(\x)^{\otimes2}\big\}_{\x}^0 \ar[r]\ar[d]& 0\\
0 \ar[r]& \Ga^{0,1}_{\fJ}\big(\Si;L^{\otimes2}(\x)\big)^{\wt\phi} \ar[r]&
\Ga^{0,1}_{\fJ}\big(\Si;L(\x)^{\otimes2}\big)^{\wt\phi} \ar[r]&  0 } $$
of Fredholm operators.
By~\eref{sum}, it induces a canonical isomorphism
\BE{canonisomev_e0a}
\det D_{L(\x)^{\otimes2}}\approx \big(\!\det D_{L^{\otimes2}(\x)}\big)\otimes 
\La_{\R}^2\big(\big\{L(\x)^{\otimes2}\big\}_{\x}^0\big)\,.\EE
The analogous exact sequence for an operator $D_{L^{\otimes2}(\x)}$ on
 $(L^{\otimes2}(\x),\wt\phi^{\otimes2})$ yields an isomorphism
\BE{canonisomev_e0b}\begin{split}
\det D_{L^{\otimes2}(\x)}&\approx \big(\!\det D_{L^{\otimes2}}\big)\otimes 
\La_{\R}^2\big( \big(L^{\otimes2}(\x)|_{x^+}\!\oplus\!L^{\otimes 2}(\x)
|_{x^-}\big)^{\wt\phi^2}\big).
\end{split}\EE
Combining these two isomorphisms with the identity isomorphism on $\det\dbar_{\Si;\C}$,
we obtain an isomorphism
\BE{canonisomev_e}
\rdet\,D_{L(\x)^{\otimes2}}\approx
\big(\rdet\,D_{L^{\otimes2}}\big) 
\otimes \La_{\R}^4\big(\big\{L(\x)^{\otimes2}\big\}_{\x}^1\big).\EE

\begin{crl}\label{canonisomev_crl}
With notation as above, suppose the real vector bundle $L^{\wt\phi}\!\lra\!\Si^{\si}$
is orientable.
The isomorphism~\eref{canonisomev_e} is orientation-preserving 
with respect to the orientations induced by Corollaries~\ref{canonisom_crl2a} and~\ref{canonisom_crl}
on $\rdet\,D_{L(\x)^{\otimes2}}$ and $\rdet\,D_{L^{\otimes2}}$
and the canonical orientation on $\{L(\x)^{\otimes2}\}_{\x}^1$. 
\end{crl}

\begin{proof}
Let $(\wh\Si,\wh\si)$ be the two-nodal symmetric surface consisting of $(\Si,\si)$ with
a 0-doublet $\P_+^1\!\sqcup\!\P_-^1$ attached at~$x^+$ and~$x^-$; see~\eref{SymSurfDbl_e}.
Let $\wh\x\!\equiv\!(\wh{x}^+,\wh{x}^-)$ be a pair  of conjugate points on $\wh\Si\!-\!\Si$, 
with $\wh{x}^+\!\in\!\P^1_+$, and $(\wh{L},\wh\phi)$  be the rank~1 real bundle pair
over $(\wh\Si,\wh\si)$ such~that 
$$\big(\wh{L},\wh\phi\big)\big|_{\Si}=\big(L,\wt\phi\big), \qquad
\big(\wh{L},\wh\phi\big)\big|_{\P^1_+\sqcup\P^1_-}=
\big(\cO_{\P_+^1}\!\sqcup\!\cO_{\P_-^1},\wh\si|_{\P^1_+\sqcup\P^1_-}\!\times\!\fc\big).$$
Choose a smoothing 
$$\pi\!: \cU\lra\De^2, \qquad \wt\tau_{\De}\!:\cU\lra\cU, \qquad s\!:\De^2\lra\cU$$
of $(\wh\Si,(\wh{x}^+,\wh{x}^-),\wh\si)$.
For $\bt\!\in\!\De_{\R}^{*2}$, $(\Si_{\bt},\si_{\bt})\!\approx\!(\Si,\si)$.\\

\noindent
Let $(V,\vph)$ be a real bundle pair over $(\cU,\wt\tau_{\De})$ that restricts to 
$(\wh{L},\wh\phi)$ over~$\wh\Si$ and
$$V(\s)=V\big(s\!+\wt\tau_{\De}\!\circ\!\!s\big).$$
For $\bt\!\in\!\De_{\R}^{*2}$,
the restrictions of 
the real bundle pairs $(V,\vph)$ and $(V(\s),\vph)$  to $(\Si_{\bt},\si_{\bt})$
are isomorphic to $(L,\wt\phi)$ and $(L(\x),\wt\phi)$, respectively.
The canonical real orientations on $(\wh{L},\wh\phi)^{\otimes2}$ and  
$(\wh{L}(\wh\x),\wh\phi)^{\otimes2}$ provided by Corollary~\ref{canonisom_crl} 
(like all other real orientations) extend to real orientations on
$(V,\vph)^{\otimes2}$ and $(V(\s),\vph)^{\otimes2}$, respectively.
The restrictions of the latter~to 
\begin{enumerate}[label=$\bu$,leftmargin=*]

\item $(\Si_{\bt},\si_{\bt})$ with $\bt\!\in\!\De_{\R}^{*2}$
are the canonical real orientations on $(L,\wt\phi)^{\otimes2}$ and $(L(\x),\wt\phi)^{\otimes2}$,
respectively,

\item $\Si\!\subset\!\wh\Si$ are the canonical real orientation on $(L,\wt\phi)^{\otimes2}$, 

\item $\P^1_+\!\sqcup\!\P^1_-$ are the canonical real orientations on
$$\big(\cO_{\P_+^1}\!\sqcup\!\cO_{\P_-^1},\wh\si|_{\P^1_+\sqcup\P^1_-}\!\times\!\fc\big)^{\otimes2}
\qquad\hbox{and}\qquad 
\big(\cO_{\P_+^1}(\wh{x}_+)\!\sqcup\!\cO_{\P_-^1}(\wh{x}_-),
\wh\si|_{\P^1_+\sqcup\P^1_-}\!\times\!\fc
\big)^{\otimes2},$$
respectively.\\
\end{enumerate}

\noindent
Let $D_{(V,\vph)^{\otimes2};\bt}$ be a family of real CR-operators
on $(V,\vph)^{\otimes2}$ as above Corollary~\ref{canonisomExt2_crl2a};
we can assume that it restricts to the standard $\dbar$-operator over the 0-doublet.
It induces a family $D_{(V(\s),\vph)^{\otimes2};\bt}$  of real CR-operators 
on $(V(\s),\vph)^{\otimes2}$.
Let 
$$D_{\wh{L}^{\otimes2}}=D_{(V,\vph)^{\otimes2};\0} \qquad\hbox{and}\qquad
D_{\wh{L}(\wh\x)^{\otimes2}}=D_{(V(\s),\vph)^{\otimes2};\0}$$
be the restrictions of these operators to $(\wh\Si,\wh\si)$.
Similarly to~\eref{canonisomev_e0a}, 
$$\rdet\,D_{(V(\s),\vph)^{\otimes2}}\approx
\big(\!\rdet\,D_{(V,\vph)^{\otimes2}}\big)
\otimes \La_{\R}^4\big(\big\{V(\s)^{\otimes2}\big\}_{\s}^1\big)$$
as real line bundles over $\bt\!\in\!\De_{\R}^2$. 
By the first bullet point above and Corollary~\ref{canonisomExt2_crl2a}, 
it is thus sufficient to show that the isomorphism
\BE{canonisomev_e9}
\rdet\, D_{\wh{L}(\wh\x)^{\otimes2}}
\approx
\big(\!\rdet\,D_{\wh{L}^{\otimes2}}\big) 
\otimes \La_{\R}^4\big(\big\{\wh{L}(\wh\x)^{\otimes2}\big\}_{\wh\x}^1\big)\EE
is orientation-preserving with respect to the orientations on $\rdet\, D_{\wh{L}(\wh\x)^{\otimes2}}$
and $\rdet\,D_{\wh{L}^{\otimes2}}$
induced by the canonical real orientations on $(\wh{L}(\wh\x),\wh\phi)^{\otimes2}$
and $(\wh{L},\wh\phi)^{\otimes2}$, respectively,
and the complex orientation on $\{\wh{L}(\wh\x)^{\otimes2}\}_{\wh\x}^1$.\\

\noindent
Let $(\wt{L},\wt\vph)$ be the lift of $(\wh{L},\wh\phi)$ to the normalization $(\wt\Si,\wt\si)$
of $(\wh\Si,\wh\si)$;
the latter consists of $(\Si,\si)$ and the 0-doublet $\P^1_+\!\sqcup\!\P^1_-$.
The isomorphisms~\eref{RealOrientCrl_e} induce a commutative diagram
\begin{equation*}\begin{split}
\xymatrix{\rdet\,D_{\wt{L}(\wh\x)^{\otimes2}}
\ar[r]\ar[d]& \rdet\,D_{\wh{L}(\wh\x)^{\otimes2}}
\otimes 
\La_{\R}^2L_{x^+}\!\otimes\!\La_{\R}^2\C \ar[d]\\
\big(\rdet\,D_{\wt{L}^{\otimes2}}\big)
\otimes \La_{\R}^4\big(\big\{\wt{L}(\wh\x)^{\otimes2}\big\}_{\wh\x}^1\big)
\ar[r]& \big(\rdet\,D_{\wh{L}^{\otimes2}}\big)
\otimes 
\La_{\R}^2L_{x^+}\!\otimes\!\La_{\R}^2\C
\otimes \La_{\R}^4\big(\big\{\wh{L}(\wh\x)^{\otimes2}\big\}_{\wh\x}^1\big).}
\end{split}\end{equation*}
By the second and third bullet points above and Corollary~\ref{RealOrient_crl}, 
the horizontal isomorphisms in this diagram are orientation-preserving  
with respect to the orientations on the relative determinants induced by
Corollaries~\ref{canonisom_crl2a} and~\ref{canonisom_crl}
and with respect to the complex orientations on 
the remaining lines.\\

\noindent
The left vertical isomorphism in the diagram is the tensor product of the isomorphisms
\BE{canonisomev_e15}\begin{split}
\rdet\,D_{\wt{L}\big(\wh\x)^{\otimes2}|_{\Si}}&
\approx \rdet\,D_{\wt{L}^{\otimes2}|_{\Si}}
\qquad\hbox{and}\\
\rdet\,D_{\wt{L}\big(\wh\x)^{\otimes2}|_{\P^1_+\sqcup\P^1_-}}
&\approx \big(\rdet\,D_{\wt{L}^{\otimes2}|_{\P^1_+\sqcup\P^1_-}}\big)
\otimes \La_{\R}^4\big(\big\{\wt{L}(\wh\x)^{\otimes2}\big\}_{\wh\x}^1\big).
\end{split}\EE
The first of these isomorphisms is orientation-preserving 
with respect to the canonical real orientations because 
$\wt{L}(\wh\x)|_{\Si}\!=\!\wt{L}|_{\Si}$.
Under the restrictions to~$\P^1_+$ as in~\eref{ComplexOrient_e5},
the second isomorphism in~\eref{canonisomev_e15} corresponds to an isomorphism
induced by two short exact sequences of $\C$-linear homomorphisms.
Thus, it is orientation-preserving with respect to the complex orientations on
$\rdet\,D_{\wt{L}(\wh\x)^{\otimes2}|_{\P^1_+\sqcup\P^1_-})}$
and $\rdet\,D_{\wt{L}^{\otimes2}|_{\P^1_+\sqcup\P^1_-}}$
as in Section~\ref{ComplexOrient_subs}.
By Lemma~\ref{ComplexOrient_lmm1}, these complex orientations are the same as the orientations
induced by any real orientations on the squares.
Thus,  the second isomorphism in~\eref{canonisomev_e15} 
and the left vertical isomorphism in the commutative diagram are
orientation-preserving with respect to the orientations on
the relative determinants induced by
Corollaries~\ref{canonisom_crl2a} and~\ref{canonisom_crl}.
Along with the last sentence of the previous paragraph, this implies that 
 the right vertical isomorphism  is also orientation-preserving 
with respect to these orientations.
\end{proof}

\noindent
Let $\wt\Si$ be a smooth Riemann surface and $x\!\in\!\wt\Si$.
A holomorphic vector field~$\xi$ on a neighborhood of~$x$ in~$\wt\Si$ 
with $\xi(x)\!=\!0$ determines an element 
$$\na\xi\big|_{x}\in T_x^*\wt\Si\otimes_{\C} T_x\wt\Si=\C\,.$$
Similarly, a meromorphic one-form~$\eta$ on a neighborhood of~$x$ in~$\wt\Si$ 
has a well-defined residue at~$x$, which we denote by~$\fR_x\eta$.
For a holomorphic line bundle $L\!\lra\!\wt\Si$, we denote by $\Om(L)$ the sheaf
of holomorphic sections of~$L$.\\

\noindent
If $(\Si,\si)$ is a symmetric Riemann surface with a pair of 
conjugate nodes $x_{12}^{\pm}\!\in\!\Si$ and
$x_1^{\pm},x_2^{\pm}\!\in\!\wt\Si$ are the preimages of the nodes in its normalization, let
$$T\wt\Si(-\x)=T\wt\Si\big(-x_1^+\!-\!x_1^-\!-\!x_2^+\!-\!x_2^-\big), \qquad
T^*\wt\Si(\x)=T^*\wt\Si\big(x_1^+\!+\!x_1^-\!+\!x_2^+\!+\!x_2^-\big).$$
The next statement is the analogue of \cite[Lemma~6.8]{RealGWsI} 
in the present situation.

\begin{lmm}\label{TSiext_lmm}
Suppose $(\pi\!:\cU\!\lra\!\De^2,\wt\tau_{\De})$ is a smoothing of~$(\Si,\si)$ as above. 
There exist holomorphic line bundles $\cT,\wh\cT\!\lra\!\cU$ with involutions~$\vph,\wh\vph$
lifting~$\wt\tau_{\De}$ such~that 
\begin{gather*}
(\cT,\vph)\big|_{\Si_{\bt}}=\big(T\Si_{\bt},\tnd\wt\tau_{\De}|_{T\Si_{\bt}}\big), \quad 
(\wh\cT,\wh\vph)\big|_{\Si_{\bt}}=\big(T^*\Si_{\bt},(\tnd\wt\tau_{\De}|_{T\Si_{\bt}})^*\big) 
\qquad\forall~\bt\!\in\!\De^{*2},\\
\begin{split}
\Om\big(\cT|_{\Si_0}\big)&=\big\{\xi\!\in\!
\Om\big(T\wt\Si(-\x)\big)\!:\,~
\na\xi|_{x_1^{\pm}}\!+\!\na\xi|_{x_2^{\pm}}\!=\!0\big\},\\
\Om\big(\wh\cT|_{\Si_0}\big)&=
\big\{\eta\!\in\!\Om\big(T^*\wt\Si(\x)\big)\!:\,~
\fR_{x_1^{\pm}}\eta\!+\!\fR_{x_2^{\pm}}\eta\!=\!0\big\}.
\end{split}\end{gather*}
Furthermore, $(\wh\cT,\wh\vph)\!\approx\!(\cT,\vph)^*$.
\end{lmm}

\begin{proof}
We continue with the notation as in~\eref{cUdfn_e} and~\eref{cU0inv_e}.
Denote by $T^{\vrt}\cU'\!\lra\!\cU'$ the vertical tangent bundle.
Let
\begin{alignat*}{2}
&\cT=\big(\cU_0^+\!\times\!\C \sqcup \cU_0^-\!\times\!\C \sqcup T^{\vrt}\cU'\big)\big/\!\sim, &\quad
&\wh\cT=\big(\cU_0^+\!\times\!\C \sqcup \cU_0^-\!\times\!\C \sqcup (T^{\vrt}\cU')^*\big)\big/\!\sim,\\
&(\bt,z_1^{\pm},z_2^{\pm},c)\sim\begin{cases}
c\,z_1^{\pm}\frac{\prt}{\prt z_1^{\pm}}\big|_{(\bt,z_1^{\pm})},&\hbox{if}~|z_1^{\pm}|\!>\!|z_2^{\pm}|;\\
-c\,z_2^{\pm}\frac{\prt}{\prt z_2^{\pm}}\big|_{(\bt,z_2^{\pm})},&\hbox{if}~|z_1^{\pm}|\!<\!|z_2^{\pm}|;\\
\end{cases} &\quad
&(\bt,z_1^{\pm},z_2^{\pm},c)\sim\begin{cases}
c\,\frac{\tnd_{(\bt,z_1^{\pm})}z_1^{\pm}}{z_1^{\pm}},&\hbox{if}~|z_1^{\pm}|\!>\!|z_2^{\pm}|;\\
-c\,\frac{\tnd_{(\bt,z_2^{\pm})}z_2^{\pm}}{z_2^{\pm}},&\hbox{if}~|z_1^{\pm}|\!<\!|z_2^{\pm}|.\\
\end{cases}
\end{alignat*}
Under the identifications~\eref{cUdfn_e}, the vector field and one-form on
a neighborhood of the node in~$\cU$ associated with 
$(\bt,z_1^{\pm},z_2^{\pm},c)\!\in\!\cU_0^{\pm}\!\times\!\C$
correspond to the vector field and one-form on~$\cU_0^{\pm}$
given~by
$$c\bigg(z_1^{\pm}\frac{\prt}{\prt z_1^{\pm}}-z_2^{\pm}\frac{\prt}{\prt z_2^{\pm}}\bigg)
\qquad\hbox{and}\qquad c\,\frac{\tnd z_1^{\pm}|_{\Si_{\bt}}}{z_1^{\pm}}
=-c\,\frac{\tnd z_2^{\pm}|_{\Si_{\bt}}}{z_2^{\pm}},$$
respectively (the above equality of one-forms holds for $t^{\pm}\!\neq\!0$).
Thus, $\cT$ and~$\wh\cT$ have the desired restriction properties.
The identifications in the construction of~$\cT$ and~$\wh\cT$ above intertwine
the trivial lift of~\eref{cU0inv_e} to a conjugation on $(\cU_0^+\!\sqcup\!\cU_0^-)\!\times\!\C$
with the conjugations on $T^{\vrt}\cU'$ and $(T^{\vrt}\cU')^*$ induced by~$\tnd\wt\tau_{\De}$.
Thus, they induce conjugations~$\vph$ and~$\wh\vph$ on~$\cT$ and~$\wh\cT$.
By the same reasoning as in the proof of \cite[Lemma~6.8]{RealGWsI},
$(\wh\cT,\wh\vph)$ and $(\cT,\vph)^*$ are isomorphic as real bundle pairs 
over~$(\cU,\wt\tau_{\De})$.
\end{proof}

\begin{crl}\label{realorient_crl}
Let $(\Si,\si)$, $(\pi,\wt\tau_{\De})$, and $\cT,\wh\cT\!\lra\!\cU$ be as in Lemma~\ref{TSiext_lmm}.
The orientation on the restriction of the real line bundle
\BE{crealorient_crl_e}
\rdet\,\dbar_{(\wh\cT,\wh\vph)^{\otimes2}}\equiv
\big(\!\det\dbar_{(\wh\cT,\wh\vph)^{\otimes2}}\big)\otimes 
\big(\!\det\dbar_{\C}\big)\lra   \De_{\R}^2\EE
to $\De_{\R}^{*2}$ determined by the canonical isomorphisms of
Corollaries~\ref{canonisom_crl2a} and~\ref{canonisom_crl} extends across $\bt\!=\!0$
as the orientation determined by  the canonical isomorphism of
Corollaries~\ref{canonisom_crl2a} and~\ref{canonisom_crl} for the nodal symmetric surface~$(\Si,\si)$.
\end{crl}

\begin{proof}
By Corollaries~\ref{canonisom_crl2a} and~\ref{canonisom_crl}, the restriction of the real bundle~pair
\BE{whcTcC_e} 
\big(\wh\cT^{\otimes2}\!\oplus\!2\cT,\wh\vph^{\otimes2}\!\oplus\!2\vph\big)
\lra(\cU,\wt\tau_{\De})\EE
to the central fiber $(\Si,\si)$  has a canonical real orientation. 
Since $\cU$ retracts onto~$\Si$ respecting the involution~$\wt\tau_{\De}$,
this real orientation extends to a real orientation on~\eref{whcTcC_e} which restricts 
to the canonical real orientation over each fiber~$(\Si_{\bt},\si_{\bt})$ with 
$\bt\!\in\!\De_{\R}^{*2}$.
By Corollary~\ref{canonisomExt2_crl2a}, the extended real orientation induces an orientation on
the real line bundle~\eref{crealorient_crl_e} over~$\De_{\R}^2$.
The restriction of this orientation to the fiber over each $\bt\!\in\!\De_{\R}^{*2}$
is the orientation induced by the restriction of the extended real orientation to
the fiber of~\eref{whcTcC_e} as in Corollary~\ref{canonisom_crl2a},
i.e.~the canonical orientation on each fiber of~\eref{crealorient_crl_e}.
\end{proof}

\noindent
The next two statements are the analogues of \cite[Lemmas~6.9,6.10]{RealGWsI}
for smoothings of two-nodal Riemann surfaces and hold for the same reasons.

\begin{lmm}[Dolbeault Isomorphism]\label{Cech_lmm}
Suppose $(\Si,\si)$ and $(\pi,\wt\tau_{\De})$  are as in Lemma~\ref{TSiext_lmm}
and $(L,\wt\phi)\!\lra\!(\cU,\wt\tau_{\De})$ is a holomorphic line bundle so that  
$\deg L|_{\Si}\!<\!0$ and $\deg L|_{\Si'}\!\le\!0$ for each irreducible component $\Si'\!\subset\!\Si$.
The families of vector spaces $H_{\dbar}^1(\Si_{\bt};L)$ and $\wch{H}^1(\Si_{\bt};L)$
then form vector bundles $R_{\dbar}^1\pi_*L$ and $\wch{R}^1\pi_*L$ over~$\De^2$ 
with conjugations lifting~$\tau_{\De}$ which are canonically isomorphic as real bundle pairs
over~$(\De^2,\tau_{\De})$.
\end{lmm}

\begin{lmm}[Serre Duality]\label{Serre_lmm}
Suppose $(\Si,\si)$, $(\pi,\wt\tau_{\De})$,   and $(\wh\cT,\wh\vph)$
are as in Lemma~\ref{TSiext_lmm}
and \hbox{$(L,\wt\phi)\!\lra\!(\cU,\wt\tau_{\De})$} is a holomorphic line bundle 
so that  $\deg L|_{\Si}\!>\!2g_a(\Si)\!-\!2$
and $\deg L|_{\Si'}\!\ge\!2g_a(\Si')\!-\!2$
for each irreducible component $\Si'\!\subset\!\Si$.
The family of vector spaces $H_{\dbar}^0(\Si_{\bt};L)$ 
then forms a vector bundle $R_{\dbar}^0\pi_*L$ over~$\De$  with a conjugation lifting~$\tau_{\De}$ 
and there is a canonical isomorphism
\BE{Serre_e}R_{\dbar}^1\pi_*\big(L^*\!\otimes\!\wh\cT\big)\approx\big(R_{\dbar}^0\pi_*L\big)^*\EE
of real bundle pairs over~$(\De^2,\tau_{\De})$.
\end{lmm}

\subsection{Canonical isomorphisms and canonical orientations}
\label{CanIsomCanOrien_subs}

\noindent
Let $(\Si,x_{12}^{\pm},\si)$ be a  symmetric surface with a pair of conjugate nodes.
We will next compare the orientations of isomorphisms of determinant lines
associated with~$(\Si,\si)$ which are induced via its smoothings~$(\Si_{\bt},\si_{\bt})$
as in Section~\ref{SymmSurfDegen_subs} and via its normalization~$(\wt\Si,\wt\si)$.
We continue with the notation introduced in Section~\ref{SymmSurfDegen_subs}.\\

\noindent
Let $\C_{\x},S_{\x}\!\lra\!\wt\Si$ denote the skyscraper sheaves
with fibers $\C$ at the preimages~$x_i^{\pm}$ of the nodes of~$\Si$
and fibers $T_{x_i^{\pm}}^*\wt\Si$, respectively.
The projections
\BE{H0TCisom_e15}\begin{split}
 H^0\big(\wt\Si;\C_{\x}\big)^{\wt\si}&\lra 
\C^2 = H^0\big(\wt\Si;\C_{x_1^+})\!\oplus\!H^0(\wt\Si;\C_{x_2^+}\big),\\
H^0\big(\wt\Si;S_{\x}\big)^{\wt\si}&\lra 
T_{x_1^+}^*\!\wt\Si\oplus T_{x_2^+}^*\!\wt\Si 
= H^0\big(\wt\Si;S_{x_1^+}\big)\!\oplus\!H^0\big(\wt\Si;S_{x_2^+}\big)
\end{split}\EE
to the values at $x_1^+$ and $x_2^+$ are isomorphisms.
We use the first isomorphism to orient $H^0(\wt\Si;\C_{\x})^{\wt\si}$ 
from the standard orientation on~$\C$.
We use the second isomorphism to orient $H^0(\wt\Si;S_{\x})^{\wt\si}$ 
from the orientations on $T_{x_1^+}^*\!\wt\Si$ and $T_{x_2^+}^*\!\wt\Si$
induced from the complex orientations on $T_{x_1^+}\!\wt\Si$ and $T_{x_2^+}\!\wt\Si$,
respectively, as in the proof of Lemma~\ref{ComplexOrient_lmm2}.
As indicated at the beginning of Section~\ref{ComplexOrient_subs}, 
the orientation on each $T_{x_i^+}^*\!\wt\Si$ induced from the complex orientation
of $T_{x_i^+}\!\wt\Si$  is the opposite of the complex orientation of $T_{x_i^+}^*\!\wt\Si$.
Thus, the induced orientation on  $T_{x_1^+}^*\!\wt\Si\!\oplus\!T_{x_2^+}^*\!\wt\Si$
agrees with the complex orientation.\\

\noindent
The residues of meromorphic one-forms on~$\wt\Si$ provide canonical identifications
$$T^*\wt\Si(\x)\big|_{x_i^+}\approx\C\,.$$
With the notation as in Corollary~\ref{canonisomev_crl},
we thus have
$$\big\{T^*\wt\Si(\x)^{\otimes2}\big\}_{\x}^1
\equiv\big\{T^*\wt\Si(\x)^{\otimes2}\big\}_{(x_1^+,x_1^-)}^1 \!\oplus\!
\big\{T^*\wt\Si^{\otimes2}(\x)\big\}_{(x_2^+,x_2^-)}^1
=H^0\big(\wt\Si;\C_{\x}\big)^{\wt\si}\!\oplus\! 
H^0\big(\wt\Si;S_{\x}\big)^{\wt\si}\,.$$
With $L\!=\!T^*\wt\Si$, \eref{canonisomev_e} becomes
\BE{canonisomevT_e}
\rdet\dbar_{(T^*\wt\Si(\x),(\tnd\wt\si)^*)^{\otimes2}}\approx
\big(\!\rdet \dbar_{(T^*\wt\Si,(\tnd\wt\si)^*)^{\otimes2}}\big)
\otimes \La_{\R}^2\C\!\otimes\!\La_{\R}^2\C
\otimes \La_{\R}^4\big(H^0(\wt\Si;S_{\x})^{\wt\si}\big).
\EE\\

\noindent
Let $\C_{x_{12}^{\pm}}\!\lra\!\Si$ be the skyscraper sheaf over~$x_{12}^{\pm}$.
By Lemma~\ref{TSiext_lmm}, there is an exact sequence
$$0\lra \cO(\wh\cT^{\otimes2}) \lra \cO\big(T^*\wt\Si(\x)^{\otimes2}\big)\lra  
\C_{x_{12}^+}\!\oplus\!\C_{x_{12}^-} \lra 0$$
of sheaves over~$\Si$.
Thus,
\eref{DvswtD_e2} applied with $(\wh\cT,\wh\vph)^{\otimes2}|_{\Si}$
and  $(\Si\!\times\!\C,\si\!\times\!\fc)$ determines an isomorphism
\BE{SqIsomExt_e0} 
\big(\!\rdet\,\dbar_{(\wh\cT,\wh\vph)^{\otimes2}|_{\Si}}\big)
\otimes  \La_{\R}^2\C\!\otimes\!\La_{\R}^2\C
\approx \rdet\,\dbar_{(T^*\wt\Si(\x),(\tnd\wt\si)^*)^{\otimes2}}\,.\EE
Combining this isomorphism with~\eref{canonisomevT_e}, we obtain an isomorphism
\BE{SqIsomExt_e}
\big(\!\rdet\,\dbar_{(\wh\cT,\wh\vph)^{\otimes2}|_{\Si}}\big)
\otimes  \La_{\R}^2\C\!\otimes\!\La_{\R}^2\C
\approx
\big(\!\rdet\,\dbar_{(T^*\wt\Si,(\tnd\wt\si)^*)^{\otimes2}}\big)
\otimes \La_{\R}^2\C\!\otimes\!\La_{\R}^2\C
\otimes \La_{\R}^4\big(H^0(\wt\Si;S_{\x})^{\wt\si}\big).
\EE

\begin{crl}\label{SqIsomExt_crl}
Let $(\Si,\si)$, $(\wt\Si,\wt\si)$,
$(\pi,\wt\tau_{\De})$, and $\cT,\wh\cT\!\lra\!\cU$ be as in Lemma~\ref{TSiext_lmm}.
The isomorphism~\eref{SqIsomExt_e} is orientation-preserving with respect~to 
\begin{enumerate}[label=$\bu$,leftmargin=*]

\item the canonical orientation of Corollary~\ref{realorient_crl} on 
$\rdet\,\dbar_{(\wh\cT,\wh\vph)^{\otimes2}|_{\Si}}$,

\item the canonical orientation of Corollaries~\ref{canonisom_crl2a} and~\ref{canonisom_crl} 
on $\rdet\,\dbar_{(T^*\wt\Si,(\tnd\wt\si)^*)^{\otimes2}}$,

\item the orientation on $H^0(\wt\Si;S_{\x})^{\wt\si}$ described above and
the complex orientation on~$\C$.
\end{enumerate}
\end{crl}

\begin{proof}
The canonical orientation of Corollary~\ref{realorient_crl} on 
$\rdet\,\dbar_{(\wh\cT,\wh\vph)^{\otimes2}|_{\Si}}$ is the orientation induced by
the canonical real orientation on the restriction of $(\wh\cT,\wh\vph)$ to~$\Si$.
The latter lifts to the canonical real orientation on the real bundle pair
\BE{SqIsomExtCrl_e1}
\big(T^*\wt\Si(\x),(\tnd\wt\si)^*\big)^{\otimes2} 
\lra (\wt\Si,\wt\si).\EE
By Corollary~\ref{RealOrient_crl}, the isomorphism~\eref{SqIsomExt_e0} is thus
orientation-preserving with respect to the orientation in the first bullet item above,
the complex orientation on~$\C$, and 
the orientation on  $\rdet\,\dbar_{(T^*\wt\Si(\x),(\tnd\wt\si)^*)^{\otimes2}}$ induced by 
the canonical real orientation on~\eref{SqIsomExtCrl_e1}.
By Corollary~\ref{canonisomev_crl},  the isomorphism~\eref{canonisomevT_e}
is orientation-preserving with respect to the latter
and the orientations in the second and third bullet items above.
The last two statements together imply the claim.
\end{proof}

\noindent
Let $(\pi,\wt\tau_{\De},s_1,\ldots,s_l)$ be a smoothing of a  marked symmetric Riemann surface 
\BE{cCdfn_e}\cC \equiv \big(\Si,(z_1^+,z_1^-),\ldots,(z_l^+,z_l^-)\big)\EE
with a pair of conjugate nodes,
$\cT,\wh\cT\!\lra\!\cU$ be the holomorphic line bundles with involutions~$\vph,\wh\vph$
as in Lemma~\ref{TSiext_lmm},
and
$$\cT\cC=
\cT\big(\!-\!s_1\!-\!\wt\tau_{\De}\!\circ\!s_1\!-\!\ldots\!-\!s_l\!-\!\wt\tau_{\De}\!\circ\!s_l\big), 
\quad
\wh\cT\cC=\wh\cT\big(s_1\!+\!\wt\tau_{\De}\!\circ\!s_1\!+\!\ldots\!+\!s_l\!+\!\wt\tau_{\De}\!\circ\!s_l\big).$$
By the last statement of Lemma~\ref{TSiext_lmm}, $\cT\cC^*\!=\!\wh\cT\cC$. 
Let
\begin{alignat}{1}
\label{wtcCdfn_e}
\wt\cC&=\big(\wt\Si,(z_1^+,z_1^-),\ldots,(z_l^+,z_l^-),(x_1^+,x_1^-),(x_2^+,x_2^-)\big),\\
\notag
T\wt\cC&=T\wt\Si\big(\!-\!z_1^+\!-\!z_1^-\!-\!\ldots\!-\!z_l^+\!-\!z_l^-\!-\!x_1^+\!-\!x_1^-
-\!x_2^+\!-\!x_2^-\big),\\
\notag
T^*\wt\cC&=T^*\wt\Si\big(z_1^+\!+\!z_1^-\!+\!\ldots\!+\!z_l^+\!+\!z_l^-\!+\!x_1^+\!+\!x_1^-
+\!x_2^+\!+\!x_2^-\big).
\end{alignat}
Let $S\cC\!\lra\!\Si$ and $S\wt\cC\!\lra\!\wt\Si$ be the skyscraper sheaves of
the cotangent bundles at the marked points as in the proof of Lemma~\ref{ComplexOrient_lmm2}.
We also denote by $S\cC\!\subset\!S\wt\cC$ the lift of~$S\cC$ to~$\wt\Si$,
i.e.~the natural complement of the subsheaf~$S_{\x}$ of~$S\wt\cC$.\\

\noindent
By Lemma~\ref{TSiext_lmm}, taking the (second order)
residues of sections of  $\wh\cT\cC\!\otimes\!\wh\cT$
at $x_1^+\!\in\!\wt\Si$ induces an isomorphism
\BE{H0TC_e}
\det\dbar_{(\wh\cT\cC\otimes\wh\cT,\wh\vph^{\otimes2})}|_{\Si} 
\approx \det\dbar_{(T^*\wt\cC\otimes T^*\wt\Si,(\tnd\wt\si)^{*\otimes2})}
\otimes \La^2_{\R}\C;\EE
it corresponds to the isomorphism~\eref{sum} for the short exact sequence of
Fredholm operators represented by the middle column in Figure~\ref{CompOrientDM_fig2}.
Combining~\eref{H0TC_e} with the isomorphism~\eref{DvswtD_e2}
for the trivial rank~1 real bundle pair~$(V,\vph)$, we obtain an isomorphism
\BE{H0TCisom_e}
\big(\!\rdet\,\dbar_{(\wh\cT\cC\otimes\wh\cT,\wh\vph^{\otimes2})|_{\Si}}\big)
\otimes \La^2_{\R}\C \approx 
\big(\!\rdet\,\dbar_{(T^*\wt\cC\otimes T^*\wt\Si,(\tnd\wt\si^*)^{\otimes2})}\big)
\otimes \La^2_{\R}\C.\EE
The exact sequence represented by the middle row in Figure~\ref{CompOrientDM_fig2} and its analogue
for~$\wt\cC$ determine isomorphisms
\BE{H0SCevisom_e}\begin{split}
\det\dbar_{(\wh\cT\cC\otimes\wh\cT,\wh\vph^{\otimes2})}|_{\Si}
&\approx \big(\!\det\dbar_{(\wh\cT,\wh\vph)^{\otimes2}|_{\Si}}\big)
\otimes \La_{\R}^{\top}\big(H^0(\Si;S\cC)^{\si}\big),\\
\det\dbar_{(T^*\wt\cC\otimes T^*\wt\Si,(\tnd\wt\si^*)^{\otimes2})} 
&\approx \big(\!\det\dbar_{(T^*\wt\Si,(\tnd\wt\si^*))^{\otimes2}}\big) 
\otimes \La_{\R}^{\top}\big(H^0(\wt\Si;S\wt\cC)^{\wt\si}\big).
\end{split}\EE
The isomorphisms~\eref{H0SCevisom_e} induce orientations on the first factors on 
the two sides of~\eref{H0TCisom_e} from
\begin{enumerate}[label=(\arabic*),leftmargin=*]

\item\label{H0SC_it} the orientations of $H^0(\Si;S\cC)^{\si}$ and $H^0(\wt\Si;S\wt\cC)^{\wt\si}$
described in the proof of Lemma~\ref{ComplexOrient_lmm2}, and

\item\label{FMisom_it} the canonical orientations on 
\BE{FMisom_it_e}\begin{split}
\rdet\,\dbar_{(\wh\cT,\wh\vph)^{\otimes2}|_{\Si}}&\equiv
\big(\!\det\dbar_{(\wh\cT,\wh\vph)^{\otimes2}|_{\Si}}\big)
\otimes \big(\!\det\dbar_{\C}|_{\Si}\big)
\qquad\hbox{and}\\
\rdet\,\dbar_{(T^*\wt\Si,(\tnd\wt\si)^*)^{\otimes2}}&\equiv
\big(\!\det\dbar_{(T^*\wt\Si,(\tnd\wt\si)^*)^{\otimes2}}\big)
\otimes \big(\!\det\dbar_{\C}|_{\wt\Si}\big)
\end{split}\EE
provided by Corollaries~\ref{canonisom_crl2a} and~\ref{canonisom_crl}.

\end{enumerate}

\begin{figure}
$$\xymatrix{& 0\ar[d]  & 0\ar[d]  & 0\ar[d]&\\
0 \ar[r]&  \Ga\big(\wt\Si;T^*\wt\Si(\x)\!\otimes\!T^*\wt\Si\big)^{\wt\si} \ar[r]\ar[d]&  
\Ga\big(\wt\Si;T^*\wt\cC\!\otimes\!T^*\wt\Si\big)^{\wt\si} \ar[r]\ar[d]& 
H^0(\wt\Si;S\cC)^{\wt\si}\ar[r]\ar[d]^{\id}& 0\\
0\ar[r]&   \Ga\big(\Si;\wh\cT^{\otimes2}\big)^{\si} \ar[r]\ar[d]^{\ev_{x_{12}^+}}& 
\Ga(\Si;\wh\cT\cC\!\otimes\!\wh\cT)^{\si} \ar[r]\ar[d]^{\ev_{x_{12}^+}} &  
H^0(\Si;S\cC)^{\si} \ar[r]\ar[d]& 0\\
0 \ar[r]&   \C\ar[r]\ar[d]&  \C\ar[r]\ar[d]& 0\\
& 0  & 0  & &}$$ 
\caption{Commutative diagram for the proof of Corollary~\ref{H0TCisom_crl}} 
\label{CompOrientDM_fig2}
\end{figure}

\begin{crl}\label{H0TCisom_crl}
The isomorphism~\eref{H0TCisom_e} is orientation-preserving with respect to 
the two orientations described above and the complex orientation on~$\C$.
\end{crl}

\begin{proof}
The exact sequence represented by the first row in Figure~\ref{CompOrientDM_fig2}
and the $l\!=\!0$ case of the second isomorphism in~\eref{H0SCevisom_e}
determine isomorphisms
\BE{H0TCisom_e3}\begin{split}
\det\dbar_{(T^*\wt\cC\otimes T^*\wt\Si,(\tnd\wt\si^*)^{\otimes2})} 
&\approx \big(\!\det\dbar_{(T^*\wt\Si(\x)\otimes T^*\wt\Si,(\tnd\wt\si^*)^{\otimes2})}\big) 
\otimes \La_{\R}^{\top}\big(H^0(\wt\Si;S\cC)^{\wt\si}\big),\\
\det\dbar_{(T^*\wt\Si(\x)\otimes T^*\wt\Si,(\tnd\wt\si^*)^{\otimes2})} 
&\approx \big(\!\det\dbar_{(T^*\wt\Si,\tnd\wt\si^*)^{\otimes2}}\big) 
\otimes \La_{\R}^4\big(H^0(\wt\Si;S_{\x})^{\wt\si}\big).
\end{split}\EE
The second isomorphism in~\eref{H0SCevisom_e} is the composition of the first isomorphism
in~\eref{H0TCisom_e3} and the second one tensored with the identity
on $\La_{\R}^{\top}(H^0(\wt\Si;S\cC)^{\wt\si})$.\\

\noindent
Combining the analogue of~\eref{H0TC_e} for $l\!=\!0$ 
(i.e.~the isomorphism induced by the left column  in Figure~\ref{CompOrientDM_fig2})
with the isomorphism~\eref{DvswtD_e2}
for the trivial rank~1 real bundle pair~$(V,\vph)$, we obtain an isomorphism
\BE{H0TCisom_e5}
\big(\!\rdet\,\dbar_{(\wh\cT,\wh\vph)^{\otimes2}|_{\Si}}\big) \otimes \La^2_{\R}\C 
\approx 
\big(\!\rdet\,\dbar_{(T^*\wt\Si(\x)\otimes T^*\wt\Si,(\tnd\wt\si^*)^{\otimes2})}\big)
\otimes \La^2_{\R}\C.\EE
The canonical orientation on the second line in~\eref{FMisom_it_e} and 
the second isomorphism in~\eref{H0TCisom_e3} induce an orientation on 
the first factor on the right-hand side of~\eref{H0TCisom_e5}.
By the commutativity of the squares in Figure~\ref{CompOrientDM_fig2},
it is sufficient to show that the isomorphism~\eref{H0TCisom_e5} is orientation-preserving
with respect to the canonical orientation on  
$\rdet\,\dbar_{(\wh\cT,\wh\vph)^{\otimes2}|_{\Si}}$, 
the above orientation on $\rdet\,\dbar_{(T^*\wt\Si(\x)\otimes T^*\wt\Si,(\tnd\wt\si^*)^{\otimes2})}$,
 and the complex orientation on~$\C$.\\

\noindent
The composition of the isomorphism~\eref{H0TCisom_e5} tensored with the identity on~$\La^2_{\R}\C$
and the second isomorphism in~\eref{H0TCisom_e3} tensored with the identities on
$\det\dbar_{\wt\Si;\C}$ and two copies of~$\La^2_{\R}\C$ 
is homotopic to the isomorphism~\eref{SqIsomExt_e}.
By the previous paragraph, the claim is thus equivalent 
to the isomorphism~\eref{SqIsomExt_e} being orientation-preserving 
with respect to the canonical orientations on  the first factors
on the two sides, the complex orientation on~$\C$,
and the orientation on $\La_{\R}^4(H^0(\wt\Si;S_{\x})^{\wt\si})$ induced
as in the paragraph containing~\eref{H0TCisom_e15}.
This is indeed the case by Corollary~\ref{SqIsomExt_crl}.
\end{proof}

\noindent
The next two statements are obtained from Lemmas~\ref{Cech_lmm} and~\ref{Serre_lmm}
in the same way as \cite[Corollary~6.12]{RealGWsI} is obtained from \cite[Lemmas~6.9,6.10]{RealGWsI}.

\begin{crl}\label{famDI_crl}
If the marked curve~\eref{cCdfn_e} is stable, the orientation on the restriction of the real line bundle
$$\La_{\R}^{\top}\big((\wch{R}^1\pi_*\cT\cC)^{\si}\big)\otimes
\La_{\R}^{\top}\big(\big(R_{\dbar}^1\pi_*\cT\cC\big)^{\si}\big)
 \lra \De_{\R}^2$$
to $\De_{\R}^{*2}$ induced by Dolbeault Isomorphism   extends across $\bt\!=\!0$
as the orientation induced by  Dolbeault Isomorphism
for the nodal symmetric surface~$(\Si,\si)$.
\end{crl}

\begin{crl}\label{famSD_crl}
If the marked curve~\eref{cCdfn_e} is stable, the orientation on the restriction of the real line bundle
$$\La_{\R}^{\top}\big(\big(R_{\dbar}^1\pi_*\cT\cC\big)^{\si}\big)
\otimes \La_{\R}^{\top}\big(\big((R_{\dbar}^0\pi_*(\wh\cT\cC\!\otimes\!\wh\cT))^{\si}\big)^*\big)
\lra \De_{\R}^2$$
to $\De_{\R}^{*2}$ induced by Serre Duality as in the proof of \cite[Proposition~5.9]{RealGWsI} 
extends across $\bt\!=\!0$
as the orientation induced by Serre Duality
for the nodal symmetric surface~$(\Si,\si)$.
\end{crl}

\noindent
We continue with the setup for~\eref{wtcCdfn_e}.
By Lemma~\ref{TSiext_lmm}, there is an exact sequence
\BE{KSext_e5a0} 0\lra \cO(\cT\cC|_{\Si}) \lra \cO\big(T\wt\cC\big)\lra  
\C_{x_{12}^+}\!\oplus\!\C_{x_{12}^-} \lra 0\EE
of sheaves with lifts of the involution~$\si$ over~$\Si$.
The projection~of
\BE{H0proj_e} \wch{H}^0\big(\Si;\C_{x_{12}^+}\!\oplus\!\C_{x_{12}^-}\big)^{\si}\subset 
\wch{H}^0\big(\Si;\C_{x_{12}^+}\big)\!\oplus\!\wch{H}^0\big(\Si;\C_{x_{12}^-}\big)
=\C\oplus\C\EE
to the first component induces an isomorphism of real vector spaces.\\

\noindent
If $\cC$ is stable, the real part of the cohomology sequence 
induced by~\eref{KSext_e5a0}, its analogue in Dolbeault cohomology,
and  Dolbeault Isomorphism induce a commutative diagram
\BE{KSext_e5a}\begin{split}
\xymatrix{ 0\ar[r]& \C \ar[r]\ar[d]^{\id}& 
\wch{H}^1\big(\Si;\cO(\cT\cC|_{\Si})\big)^{\si} \ar[r]\ar[d]^{\tn{DI}}&
\wch{H}^1\big(\wt\Si;\cO\big(T\wt\cC\big)\big)^{\si}  \ar[r]\ar[d]^{\tn{DI}}& 0\\
0\ar[r]& \C \ar[r]& H^1\big(\Si;\cT\cC\big)^{\si} \ar[r]& 
H^1\big(\wt\Si;T\wt\cC\big)^{\si} \ar[r]& 0}
\end{split}\EE
of exact sequences.
In particular, there are canonical isomorphisms
\BE{KSext_e5b}\begin{split}
\La_{\R}^{\top}\big(\wch{H}^1\big(\Si;\cO(\cT\cC|_{\Si})\big)^{\si}\big)
&\approx \La_{\R}^{\top}\big(\wch{H}^1\big(\wt\Si;\cO\big(T\wt\cC\big)\big)^{\si}\big)
\otimes\La_{\R}^{\top}\C\,,\\
\La_{\R}^{\top}\big(H^1\big(\Si;\cT\cC\big)^{\si}\big)
&\approx \La_{\R}^{\top}\big(H^1\big(\wt\Si;T\wt\cC\big)^{\si}\big)
\otimes\La_{\R}^{\top}\C\,.
\end{split}\EE
Combining them together, we obtain an isomorphism
\BE{DIisom_e}\begin{split}
&\La_{\R}^{\top}\big(\wch{H}^1\big(\Si;\cO(\cT\cC|_{\Si})\big)^{\si}\big)
\!\otimes\!
\La_{\R}^{\top}\big(H^1(\Si;\cT\cC)^{\si}\big)\\
&\hspace{.5in}\approx 
\Big(\La_{\R}^{\top}\big(\wch{H}^1\big(\wt\Si;\cO(T\wt\cC\big)\big)^{\si}\big)
\!\otimes\!\La_{\R}^{\top}\big(H^1(\wt\Si;T\wt\cC)^{\si}\big)\Big)
\otimes \La_{\R}^2\C\!\otimes\!\La_{\R}^2\C.
\end{split}\EE

\begin{crl}\label{DIisom_crl}
The isomorphism~\eref{DIisom_e} is orientation-preserving with respect to
the canonical orientation of Corollary~\ref{famDI_crl} on the left-hand side,
the orientation on the first tensor product on the right-hand side induced 
by Dolbeault Isomorphism, and the canonical orientation on the last tensor product.
\end{crl}

\begin{proof}
By the commutativity of the diagram~\eref{KSext_e5a},
the isomorphism \eref{DIisom_e} is  orientation-preserving with respect to the orientations 
on the left-hand side and on the first tensor product on the right-hand side
induced by Dolbeault Isomorphism.
The former is the orientation of Corollary~\ref{famDI_crl}.
\end{proof}

\noindent
Combining the dual of~\eref{H0TC_e} with the second isomorphism in~\eref{KSext_e5b},
we obtain an isomorphism
\BE{SDisom_e}
\begin{split}
&\La_{\R}^{\top}\big(H^1(\Si;\cT\cC)^{\si}\big)
\!\otimes\!
\La_{\R}^{\top}\big(\big(H^0(\Si;\wh\cT\cC\!\otimes\!\wh\cT)^{\si}\big)^*\big)\\
&\hspace{.3in}\approx 
\Big(\La_{\R}^{\top}\big(H^1(\wt\Si;T\wt\cC)^{\si}\big) \!\otimes\!
\La_{\R}^{\top}\big(\big(H^0(\wt\Si;T^*\wt\cC\!\otimes\!T^*\wt\Si)^{\si}\big)^*\big)\Big)
\otimes \La_{\R}^2\C\!\otimes\!\La_{\R}^2(\C^{\vee})\,,
\end{split}\EE
where $\C^{\vee}\!=\!\Hom_{\C}(\C,\C)$.
The complex orientation on~$\C$ induces an orientation on~$\C^{\vee}$
under the isomorphism~\eref{duaVorient_e}.
The latter is the opposite of the complex orientation of~$\C^{\vee}$.

\begin{crl}\label{SDisom_crl}
The isomorphism~\eref{SDisom_e} is orientation-preserving with respect to
the canonical orientation of Corollary~\ref{famSD_crl} on the left-hand side,
the orientation on the first tensor product on the right-hand side induced 
by Serre Duality, and the complex orientations on~$\C$ and~$\C^{\vee}$.
\end{crl}

\begin{proof}
Since the diagram
$$\xymatrix{0  \ar[r]& \C \ar[r] \ar@{}[d]|\otimes&  
H^1\big(\Si;\cT\cC\big)^{\si} \ar[r] \ar@{}[d]|\otimes& 
H^1\big(\wt\Si;T\wt\cC\big)^{\si} \ar[r] \ar@{}[d]|\otimes& 0\\
0 &\ar[l]\ar[d]  \C  &\ar[l]\ar[d]  H^0\big(\Si;\wh\cT\cC\!\otimes\!\wh\cT\big)^{\si} &\ar[l]\ar[d]
H^0(\wt\Si;T^*\wt\cC\!\otimes\!T^*\wt\Si)^{\si} &\ar[l]  0\\
& \R \ar@{=}[r]  & \R \ar@{=}[r]& \R}$$
induced by the imaginary parts of the Serre Duality pairings commutes,
the isomorphism \eref{SDisom_e} is  orientation-preserving with respect to the orientations 
on the left-hand side and on the first tensor product on the right-hand side
induced by Serre Duality
and the complex orientations on~$\C$ and~$\C^{\vee}$.
The former is the orientation of Corollary~\ref{famSD_crl}.
The first pairing in the above diagram is the real part of a $\C$-linear pairing 
and thus identifies the oriented real vector space~$\C$ in the first row with 
the complex dual~$\C^{\vee}$ of the vector space~$\C$ in the second row.
\end{proof}

\subsection{Comparison of the canonical orientations}
\label{CanOrientComp_subs}

\noindent
Before establishing Theorem~\ref{CompOrient_thm} at the end of this section,
we obtain its analogue for the real Deligne-Mumford moduli spaces;
see Proposition~\ref{CompOrient_prp} below.
This is done after comparing the behavior of the Kodaira-Spencer (\sf{KS}) map
under the smoothings and normalization of a symmetric surface~$(\Si,\si)$
with a pair of conjugate nodes; see Lemma~\ref{KSext_lmm}.\\

\noindent
Let $g\!\in\!\Z$ and $l\!\in\!\Z^{\ge0}$ be such that $g\!+\!l\!\ge\!2$.
The identification of the last two pairs of conjugate marked points induces
an immersion 
\BE{iodfn_e}\io\!: \R\ov\cM_{g-2,l+2}^{\bu} \lra \R\ov\cM_{g,l}^{\bu}\,;\EE
the image $\R\cN_{g,l}^{\bu}$ of $\R\cM_{g-2,l+2}^{\bu}$ under this immersion consists 
of symmetric surfaces with one pair of conjugate nodes.
There is a canonical isomorphism
$$\cN\io\equiv \frac{\io^*T\R\ov\cM_{g,l}^{\bu}}{T\R\ov\cM_{g-2,l+2}^{\bu}}
\approx \cL_{l+1}\!\otimes_{\C}\!\cL_{l+2}$$
of the normal bundle of~$\io$ with the tensor product of the universal tangent line bundles
for the first points in the last two conjugate pairs.
Thus, there is a canonical isomorphism 
\BE{cNsmooth_e} 
 \io^*\big(\La_{\R}^{\top}\big(T\R\ov\cM_{g,l}^{\bu}\big)\big)
\approx \La_{\R}^{\top}(T\R\ov\cM_{g-2,l+2}^{\bu})\otimes
\La_{\R}^2\big(\cL_{l+1}\!\otimes_{\C}\!\cL_{l+2}\big)\EE
of real line bundles over $\R\ov\cM_{g-2,l+2}^{\bu}$.\\

\noindent
Combining the isomorphism~\eref{cNsmooth_e} with the isomorphism~\eref{DvswtD_e2} 
for the trivial rank~1 real bundle pair~$(V,\vph)$, we obtain an isomorphism
\BE{DMtensor_e}\begin{split}
&
\big(\La_{\R}^{\top}(T\R\cM_{g-2,l+2}^{\bu})\!\otimes\!(\det\dbar_{\C})\big)\otimes
\La_{\R}^2\big(\cL_{l+1}\!\otimes_{\C}\!\cL_{l+2}\big)\\
&\hspace{2in}\approx
\io^*\big(\La_{\R}^{\top}\big(T\R\ov\cM_{g,l}^{\bu}\big)\!\otimes\!(\det\dbar_{\C})\big)
\otimes\big(\La_{\R}^2\C\big)
\end{split}\EE
of real line bundles over $\R\ov\cM_{g-2,l+2}^{\bu}$.
Since the complement of~$\R\cN_{g,l}^{\bu}$ in a small neighborhood in $\R\ov\cM_{g,l}^{\bu}$ 
is connected and consists of smooth curves, 
the canonical orientation on the real line bundle~\eref{CidentDM_e} 
provided by \cite[Proposition~5.9]{RealGWsI} extends across $\R\cN_{g,l}^{\bu}$
and thus induces an orientation on the first tensor product 
on the right-hand side of~\eref{DMtensor_e}.

\begin{prp}\label{CompOrient_prp}
Let $g\!\in\!\Z$ and $l\!\in\!\Z^{\ge0}$ be such that $g\!+\!l\!\ge\!2$.
The isomorphism~\eref{DMtensor_e} is orientation-reversing with respect 
to the orientations on the real line bundles~\eref{CidentDM_e}  provided by 
\cite[Proposition~5.9]{RealGWsI} and the complex orientations of 
$\cL_{l+1}\!\otimes_{\C}\!\cL_{l+2}$ and~$\C$.
\end{prp}

\noindent
Suppose $\cC$, $\wt\cC$, $(\pi,\wt\tau_{\De},s_1,\ldots,s_l)$, 
$(\cT,\vph)$, and $(\wh\cT,\wh\vph)$ are as in~\eref{cCdfn_e} and~\eref{wtcCdfn_e}
with $\cU|_{\De_{\R}^2}\!\lra\!\De_{\R}^2$ embedded inside of the universal curve fibration
over~$\R\ov\cM_{g,l}^{\bu}$.
Combining the first isomorphism in~\eref{KSext_e5b} and~\eref{cNsmooth_e}, we obtain an isomorphism
\BE{KSlbr_e}\begin{split}
&\La_{\R}^{\top}\big(T_{[\cC]}\R\ov\cM_{g,l}^{\bu}\big)
\!\otimes\!\La_{\R}^{\top}\big(\wch{H}^1\big(\Si;\cO(\cT\cC|_{\Si})\big)^{\si}\big)\\
&\qquad\approx  \Big(\La_{\R}^{\top}(T_{[\wt\cC]}\R\ov\cM_{g-2,l+2}^{\bu})
\!\otimes\!\La_{\R}^{\top}\big(\wch{H}^1\big(\wt\Si;\cO\big(T\wt\cC\big)\big)^{\si}\big)\Big)
\otimes
\La_{\R}^2\big(\cL_{l+1}\!\otimes_{\C}\!\cL_{l+2}\big)\!\otimes\!\La_{\R}^2\C\,.
\end{split}\EE
The KS map induces an orientation on the left-hand side 
of~\eref{KSlbr_e} whenever $\cC$ is a smooth curve.
Since the complement of~$\R\cN_{g,l}^{\bu}$ in a small neighborhood in $\R\ov\cM_{g,l}^{\bu}$ 
is connected and consists of smooth curves, this orientation extends over~$\R\cN_{g,l}^{\bu}$.

\begin{lmm}\label{KSext_lmm}
The isomorphism~\eref{KSlbr_e} is orientation-preserving with respect 
to the orientations on the left-hand side and the first tensor product on
the right-hand side determined by the KS map
and the canonical orientations of 
$\cL_{l+1}\!\otimes_{\C}\!\cL_{l+2}$ and~$\C$.
\end{lmm}

\begin{proof}
The proof is similar to that of \cite[Lemma~6.16]{RealGWsI}.
The parameter~$t^+$ in Section~\ref{SymmSurfDegen_subs}
can be viewed as an element of the complex line bundle
$\cL_{l+1}\!\otimes_{\C}\!\cL_{l+2}$ and parametrizes the smoothings of the marked symmetric
surface~$\cC$ as in~\eref{cCdfn_e}.
In the notation of Section~\ref{SymmSurfDegen_subs}, 
they are described by $\bt\!=\!(t^+,t^-)$ with $t^-\!=\!\ov{t^+}$. 
Denote by  $\cT\wt\cC\!\lra\!\wt\cU_{g-2,l+2}$
the twisted down vertical tangent bundle over the universal curve 
$\pi\!:\wt\cU_{g-2,l+2}\!\lra\!\R\cN_{g,l}^{\bu}$.\\

\noindent
As in the proof of \cite[Lemma~6.16]{RealGWsI}, the vector bundles 
$$T\R\cN_{g,l}^{\bu}, \big(\wch{R}^1\pi_*(\cT\wt\cC)\big)^{\si}\lra \R\cN_{g,l}^{\bu}$$
extend over a neighborhood of $\R\cN_{g,l}^{\bu}$ in $\R\ov\cM_{g,l}^{\bu}$
as a subbundle of $T\R\ov\cM_{g,l}^{\bu}$ and a quotient bundle of 
$(\wch{R}^1\pi_*\cT\cC)^{\si}$.
The KS map induces an isomorphism between these two extensions and gives rise
to a~diagram 
$$\xymatrix{ T_{\wt\cC}\R\cM_{g-2,l+2}^{\bu} \ar[r]\ar[d]_{\tn{KS}}^{\approx}& 
T_{\cC_{\bt}}\R\ov\cM_{g,l}^{\bu} \ar[d]_{\tn{KS}}^{\approx}\ar[r]& 
 \cL_1\!\otimes_{\C}\!\cL_2|_{\wt\cC}\ar[d]_{\tn{KS}}^{\approx}\\
\wch{H}^1\big(\wt\Si;\cO\big(T\wt\cC\big)\big)^{\si}
& \ar[l] \wch{H}^1\big(\Si_{\bt};\cO(\cT\cC|_{\Si_{\bt}})\big)^{\si_{\bt}} &\ar[l]\C}$$
commuting up to homotopy of the isomorphisms given by the vertical arrows.
The crucial point  is that the KS map sends the deformation parameter 
$t^+\!\in\!\cL_1\!\otimes_{\C}\!\cL_2$ to the $\C$-factor in~\eref{KSlbr_e}
in an orientation-preserving fashion.
This is shown in the next paragraph.\\

\noindent
Similarly to the last part of the proof of \cite[Lemma~6.16]{RealGWsI},
we cover a neighborhood of~$\Si_{\bt}$ in~$\cU$ by the open~sets
$$\cU_1^{\pm}=\big\{(t^+,t^-,z_1^{\pm},z_2^{\pm})\!\in\!\cU_0^{\pm}\!:\,2|z_2^{\pm}|\!<\!1\big\} 
\quad\hbox{and}\quad
\cU_2^{\pm}=\big\{(t^+,t^-,z_1^{\pm},z_2^{\pm})\!\in\!\cU_0^{\pm}\!:\,2|z_1^{\pm}|\!<\!1\big\},$$
along with coordinate charts each of which intersects at most one of~$\cU_1^{\pm}$ and~$\cU_2^{\pm}$.
By the same computation as before, the \v{C}ech 1-cocycle 
corresponding to the radial vector field~\cite[(6.25)]{RealGWsI} 
for the smoothing parameter~$t\!=\!t^+$ is given~by
\BE{KSext_e15}
\wh\th_{0;12}^{\pm}\equiv z_1^{\pm}\frac{\prt}{\prt z_1^{\pm}} - z_2^{\pm}\frac{\prt}{\prt z_2^{\pm}}, 
\qquad
\wh\th_{0;21}^{\pm}\equiv -z_1^{\pm}\frac{\prt}{\prt z_1^{\pm}} + z_2^{\pm}\frac{\prt}{\prt z_2^{\pm}}\EE
on $\cU_1^{\pm}\!\cap\!\cU_2^{\pm}$ after re-scaling by $|t|^{-1}$ 
and vanishes on all remaining overlaps.
In order to determine the image of the angular vector field, we replace~$t$ with $\ne^{\fI\th}t$
in the computation in the proof of \cite[Lemma~6.16]{RealGWsI}
and differentiate the resulting overlap maps~$f_{12}^{\pm}$ and~$f_{21}^{\pm}$ 
with respect to~$\th$ at $\th\!=\!0$.
Over~$\cU_0^{\pm}$, we then obtain the right-hand sides of the two expressions
in~\eref{KSext_e15} multiplied by~$\pm\fI$.
Thus, the KS map sends $t^+\!\in\!\cL_1\!\otimes_{\C}\!\cL_2$ to the $\C$-factor in~\eref{KSlbr_e}
in an orientation-preserving fashion.
\end{proof}

\begin{proof}[{\bf\emph{Proof of Proposition~\ref{CompOrient_prp}}}] 
Let  $(\wt\cC,\wt\si)$  be an element of $\R\ov\cM_{g-2,l+2}^{\bu}$.
Its image under~$\io$ is a marked symmetric curve~$(\cC,\si)$ with a pair of conjugate nodes. 
We continue with the notation and setup in the proof of Lemma~\ref{KSext_lmm}.\\

\noindent
The isomorphisms~\eref{cNsmooth_e} and~\eref{H0TC_e} induce an isomorphism
\begin{alignat}{1}
\label{CompOrientPf_e1}
&\La_{\R}^{\top}\big(T_{[\cC]}\R\ov\cM_{g,l}^{\bu}\big)\otimes
\La_{\R}^{\top}\big(\big(H^0(\Si;\wh\cT\cC\!\otimes\!\wh\cT)^{\si}\big)^*\big)\\
\notag
&\quad
\approx\Big( \La_{\R}^{\top}\big(T_{[\wt\cC]}\R\ov\cM_{g-2,l+2}^{\bu}\big)\!\otimes\!
\La_{\R}^{\top} \big(\big(H^0(\wt\Si;T^*\wt\cC\!\otimes\!T^*\wt\Si)^{\si}\big)^*\big)
\Big) \otimes \La_{\R}^2\big(\cL_{l+1}\!\otimes_{\C}\!\cL_{l+2}\big) \otimes \La_{\R}^2(\C^{\vee}).
\end{alignat}
Orientations on the left-hand side of~\eref{CompOrientPf_e1} and the first tensor product 
on the right-hand side are obtained by tensoring the orientations on 
the corresponding terms 
\begin{enumerate}[label=(\arabic*),leftmargin=*]

\item\label{KSisom_it} in~\eref{KSlbr_e} determined by the KS map,

\item\label{DIisom_it} in~\eref{DIisom_e} determined by  Dolbeault Isomorphism 
and Corollary~\ref{famDI_crl},

\item\label{SDisom_it} in~\eref{SDisom_e} determined by Serre Duality and Corollary~\ref{famSD_crl}. 

\end{enumerate}
By Lemma~\ref{KSext_lmm} and Corollaries~\ref{DIisom_crl} and~\ref{SDisom_crl},
the isomorphism~\eref{CompOrientPf_e1} is orientation-preserving with respect 
to these two orientations and the complex orientations on  
$\cL_{l+1}\!\otimes_{\C}\!\cL_{l+2}$ and~$\C^{\vee}$.\\

\noindent
The orientations on
$$\La_{\R}^{\top}\big(T_{[\cC]}\R\ov\cM_{g,l}^{\bu}\big)\otimes 
\big(\!\det\dbar_{\C}|_{[\cC]}\big)
\qquad\hbox{and}\qquad
\La_{\R}^{\top}\big(T_{[\wt\cC]}\R\ov\cM_{g-2,l+2}^{\bu}\big)\otimes \big(\!\det\dbar_{\C}|_{[\wt\cC]}\big)$$
provided by \cite[Proposition~5.17]{RealGWsI} are the tensor products of the orientations~on
\begin{enumerate}[label=(\arabic*),leftmargin=*]

\item the left-hand side of~\eref{CompOrientPf_e1} and the first tensor product 
on the right-hand side described above and

\item the first tensor products on the two sides of~\eref{H0TCisom_e} 
described below~\eref{H0SCevisom_e}.\\

\end{enumerate} 

\noindent
The isomorphism~\eref{duaVorient_e} with $V\!=\!\C$ induces a homotopy class of identifications
of $(\La^2_{\R}\C)^*\!\otimes\!\La_{\R}^2(\C^{\vee})$ with~$\R$.
By the previous paragraph and Corollary~\ref{H0TCisom_crl}, the isomorphisms 
\begin{equation*}\begin{split}
&\Big(\La_{\R}^{\top}\big(T_{[\cC]}\R\ov\cM_{g,l}^{\bu}\big)
\!\otimes\!\big(\!\det\dbar_{\cC;\C}\big)^*\Big) \otimes
\big(\La^2_{\R}\C\big)^*\\
&\quad
\approx\Big( \La_{\R}^{\top}\big(T_{[\wt\cC]}\R\ov\cM_{g-2,l+2}^{\bu}\big)\!\otimes\!
\big(\!\det\dbar_{\wt\cC;\C}\big)^*\Big) \otimes 
\La_{\R}^2\big(\cL_{l+1}\!\otimes_{\C}\!\cL_{l+2}\big)\!\otimes\! 
\big(\La^2_{\R}\C\big)^*\!\otimes\!\La_{\R}^2(\C^{\vee})
\end{split}\end{equation*}
induced by the isomorphism~\eref{cNsmooth_e}, the isomorphism~\eref{DvswtD_e2} 
for the trivial rank~1 real bundle pair~$(V,\vph)$,
and trivializations of $(\La^2_{\R}\C)^*\!\otimes\!\La_{\R}^2(\C^{\vee})$ in the above homotopy class
are orientation-preserving with respect to the orientations of Proposition~\ref{CompOrient_prp} 
and the complex orientations of~$\C$ and~$\C^{\vee}$.
Since the isomorphism~\eref{duaVorient_e} with $V\!=\!\C$ is orientation-reversing with 
respect to the complex orientations of~$\C$ and~$\C^{\vee}$, the isomorphism~\eref{DMtensor_e}
is also orientation-reversing 
with respect to the orientations of Proposition~\ref{CompOrient_prp}.
\end{proof}

\begin{proof}[{\bf\emph{Proof of Theorem~\ref{CompOrient_thm}}}] 
Throughout this argument, we omit $(X,B,J)^{\phi}$
from the notation for the moduli spaces of~maps and let
$$\cL=\cL_{l+1}\!\otimes_{\C}\!\cL_{l+2}\,.$$
By the construction of the orientations in the proofs of 
Corollary~5.10 and Theorem~1.3 in~\cite{RealGWsI}, it is sufficient to verify the claim 
over an element $[\wt{u}]\!\in\!\ov\fM_{g-2,l+2}'^{\bu}$ with a smooth stable domain.
Let $u$ be the induced real map from the corresponding nodal symmetric surface.
We denote the marked domains of~$\wt{u}$ and~$u$ by~$\wt\cC$ and~$\cC$, respectively,
and let $q\!=\!\ev_{l+1}(\wt{u})$.\\

\noindent
The forgetful morphisms~\eref{ffdfn_e} induce the short exact sequences represented by
the left and middle columns in the two diagrams of Figure~\ref{CompOrient_fig}.
The top row in the first diagram is the exact sequence on the indices of Fredholm
operators determined by the exact sequence~\eref{DvswtD_e} with $(V,\vph)\!=\!u^*(TX,\tnd\phi)$;
the middle row is the exact sequence above~\eref{SubIsom_e}.
The middle and bottom rows in the second diagram are the exact sequences associated
with the normal bundles~$\cN\io$ above~\eref{RestrOrient_e0} and~\eref{cNsmooth_e}, respectively.\\

\begin{figure}
$$\xymatrix{& 0 \ar[d] & 0 \ar[d] & 0 \ar[d] &\\
0\ar[r]&   \Ind D_u \ar[r]\ar[d]& \Ind D_{\wt{u}}\ar[r]^{\ev_{x_{l+1}^+}}\ar[d] 
& T_qX \ar[r]\ar[d]^{\id}&0\\
0 \ar[r]&  T_{\wt{u}}\ov\fM_{g-2,l+2}'^{\bu} \ar[r]\ar[d]^{\tnd\ff}& 
T_{\wt{u}}\ov\fM_{g-2,l+2}^{\bu} \ar[r]^>>>>>>{\tnd_{\wt{u}}\ev_{l+1}}\ar[d]^{\tnd\ff}&
T_qX\ar[r]\ar[d]& 0\\
0 \ar[r]&  T_{\wt\cC}\R\ov\cM_{g-2,l+2}^{\bu} \ar[r]^{\id}\ar[d]& 
T_{\wt\cC}\R\ov\cM_{g-2,l+2} \ar[r]\ar[d]& 0\\
& 0  & 0  & &\\
&&&&\\
& 0 \ar[d] & 0 \ar[d] & &\\
0\ar[r]&  \Ind D_u\ar[r]^{\id}\ar[d]& \Ind D_u\ar[r]\ar[d] & 0\ar[d]&\\
0 \ar[r]&  T_{\wt{u}}\ov\fM_{g-2,l+2}'^{\bu} \ar[r]^{\tnd\io}\ar[d]^{\tnd\ff}& 
T_{u}\ov\fM_{g,l}^{\bu} \ar[r]\ar[d]^{\tnd\ff}& 
\cL|_{\wt{u}}\ar[r]\ar[d]^{\id}& 0\\
0 \ar[r]&  T_{\wt\cC}\R\ov\cM_{g-2,l+2}^{\bu} \ar[r]^{\tnd\io}\ar[d]& 
T_{\cC}\R\ov\cM_{g,l}^{\bu} \ar[r]\ar[d]& \cL|_{\wt\cC}\ar[r]\ar[d] & 0\\
& 0  & 0  & 0 &}$$ 
\caption{Commutative diagrams for the proof of Theorem~\ref{CompOrient_thm}} 
\label{CompOrient_fig}
\end{figure}

\noindent
The middle row and column in the first diagram in Figure~\ref{CompOrient_fig} determine
isomorphisms
\BE{CompOrientPf_e3}\begin{split}
\La_{\R}^{\top}\big(T_{\wt{u}}\ov\fM_{g-2,l+2}'^{\bu}\big)
\otimes \La_{\R}^{2n}(T_qX) \otimes
\La_{\R}^2\big(\cL|_{\wt{\cC}}\big)
&\approx 
\La_{\R}^{\top}\big(T_{\wt{u}}\ov\fM_{g-2,l+2}^{\bu}\big) \otimes
\La_{\R}^2\big(\cL|_{\wt{\cC}}\big)\\
&\approx 
\big(\!\det D_{\wt{u}}\big)\otimes 
\La_{\R}^{\top}\big(T_{\wt\cC}\R\ov\cM_{g-2,l+2}^{\bu} \big)
\otimes \La_{\R}^2\big(\cL|_{\wt{\cC}}\big)\,.
\end{split}\EE
By the commutativity of the squares in this diagram, 
the composition of the two isomorphisms in~\eref{CompOrientPf_e3} 
equals to the composition of the isomorphism
\BE{CompOrientPf_e5}\begin{split}
&\La_{\R}^{\top}\big(T_{\wt{u}}\ov\fM_{g-2,l+2}'^{\bu}\big)
\otimes \La_{\R}^{2n}(T_qX) \otimes
\La_{\R}^2\big(\cL|_{\wt{\cC}}\big)\\
&\hspace{1in}\approx 
\big(\!\det D_u\big)\otimes 
\La_{\R}^{\top}\big(T_{\wt\cC}\R\ov\cM_{g-2,l+2}^{\bu} \big)
\otimes \La_{\R}^{2n}(T_qX) 
\otimes \La_{\R}^2\big(\cL|_{\wt{\cC}}\big)
\end{split}\EE
induced by the first column and the isomorphism~\eref{DvswtD_e2} 
 with $(V,\vph)\!=\!u^*(TX,\tnd\phi)$; 
the latter is induced by the first row.\\

\noindent
The middle row and column in the second diagram in Figure~\ref{CompOrient_fig} determine
isomorphisms
\BE{CompOrientPf_e7}\begin{split}
&\La_{\R}^{\top}\big(T_{\wt{u}}\ov\fM_{g-2,l+2}'^{\bu}\big)
\otimes \La_{\R}^{2n}(T_qX) \otimes
\La_{\R}^2\big(\cL|_{\wt\cC}\big)\\
&\qquad
\approx 
\La_{\R}^{\top}\big(T_u\ov\fM_{g,l}^{\bu}\big) \otimes  \La_{\R}^{2n}(T_qX)
\approx \big(\!\det D_u\big)\otimes 
\La_{\R}^{\top}\big(T_{\cC}\R\ov\cM_{g,l}^{\bu} \big)
\otimes  \La_{\R}^{2n}(T_qX)\,.
\end{split}\EE
By the commutativity of the squares in this diagram, 
the composition of the two isomorphisms in~\eref{CompOrientPf_e7} 
equals to the composition of the isomorphisms~\eref{CompOrientPf_e5}
and~\eref{cNsmooth_e};  the latter is induced by the bottom row.
Thus, the isomorphism
\BE{CompOrientPf_e9}\begin{split}
&\big(\!\det D_{\wt{u}}\big)\otimes 
\La_{\R}^{\top}\big(T_{\wt\cC}\R\ov\cM_{g-2,l+2}^{\bu} \big)
\otimes \La_{\R}^2\big(\cL|_{\wt{\cC}}\big)
\approx \big(\!\det D_u\big)\otimes 
\La_{\R}^{\top}\big(T_{\cC}\R\ov\cM_{g,l}^{\bu} \big)
\otimes  \La_{\R}^{2n}(T_qX)
\end{split}\EE
induced by~\eref{CompOrientPf_e3} and~\eref{CompOrientPf_e7} 
is the tensor product of  the isomorphism~\eref{DvswtD_e2} 
with $(V,\vph)\!=\!u^*(TX,\tnd\phi)$ and the isomorphism~\eref{cNsmooth_e}.\\

\noindent
The isomorphism~\eref{CompOrientPf_e3} induces an isomorphism
\BE{CompOrientPf_e15}\begin{split}
&\La_{\R}^{\top}\big(T_{\wt{u}}\ov\fM_{g-2,l+2}'^{\bu}\big)
\otimes \La_{\R}^{2n}(T_qX) \otimes
\La_{\R}^2\big(\cL|_{\wt{\cC}}\big)
\otimes\big(\!\det\dbar_{\wt\Si;\C}\big)^{\otimes(n+1)}\\
&\hspace{.2in}\approx 
\Big(\big(\!\det D_{\wt{u}}\big)\!\otimes\!\big(\!\det\dbar_{\wt\Si;\C}\big)^{\otimes n}\Big)
\otimes 
\Big(\La_{\R}^{\top}\big(T_{\wt\cC}\R\ov\cM_{g-2,l+2}^{\bu}\big)
\!\otimes\!\big(\!\det\dbar_{\wt\Si;\C}\big)\Big)
\!\otimes\!\La_{\R}^2\big(\cL|_{\wt{\cC}}\big)\,.
\end{split}\EE
The isomorphism~\eref{CompOrientPf_e7} and the isomorphisms~\eref{DvswtD_e2} with 
$$(V,\vph)=u^*\big(TX,\tnd\phi\big),\big(\Si\!\times\!\C,\si\!\times\!\fc\big)$$
induce an isomorphism
\BE{CompOrientPf_e17}\begin{split}
&\La_{\R}^{\top}\big(T_{\wt{u}}\ov\fM_{g-2,l+2}'^{\bu}\big)
\otimes \La_{\R}^{2n}(T_qX) \otimes
\La_{\R}^2\big(\cL|_{\wt{\cC}}\big)
\otimes\big(\!\det\dbar_{\wt\Si;\C}\big)^{\otimes(n+1)}\\
&\qquad
\approx 
\Big(\big(\!\det D_u\big)\!\otimes\!\big(\!\det\dbar_{\Si;\C}\big)^{\otimes n}\Big)
\!\otimes\!\La_{\R}^{2n}(T_qX)\!\otimes\!\La_{\R}^{2n}\C^n\\
&\hspace{2.2in}
\otimes
\Big(\La_{\R}^{\top}\big(T_{\cC}\R\ov\cM_{g,l}^{\bu}\big)\!
\otimes\!\big(\!\det\dbar_{\C;\Si}\big)\Big)\!\otimes\!\La_{\R}^2\C\,.
\end{split}\EE
A real orientation on~$(X,\om,\vph)$ induces orientations on
\BE{CompOrientPf_e25}
\rdet\,D_{\wt{u}}\equiv
 \big(\!\det D_{\wt{u}}\big)\!\otimes\!\big(\!\det\dbar_{\wt\Si;\C}\big)^{\otimes n}
\quad\hbox{and}\quad
\rdet\,D_u\equiv
 \big(\!\det D_u\big)\!\otimes\!\big(\!\det\dbar_{\Si;\C}\big)^{\otimes n}.\EE
The isomorphisms~\eref{CompOrientPf_e15} and~\eref{CompOrientPf_e17}
induce orientations on their common domain from the orientations in~\eref{CompOrientPf_e25},
the orientation of~\eref{CidentDM_e} provided by \cite[Proposition~5.9]{RealGWsI},
and the canonical orientations on~$\cL$, $TX$, and~$\C$.
The substance of Theorem~\ref{CompOrient_thm} is that the two induced orientations are
different.\\ 

\noindent
The two induced orientations are different if the composition of the inverse of 
the isomorphism in~\eref{CompOrientPf_e15} with the isomorphism in~\eref{CompOrientPf_e17} 
is orientation-reversing.
By the sentence containing~\eref{CompOrientPf_e9}, this composition is the tensor product~of 
\begin{enumerate}[label=(\arabic*),leftmargin=*]

\item the isomorphism~\eref{RealOrientCrl_e}  with $(V,\vph)\!=\!u^*(TX,\tnd\phi)$ and

\item the isomorphism~\eref{DMtensor_e}.

\end{enumerate}
By Corollaries~\ref{RealOrient_crl} and~\ref{canonisomExt2_crl2a}, 
the first isomorphism is orientation-preserving.
By Proposition~\ref{CompOrient_prp}, the second isomorphism is orientation-reversing.
\end{proof}

\begin{rmk}\label{twist_rmk2}
A real orientation on a $2n$-dimensional manifold $X$ determines orientations 
on the moduli spaces of real spaces if $n$ is odd. 
If $n$ is even,  a real orientation on~$X$ determines orientations on 
the tangent bundles of the moduli spaces of real maps twisted 
by the tangent bundles of the moduli spaces of real curves; 
the real spaces in this case are generally not orientable. 
If $n\!\in\!2\Z$ and $2g\!+\!l\!\ge\!3$, the comparison of Theorem~\ref{CompOrient_thm} 
should thus be made with the tangent bundles of the moduli spaces of maps twisted 
by the tangent bundles of the moduli spaces of curves as in~\cite[(1.3)]{RealGWsI}.
The isomorphism~\eref{CompOrient_e} is then replaced by its tensor product with
the inverse of~\eref{cNsmooth_e}.
The proof of Theorem~\ref{CompOrient_thm} implies that this isomorphism is orientation-preserving,
since the orientation-reversing isomorphism~\eref{DMtensor_e} now enters twice.
This $n\!\in\!2\Z$ analogue of Theorem~\ref{CompOrient_thm} is also 
invariant under the reordering of the nodes, since it now preserves 
the orientation of~$TX$ and $\cL_{l+1}\!\otimes_{\C}\!\cL_{l+2}$ appears~twice.
\end{rmk}

\vspace{.2in}

\noindent
{\it  Institut de Math\'ematiques de Jussieu - Paris Rive Gauche,
Universit\'e Pierre et Marie Curie, 
4~Place Jussieu,
75252 Paris Cedex 5,
France\\
penka.georgieva@imj-prg.fr}\\

\noindent
{\it Department of Mathematics, Stony Brook University, Stony Brook, NY 11794\\
azinger@math.stonybrook.edu}\\

\vspace{.2in}


\end{document}